\documentclass{amsart}
\usepackage[all,cmtip,ps]{xy}
\xyoption{rotate}
\usepackage{amsmath}
\usepackage{amsthm}
\usepackage{amssymb}
\usepackage{amscd}
\usepackage{hyperref}
\usepackage{dashrule}
\newcommand{\cA}{\mathcal{A}}

\newcommand{\cC}{\mathcal{C}}
\newcommand{\cZ}{\mathcal{Z}}
\newcommand{\cE}{\mathcal{E}}
\newcommand{\cI}{\mathcal{I}}
\newcommand{\cD}{\mathcal{D}}

\newcommand{\cG}{\mathcal{G}}

\newcommand{\cT}{\mathcal{T}}

\newcommand{\cH}{\mathcal{H}}

\newcommand{\cU}{\mathcal{U}}
\newcommand{\cV}{\mathcal{V}}
\newcommand{\cX}{\mathcal{X}}
\newcommand{\cY}{\mathcal{Y}}
\newcommand{\cS}{\mathcal{S}}
\newcommand{\bN}{\mathbb{N}}
\newcommand{\cW}{\mathcal{W}}
\newcommand{\bZ}{\mathbb{Z}}
\newcommand{\point}{{\scriptscriptstyle\bullet}}

\DeclareMathOperator{\Cone}{Cone}

\DeclareMathOperator{\Coker}{Coker}
\DeclareMathOperator{\coker}{Coker}

\DeclareMathOperator{\Hom}{Hom}
\DeclareMathOperator{\End}{End}
\DeclareMathOperator{\Ext}{Ext}
\DeclareMathOperator{\Tor}{Tor}
\DeclareMathOperator{\Ker}{Ker}
\DeclareMathOperator{\E}{E}

\DeclareMathOperator{\Ind}{Ind}

\DeclareMathOperator{\id}{id}

\DeclareMathOperator{\R}{\mathbf{R}}

\newcommand{\HoColim}{\mathrm{HoColim}}
\newcommand*{\lMod}{\textrm{\textup{-Mod}}}

\newtheorem{theorem}{Theorem}[section]
\newtheorem{lemma}[theorem]{Lemma}
\newtheorem{proposition}[theorem]{Proposition}
\newtheorem{corollary}[theorem]{Corollary}
\theoremstyle{definition}
\newtheorem{definition}[theorem]{Definition}
\newtheorem{remark}[theorem]{Remark}
\newtheorem{example}[theorem]{Example}
\newtheorem{numero}[theorem]{}

\def\dualita#1#2{\mathrel{
                 \mathop{\vcenter{
                 \offinterlineskip
                 \hbox to 1.2truecm{\rightarrowfill}
                 \hbox to 1.2truecm{\leftarrowfill}}}
                 \limits_{#2}^{#1}}}
\def\dual{\mathrel{
                 \mathop{\vcenter{
                 \offinterlineskip
                 \hbox to .6truecm{\rightarrowfill}
                 \hbox to .6truecm{\leftarrowfill}}}
                 }}
                 
\def\cDll#1{{\cD}^{\leq #1}}
\def\cDgg#1{{\cD}^{\geq #1}}

\def\cTll#1{{\cT}^{\leq #1}}
\def\cTgg#1{{\cT}^{\geq #1}}

\def\limind{\mathop{\lim\limits_{\displaystyle\rightarrow}}}
\def\ra{{\rightarrow}}
\newcommand{\hH}{\mathrm{H}}

\begin{document}
\title{A classification theorem for $t$-structures}
\author[L. Fiorot]{Luisa Fiorot}
\author[F. Mattiello]{Francesco Mattiello}
\author[A. Tonolo]{Alberto Tonolo}
\address{ Dipartimento di Matematica, Universit\`a degli studi di Padova, via Trieste 63, I-35121 Padova, Italy}
\email{fiorot@math.unipd.it, mattiell@math.unipd.it, tonolo@math.unipd.it}
\date{}
\thanks{The first author is supported by PRIN 2010-2011 Arithmetic Algebraic Geometry and Number Theory. The second and third authors are supported by Progetto di Eccellenza Fondazione Cariparo ``Algebraic structures
and their applications: Abelian and derived categories, algebraic entropy and representation of algebras''.}

\begin{abstract}
We give a classification theorem for a relevant class of  $t$-structures in triangulated categories, which includes in the case of the derived category of a Grothendieck category, the $t$-structures whose hearts have at most $n$ fixed consecutive non-zero cohomologies.
Moreover, by this classification theorem, we deduce the construction of the \emph{$t$-tree}, a new technique which generalises the filtration induced by a torsion pair.
At last we apply our results in the tilting context 
generalizing the 1-tilting equivalence proved 
by Happel, Reiten and Smal{\o} \cite{MR1327209}. The last section provides applications to classical $n$-tilting objects, examples of $t$-trees for modules over a path algebra, and new developments on compatible $t$-structures \cite{zbMATH04083879}, \cite{KeDCT}.
\end{abstract}

\maketitle
\tableofcontents

\section*{Introduction}

In \cite{MR751966}, Be{\u\i}linson, Bernstein and Deligne introduced the notion of $t$-structure in a triangulated category. A triangulated category $\cC$ can have plenty of $t$-structures, and each of these $t$-structures determines a full abelian subcategory of $\cC$: the heart of the $t$-structure. 
The theory of $t$-structures has several applications in different mathematical areas
as: algebraic analysis, algebraic geometry, motives, K-theory, representation theory, etc. 
\\
Happel, Reiten and Smal\o\ in their seminal paper \cite{MR1327209} introduced a technique to construct, starting from a given $t$-structure $\cD$ and a torsion pair on its heart, a new $t$-structure, called the \emph{tilted $t$-structure with respect to the given torsion pair}. By a result of Polishcuk \cite{MR2324559}, in such a way
one gets all the $t$-structures $\cT$ whose aisles satisfy
\[\cD^{\leq -1}\subseteq\cT^{\leq 0}\subseteq\cD^{\leq 0}.\]
Happel, Reiten and Smal\o\ proved that if the torsion pair we tilt by is \emph{tilting or cotilting}, i.e., if the torsion class is cogenerating or the torsion-free class is generating, then the heart of the new $t$-structure is derived equivalent to the heart of the old one (\cite[Chapter~I, Theorem~3.3]{MR1327209}).
A motivating example for this result is given by classical 1-tilting objects in a Grothendieck category $\cG$: the heart of the $t$-structure obtained by tilting the canonical $t$-structure in the derived category $D(\cG)$ with respect to the torsion pair generated by a classical 1-tilting object is derived equivalent to $\cG$.\\
In his report on \cite{MR1327209} for the Mathematical Reviews,  Rickard observed that ``Although the theory of tilting modules has undergone many fruitful generalizations, the original version, involving tilting modules with projective dimension one, had one aspect that did not generalize. This was the torsion theory on the module category determined by the tilting module.'' A classical tilting object $T$ with projective dimension one in the category $R\lMod$ of left modules over an arbitrary ring $R$ determines the torsion pair whose torsion class is $\{M\in R\lMod:\Ext^1_R(T,M)=0\}$ and whose torsion-free class is $\{M\in R\lMod:\Hom_R(T,M)=0\}$. Therefore every module in $R\lMod$ decomposes in pieces where at most only one among the derived functors of $\Hom_R(T,-)$ acts non trivially. It is well known that we lose this possibility when passing to classical tilting objects in $R\lMod$ with projective dimension greater than one (see \cite{zbMATH01958651}).

In this paper we want to generalise the Happel-Reiten-Smal\o\ result and, meeting Rickard's demand, to recover the torsion torsion-free decomposition, passing, referring to the motivating example,
from classical 1-tilting objects to classical n-tilting objects. 

In particular given a \emph{filterable pair} $(\cD,\cT)$ of $t$-structures \emph{of type $(n,0)$} we prove that
\begin{enumerate}
\item if $(\cD,\cT)$ is  \emph{$n$-(co)tilting} then the hearts of $\cD$ and $\cT$ are derived equivalent (the case $n=1$ recovers \cite[Chapter~I, Theorem~3.3]{MR1327209});
\item for any object in the heart of $\cD$ we construct a finite \emph{tree of short exact sequences} of height $n$ whose \emph{leaves} have at most one $\cT$-cohomology different from zero (the case $n=1$ gets back the usual short exact sequence produced by the torsion pair associated to a classical 1-tilting object \cite{zbMATH03697330}).
\end{enumerate}

The paper is divided into six sections.

In Section $1$, which is of preliminary nature, the basic concepts used later are introduced. Here we recall some definitions and results, most of them well-known, on $t$-structures in triangulated categories and on tilting objects in Grothendieck categories. We discuss briefly the connection between torsion pairs and 
$t$-structures and indicate the relationship with tilting theory (the main references are \cite{MR751966} and \cite{MR2327478}). One of the main tool that we recall is the Happel, Reiten, Smal\o\ construction (see Proposition~\ref{prop:HRS}) which, starting from the heart of a (non-degenerate) $t$-structure on a triangulated category 
$\cC$ and a torsion pair on this heart, permits to produce a new $t$-structure on $\cC$. 

In Section $2$, we introduce the notions of \emph{shift} and \emph{gap} for an ordered pair $(\cD,\cT)$ of 
$t$-structures on a triangulated category $\cC$ (see Definition~\ref{FL}). The motivating example comes from tilting theory, when $\cD$ is the natural $t$-structure on the derived category of a Grothendieck category $\cG$ and $\cT$ is the $t$-structure compactly generated by a classical $n$-tilting object in $\cG$.
Polishchuk in \cite[Lemma 1.2.2]{MR2324559} proved that the pairs of $t$-structures 
$(\cD,\cT)$ satisfying $\cD^{\leq -1} \subseteq \cT^{\leq 0}\subseteq \cD^{\leq0}$
are exactly the pairs in which $\cT$ is obtained by tilting $\cD$ with respect to a torsion pair. Generalizing the result of Polishchuk, we prove that an \emph{iterated HRS procedure} permits to recover all the \emph{right filterable pairs} $(\cD,\cT)$ (see Definition~\ref{def:filterable}) which satisfy the condition 
$\cD^{\leq -m}\subseteq \cT^{\leq 0}\subseteq \cD^{\leq 0}$
(see~Theorem~\ref{teo:genPol1} 
and Corollary~\ref{cor:genPol1}). 

In Section $3$, we introduce and investigate in detail, for a right filterable pair $(\cD,\cT)$ of $t$-structures of gap $n$,a factorization of the objects in the heart of $\cD$ in a finite binary \emph{$t$-tree} of height $n$ (see Definition~\ref{tree}), whose $2^n$ leaves are objects of $\cC$ living in (shifts of) the heart of $\cT$. This $t$-tree generalises the Brenner and Butler factorization of modules induced by the torsion pair generated by a classical 1-tilting module.

In Section $4$, we collect the same results of Sections 2 and 3 in the dual hypothesis of left filterability.

In Section $5$, we define and study the so called \emph{$n$-(co)tilting} $t$-structures: a pair $(\cD,\cT)$ of $t$-structures in a triangulated category $\cC$ is \emph{$n$-tilting} (resp. \emph{$n$-cotilting}) if
$(\cD,\cT)$ is filterable, and the full subcategory $\cH_\cD \cap \cH_\cT$ of $\cH_\cD$ cogenerates $\cH_\cD$ (resp. the full subcategory $\cH_\cD \cap \cH_\cT[-n]$ of $\cH_\cD$ generates $\cH_\cD$). The 1-(co)tilting  pairs of $t$-structures coincide with the $t$-structures induced by (co)tilting torsion pairs studied by Happel, Reiten and Smal\o\ (see~\cite[Ch.~I, \S3]{MR1327209}).
We generalise in Theorem~\ref{genheartderiv} the derived equivalence of Happel, Reiten, Smal\o~\cite[Ch.~I, Theorem~3.3]{MR1327209} to the case $n\geq 1$.

Section $6$ is devoted to some applications. One concerns derived equivalences induced by $n$-tilting objects in a Grothendieck category $\cG$. As remarked before, when $\cD$ is the natural $t$-structure on the derived category of $\cG$ and $\cT$ is the $t$-structure compactly generated by a $n$-tilting object $T$ in $\cG$, then the pair $(\cD,\cT)$ is $n$-tilting. The machinery developed in the previous sections applies, providing a commutative diagram of equivalences which clarifies the derived Morita equivalence induced by $T$ (see Lemma~\ref{lemma:senzator}).
Another application regards compatible $t$-structures, introduced by Keller and Vossieck in \cite{zbMATH04083879}. We prove in Theorem~\ref{teo:leftcomHRS} that given a left-filterable pair $(\cD,\cT)$, the $t$-structure $\cT$ is left $\cD$-compatible if and only if in its generating HRS procedure the torsion classes are all contained in $\cH_\cD$. In particular, we deduce in Corollary~\ref{Cor:tiltnoleftcom} that if $(\cD,\cT)$ is a $n$-tilting (resp. $n$-cotilting) pair of $t$-structures, then 
$\cT$ is left $\cD$-compatible if and only if $n=0$ or $n=1$, as suggested by Keller in \cite[pag. 26]{KeDCT}.
\maketitle
\tableofcontents

\section{Preliminaries}

\textbf{I. Notations.} 
Let $\cC$ be an additive category.
In what follows, any full subcategory of $\cC$ will be strictly full
(i.e., closed under isomorphisms) and additive. Any functor between
additive categories will be an additive functor. 
For any full subcategory $\cS$ of $\cC$ we denote by $^\perp\cS$ the \emph{left orthogonal subcategory} of $\cS$, that is, 
\[
^\perp\cS:=\{X\in \cC \,|\, \Hom_\cC(X,S)=0, \text{~for all~} S\in\cS\},
\]
 and by $\cS^\perp$ the \emph{right orthogonal subcategory} of $\cS$, that is, 
\[
\cS^\perp:=\{X\in \cC \,|\, \Hom_\cC(S,X)=0, \text{~for all~} S\in\cS\}.
\]
If $\cC$ is a triangulated category, we will denote its suspension functor by $[1]$.
\bigskip

\textbf{II. t-structures.} Be{\u\i}linson, Bernstein and Deligne \cite{MR751966} introduced the notion of a $t$-structure in a triangulated category in their study of perverse sheaves on an algebraic or analytic variety. 

Let $\cC$ be a triangulated category.

\begin{definition}\label{t-structure}
A \emph{$t$-structure} in $\cC$ is a pair $\cD:=(\cD^{\leq0},\cD^{\geq0})$ of full subcategories of $\cC$ such that, setting $\cD^{\leq n}:=\cD^{\leq0}[-n]$ and $\cDgg n:=\cDgg0[-n]$, one has:
\begin{enumerate}
\item[\rm (i)] $\cDll0\subseteq\cDll1$ and $\cDgg0\supseteq\cDgg1$;
\item[\rm (ii)] $\Hom_\cC(X,Y)=0$, for every $X$ in $\cDll0$ and every $Y$ in $\cDgg1$;
\item[\rm (iii)] For any object $X\in\cC$ there exists a distinguished triangle 
\[A\to X\to B \to A[1]
\]
in $\cC$, with
$A\in\cDll0$ 
and $B\in\cDgg1$.
\end{enumerate}
The classes $\cD^{\leq 0}$ and $\cD^{\geq 0}$ are called the \emph{aisle} and the \emph{co-aisle} of the $t$-structure $\cD$.
\end{definition}
The following proposition summarizes the basic properties of a $t$-structure.

\begin{proposition} \cite[Proposition~1.3.3, Theorem~1.3.6]{MR751966}
\label{t-adj}
Let $\cD=(\cDll0,\cDgg0)$ be a $t$-structure in a triangulated category $\cC$.
\begin{enumerate}
\item[\rm (i)] The inclusion of $\cDll n$ in $\cC$ admits a right adjoint $\delta^{\leq n}$, and the inclusion of $\cDgg n$ in $\cC$ a left adjoint $\delta^{\geq n}$
 called the {\rm truncation functors}.
\item[\rm (ii)] For every object $X$ in $\cC$ there exists a unique morphism 
$d\colon\delta^{\geq 1}(X)\to \delta^{\leq 0}(X)[1]$ such that the triangle
\[
\delta^{\leq 0}(X)\ra X\ra\delta^{\geq 1}(X)\overset{d}\ra \delta^{\leq 0}(X)[1]
\]
is distinguished. 
This triangle is (up to a unique isomorphism) the unique distinguished triangle 
$(A,X,B)$ with $A$ in $\cDll0$ and $B$ in $\cDgg1$ and it is called
the \emph{approximating triangle} of $X$ (for the $t$-structure $\cD$).
\item[\rm (iii)] The category $\cH_\cD:=\cDll0\cap\cDgg0$ is abelian and is called the {\rm heart} of the $t$-structure. The truncation functors induce functors
$\hH_\cD^i\colon \cC \to \cH_\cD$, $i\in\mathbb Z$, called the {\rm $t$-cohomological functors associated with the $t$-structure $\cD$}, defined as follows:
$\hH_\cD^0(X):=\delta^{\geq0}\delta^{\leq0}(X)\simeq\delta^{\leq0}\delta^{\geq0}(X)$ and for every $i\in \bZ$, 
$\hH_\cD^i(X):=\hH_\cD^0(X[i])$.
\end{enumerate}
\end{proposition}

\begin{remark}\label{opposite}
Let $\cC$ be a triangulated category endowed with a $t$-structure
$\cD=(\cDll0,\cDgg0)$. We recall that its opposite category $\cC^\circ$ is a triangulated category too
and the pair $(\cDgg0,\cDll0)$ defines a  $t$-structure $\cD^\circ$ on $\cC^\circ$.
\end{remark}

Given an abelian category $\cA$, its (unbounded) derived category $D(\cA)$ is a triangulated category 
which admits a $t$-structure, called the \emph{natural $t$-structure}, whose aisle $D(\cA)^{\leq 0}$ 
(resp.  co-aisle $D(\cA)^{\geq 0}$) is the subcategory of complexes without cohomology in positive (resp. negative) degrees. The $t$-cohomological functors associated with the natural $t$-structure are simply denoted  by $\hH^i$, $i\in\mathbb Z$.

\begin{definition}\label{BND}
A $t$-structure $(\cD^{\leq0},\cD^{\geq0})$ in $\cC$ is called \emph{non-degenerate} if $\bigcap_{n\in\bZ}\cD^{\leq n}=0$ and $\bigcap_{n\in\bZ}\cD^{\geq n}=0$.
\end{definition}

\begin{remark}\label{NonDeg}
The property of being non-degenerate implies that an object $C\in\cC$:
\begin{enumerate}
\item[-] vanishes if and only if $\hH_\cD^n(C)=0$, for each $n\in\bZ$;
\item[-] belongs to $\cDll0$ if and only if $\hH_\cD^n(C)=0$, for each $n>0$;
\item[-] belongs to $\cDgg0$ if and only if $\hH_\cD^n(C)=0$, for each $n<0$.
\end{enumerate}
\end{remark}

From now on  we will consider only {\bf non-degenerate} $t$-structures and in particular
we will extensively use the characterization of  $\cDll0$ and $\cD^{\geq 0}$ in cohomological terms.

\begin{definition}\label{def:static}
Given a $t$-structure $\cT$ in $\cC$, we call \emph{$\cT$-static of degree $d$} the objects in $\cC$ belonging to
$\cH_\cT[-d]$, and \emph{$\cT$-static} the objects in $\cC$ which are $\cT$-static of degree $d$
for some $d\in\mathbb Z$. 
\end{definition}

An object $X$ is $\cT$-static of degree $d$ if and only if $X[d]$ belongs to $\cH_\cT$, i.e., $H^j_\cT(X)=0$ for each $j\not=d$.

\begin{definition}
A \emph{torsion pair} in an abelian category $\cA$ is a pair $(\cX,\cY)$ of full subcategories of 
$\cA$ satisfying the following conditions:
\begin{enumerate}
\item[\rm (i)] $\Hom_\cA(X,Y)=0$,  for every $X\in \cX$ and every $Y\in\cY$.
\item[\rm (ii)] For any object $C\in\cA$ there exists a short exact sequence:
\[0 \to X \to C \to Y \to 0\]
in $\cA$, with $X\in\cX$ and $Y\in\cY$.
\end{enumerate}
\end{definition}

A torsion pair $(\cX,\cY)$ in the heart $\cH_\cD$ of a $t$-structure $\cD$ in a triangulated category $\cC$ induces a new $t$-structure $\cT_{(\cX,\cY)}$ in $\cC$:

\begin{proposition} \cite[Ch.~I, Proposition~2.1]{MR1327209}\cite[Proposition~2.5]{bridgeland2005t}\label{prop:HRS}
Let $\cH_\cD$ be the heart of a 
$t$-structure $\cD=(\cDll0,\cDgg0)$ on a triangulated category $\cC$ and let 
$(\cX,\cY)$ be a torsion pair on $\cH_\cD$. Then the pair $\cT_{(\cX,\cY)}:=(\cTll0_{(\cX,\cY)},\cTgg0_{(\cX,\cY)})$ of full subcategories of $\cC$
\[
\begin{matrix}
\cT^{\leq 0}_{(\cX,\cY)}= & \{ C\in \cC\; | \; H_\cD^0(C)\in\cX,\; H_\cD^i(C)=0 \;  \forall i>0 \} \hfill\\
\cT^{\geq 0}_{(\cX,\cY)}= & \{C\in \cC\; | \; H_\cD^{-1}(C)\in\cY,\; H_\cD^i(C)=0 \;  \forall i<-1 \} \\
\end{matrix}
\]
is a $t$-structure on $\cC$. We say that $\cT_{(\cX,\cY)}$ is obtained \emph{by tilting $\cD$ with respect to the torsion pair $(\cX,\cY)$}.\end{proposition}
\begin{remark}\label{rem:tauHRS}
The non degeneracy of $\cD=(\cDll0,\cDgg0)$ and the orthogonality of the classes $\cX$ and $\cY$ in $\cH_\cD$ imply the orthogonality of $\cT^{\leq 0}_{(\cX,\cY)}$ and $\cT^{\geq 1}_{(\cX,\cY)}$ in $\cC$.
Let us describe for any object $C$ in $\cC$ the approximating triangle
\[\tau^{\leq 0}C\to C\to \tau^{\geq 1}C\stackrel{+1}{\to}
\]
of $C$ for the $t$-structure $\cT_{(\cX,\cY)}$. Denote by $\delta^{\leq n}$ and $\delta^{\geq n}$ the truncation functors of the $t$-structure $\cD$ and  by $X$ the torsion part of $H^0_\cD(C)$ with respect the torsion pair $(\cX,\cY)$. From the diagram
\[\xymatrix{&X\ar@{^(->}[d]\ar@{-->}[dr]^f\\
{}\delta^{\leq 0}C\ar[r] &H^0_\cD(C)\ar[r]&(\delta^{\leq -1}C)[1]\ar[r]^-{+1}&}
\]
there exists a map $h$ such that the following diagram commutes
\[\xymatrix{X\ar[r]^-f\ar[d]&(\delta^{\leq -1}C)[1]\ar[r]\ar@{=}[d]&\Cone f\ar@{-->}[d]^{h[1]}\ar[r]^-{+1}&\\
H^0_\cD(C)\ar[r]\ar[d]&(\delta^{\leq -1}C)[1]\ar[r]\ar@{=}[d]&(\delta^{\leq 0}C)[1]\ar[d]^{\iota[1]}\ar[r]^-{+1}&\\
\delta^{\geq 0}C\ar[r]&(\delta^{\leq -1}C)[1]\ar[r]&C[1]\ar[r]^-{+1}&
}
\]
The distinguished triangle $\xymatrix{\Cone f[-1]\ar[r]^-{\iota\circ h}&C\ar[r]&\Cone(\iota\circ h)\ar[r]^-{+1}&}$ is the approximating triangle of $C$ with respect to the $t$-structure $\cT_{(\cX,\cY)}$, i.e., $\tau^{\leq 0}(C)=\Cone f[-1]$ and 
$\tau^{\geq 1}(C)=\Cone(\iota\circ h)$. Indeed from
\[0=H^{-1}_\cD((\delta^{\leq -1}C)[1])\to H^0_\cD(\Cone f[-1])\to H^0_\cD X=X\to H^{0}_\cD((\delta^{\leq -1}C)[1])=0\text{ and}\]
\[0=H^{i}_\cD((\delta^{\leq -1}C)[1])\to H^{i+1}_\cD(\Cone f[-1])\to H^{i+1}_\cD X=0\ \forall i\geq 0\]
we get $\Cone f[-1]\in\cT_{(\cX,\cY)}^{\leq 0}$. Next for each $j<0$ we have the commutative diagram
\[\xymatrix{0=H_\cD^{j-1} X\ar[r]\ar[d]& H_\cD^{j}(\delta^{\leq -1}C)\ar[r]\ar@{=}[d]& H_\cD^{j}(\Cone f[-1])\ar[r]\ar[d]& H_\cD^{j} X=0\ar[d]\\
0=H_\cD^{j-1}(\delta^{\geq 0}C)\ar[r]&H_\cD^{j}(\delta^{\leq -1}C)\ar[r]&H_\cD^{j}C\ar[r]&H_\cD^{j}(\delta^{\geq 0}C)=0}
\]
and the exact sequence
\[\xymatrix@-1.5pc{...\ar[r]&H^{-1}_\cD(\Cone f[-1])\ar[r]^-\cong&H^{-1}_\cD C\ar[r]&H^{-1}_\cD(\Cone(\iota\circ h))\ar[r]&{\phantom{AAAAAAAAAA}}\\
{}\ar[r]&H^{0}_\cD(\Cone f[-1])=X\ar@{^(->}[r]&H^{0}_\cD C\ar[r]&H^{0}_\cD(\Cone(\iota\circ h))\ar[r]&H^{1}_\cD(\Cone f[-1])=0.
}
\]
Therefore we get $H_\cD^{j}(\Cone(\iota\circ h))=0$ for each $j<0$, and that $H_\cD^{0}(\Cone(\iota\circ h))\cong H_\cD^{0}(C)/X$ belongs to $\cY$, i.e., $\Cone(\iota\circ h)\in\cT_{(\cX,\cY)}^{\geq 1}$.\\
In particular, from these diagrams it results that $H^0_\cD(\tau^{\leq 0}(C))=H^0_\cD(\Cone f[-1])$ is isomorphic to the $\cX$-torsion part of $H^0_\cD(C)$, $H^i_\cD(\tau^{\leq 0}(C))=0$ for each $i>0$ and $H^j_\cD(\tau^{\leq 0}(C))=H^j_\cD C$ for each $j<0$. Analogously, one obtains that $H^{0}_\cD(\tau^{\geq 1}(C))=H^{0}_\cD(\Cone (\iota\circ h))$ is isomorphic to the $\cY$-torsion free part of $H^0_\cD(C)$, $H^j_\cD(\tau^{\geq 1}(C))=0$ for each $j<0$ and $H^i_\cD(\tau^{\geq 1}(C))=H^i_\cD C$ for each $i>0$.
\end{remark}
We say that the $t$-structure $\cT_{(\cX,\cY)}$ is obtained by \emph{tilting $\cD$ with respect to $(\cX,\cY)$}.
The torsion class $\cX$ is the subcategory of all $\cT_{(\cX,\cY)}$-static objects in $\cH_\cD$ of degree 0; the torsion free class $\cY$ is the subcategory of all $\cT_{(\cX,\cY)}$-static objects in $\cH_\cD$ of degree 1.

\bigskip

\textbf{III. Compactly generated $t$-structures.}
Let $\cC$ be a triangulated category with small direct sums. 

An object $T \in \cC$ is called \emph{compact} if for any  family $\left\{Y_i\right\}_ {i\in I}$ of objects of $\cC$ the canonical morphism of abelian groups:
\[
\bigoplus_{i\in I}\Hom_\cC(T,Y_i) \to \Hom_\cC(T,\bigoplus_{i\in I}Y_i)
\]
is an isomorphism (see~\cite[Definition~1.6]{neeman1996grothendieck}).

\begin{proposition}\cite[Ch.~III, Theorem~2.3]{MR2327478}\label{compgen}
Let $\cC$ be a triangulated category with small direct sums and suppose that $\mathfrak{T}$ is a set of compact objects of $\cC$. Then the following pair $(\cTll0_{\mathfrak{T}}, \cTgg0_{\mathfrak{T}})$ of full subcategories determines a $t$-structure $\cT_{\mathfrak{T}}$ in $\cC$:
\[
\cTgg0_{\mathfrak{T}}=\{Y\in \cC \, |  \Hom_\cC(T,Y[n])=0, \text{~for all~} T\in\mathfrak{T} 
\text{~and~} n<0   \, \}, \quad \cTll0_{\mathfrak{T}}={}^\perp(\cTgg0_{\mathfrak{T}}).
\]
The $t$-structure $(\cTll0_{\mathfrak{T}}, \cTgg0_{\mathfrak{T}})$ is called to be \emph{compactly generated} by the set of compact objects $\mathfrak{T}$.
\end{proposition}

Let us recall the definition of homotopy colimit:

\begin{definition}\cite[\S2]{neeman1996grothendieck}
Suppose that $\cC$ has direct sums. Let
\[
X_0 \overset{f_0}\to X_1 \overset{f_1}\to X_2 \overset{f_2}\to \cdots
\]
be a sequence of objects and morphisms in $\cC$. Then the \emph{homotopy colimit} of this sequence is by definition the mapping cone of the morphism
\[
\xymatrix{
\bigoplus_{i\in\bN}X_i \ar[rr]^{\id-\bigoplus_if_i} && \bigoplus_{i\in\bN}X_i 
}
\]

Dually when $\cC$ admits direct products one gets the notion of \emph{homotopy limit} taking the homotopy colimit in $\cC^\circ$.
\end{definition}

\begin{lemma}\cite[Lemma~2.8]{neeman1996grothendieck}\label{neeman}
Let us assume that $T$ is a compact object in a triangulated category $\cC$ having  direct sums and that
\[
X_0 \to X_1 \to X_2 \to \cdots
\]
is a sequence of objects and morphisms in $\cC$.  
Then the canonical morphism of abelian groups:
\[
\varinjlim\Hom_\cC(T,X_i) \to \Hom_\cC(T,\HoColim X_i)
\]
is an isomorphism.
\end{lemma}

\begin{remark}\label{rem:tiltingcoaisle}
If $\cT$ is a compactly generated $t$-structure in $\cC$, then by Lemma~\ref{neeman} its co-aisle is closed under taking homotopy colimits in $\cC$. In particular, $\cT$ is of \emph{finite type} (see \cite[Ch.~III, Definition~1.1]{MR2327478}), i.e., its co-aisle is closed under taking small direct sums  in $\cC$.
\end{remark}

\begin{remark}\label{rem:PrelGroth}
If $\cG$ is a Grothendieck category, that is, $\cG$ is an abelian category with a generator and whose filtered direct limits are representable and exact, then $D(\cG)$ has both small direct sums and small products. In fact, direct sums are obtained by taking term-wise direct sums and products are obtained by taking term-wise products of $K$-injective replacements (recall that $\cG$ has enough injectives). Moreover, the direct sum (resp. product) of a family of distinguished triangles in $D(\cG)$ is a distinguished triangle 
(see~\cite[Proposition~1.2.1 and Remark~1.2.2]{zbMATH01573275}) and both the classes $D(\cG)^{\leq 0}$ and $D(\cG)^{\geq 0}$ are closed under direct sums in $D(\cG)$: therefore $
\cD_\cG:=(D(\cG)^{\leq 0}, D(\cG)^{\geq 0})$ is of finite type.
\end{remark}

\bigskip 

\textbf{IV. Classical $n$-tilting objects and tilting torsion pairs.} 
An object $T$ in a Grothendieck category $\cG$ is called \emph{$n$-tilting} if the following four properties are satisfied:
\begin{enumerate}
\item[(T1)] $T$ is a compact object in $D(\cG)$;
\item[(T2)] $T$ is \emph{rigid}, i.e., $\Ext^i_\cG(T,T)=\Hom_{D(\cG)}(T,T[i])=0$ for each $i>0$;
\item[(T3)] $T$ is a \emph{generator of $D(\cG)$}, i.e., given a non-zero object $X$ in $D(\cG)$ there exists $i\in\mathbb Z$ such that $\Hom_{D(\cG)}(T[i],X)\not=0$;
\item[(T4)] $\Ext^n_\cG(T,-)\not=0$ and $\Ext^{n+1}_\cG(T,-)=0$.
\end{enumerate}
In such a case the derived functor $\R\Hom_\cG(T,-):D(\cG)\to D(\End_\cG(T))$ is a triangle equivalence sending $T$ to $\End_\cG(T)$ (see \cite{zbMATH03975153}, \cite{Happel1987}, \cite{MR1002456}, \cite{MR1258406}). The aisle and the co-aisle of the $t$-structure $\cT_T$ compactly generated by $T$ are equal to
\[
\begin{array}{ll}
\cT_T^{\leq 0}    &=\{X\in D(\cG):\Hom_{D(\cG)}(T, X[i])=0 \text{ for each }i>0\}   \\
    &  =\{X\in D(\cG): H^i(\R\Hom_{\cG}(T,X))=0\; \text{ for each } i>0\}
    \end{array}\]
\[\begin{array}{ll}
   \cT_T^{\geq 0} &:=\{X\in D(\cG):\Hom_{D(\cG)}(T, X[i])=0 \text{ for each }i<0\}   \\
    &  =\{X\in D(\cG):H^i(\R\Hom_{\cG}(T,X))=0\; \text{ for each } i<0\}.
\end{array}\]
The derived functor $\R\Hom_\cG(T,-)$ sends the $t$-structure $\cT_T$ compactly generated by $T$ to the natural $t$-structure in $D(\End_\cG(T))$ compactly generated by $\End_\cG(T)$. A $\cT_T$-static object of degree $d$ in $D(\cG)$ (see Definition~\ref{def:static}) is a complex $X$ such that $H^i\R\Hom_\cG(T,X)=0$ for each $i\not=d$. In particular for a $\cT_T$-static object $M$ of degree $d$ in $\cG$ one has
$\Ext_{\cG}^i(T,M)=0$ for each $i\not=d$: the classes of $\cT_T$-static objects in $\cG$ are the classes $KE_d(T)$, $d=0,1,...,n$,  studied by Miyashita in \cite{Miya}. If $T$ is a classical 1-tilting object, then the class of $\cT_T$-static objects of degree 0 and the class of $\cT_T$-static objects of degree 1 in $\cG$ form a \emph{torsion pair} (see \cite{zbMATH03697330}, \cite{colpi1999tilting}, \cite{colby2004equivalence}); any object in $\cG$ is an extension of a $\cT_T$-static object of degree 1 by a $\cT_T$-static object of degree 0. If $T$ is a classical $n$-tilting object, $n\geq 2$, it is not anymore possible in general to decompose an object in $\cG$ in $\cT_T$-static objects (see \cite{zbMATH01958651} for examples in the case of module categories and a characterisation of modules which are extensions of $\cT_T$-static objects).

A torsion class $\cX$ in an abelian category $\cA$ is a \emph{tilting torsion class} 
(see~\cite[Ch.~I, \S3]{MR1327209}) if $\cX$ cogenerates $\cA$, i.e., for all $A$ in $\cA$ there is $X_A\in\cX$ and a monomorphism $A\hookrightarrow X_A$. The torsion class generated by a classical 1-tilting object in a Grothendieck category (see \cite[Definition~2.3]{colpi1999tilting}) is an example of a tilting torsion class. We recall the fundamental result originally due to Happel, Reiten, Smal\o, and independently improved by Bondal and Van den Bergh in \cite[Proposition~5.4.3]{BonVdBergh} and by Noohi in \cite[Theorem~7.6]{MR2486794}:

\begin{theorem}\cite[Ch.~I, Theorem~3.3]{MR1327209}\label{thm:HRSfundamental} Let $\cT:=\cT_{(\cX,\cY)}$ be the $t$-structure on $D(\cA)$ induced by a tilting torsion pair $(\cX,\cY)$ in $\cA$.
There exists a triangle equivalence $D(\cA)\to D(\cH_\cT)$ between the derived category of $\cA$ and the derived category of the heart $\cH_\cT$ of the $t$-structure $\cT$, which extends the natural inclusion $\cH_\cT\subseteq D(\cA)$.
\end{theorem}
\section{A classification theorem for $t$-structures with finite gap}\label{sect:cloth}

Throughout this section $\cC$ is a triangulated category,  and $\cD:=(\cD^{\leq0}, \cD^{\geq0})$,  $\cT:=(\cT^{\leq0}, \cT^{\geq0})$ are two $t$-structures on $\cC$ whose
truncation functors are denoted by $\delta^{\leq0}$, $\delta^{\geq0}$ and
$\tau^{\leq0}$, $\tau^{\geq0}$ respectively. 
We denote by $\cH_\cD$ and $\cH_\cT$ the hearts of 
$\cD$ and $\cT$, and by $H_\cD$ and $H_\cT$ the associated $t$-cohomological functors. 
We will also use
the notation $\cD^{[a,b]}=\cD^{\geq a}\cap \cD^{\leq b}$, where $a\leq b \in \bZ$.

\begin{definition}\label{FL}
We say that a pair of $t$-structures $(\cD,\cT)$ has \emph{shift $k\in\mathbb Z$} and \emph{gap} $n\in\Bbb N$
if $k$ is the maximal number such that $\cT^{\leq k} \subseteq \cD^{\leq 0}$ (or equivalently $\cD^{\geq 0}\subseteq \cT^{\geq k}$) and $n$ is the minimal number such that $\cD^{\leq -n}\subseteq \cT^{\leq k}$ (or equivalently $\cT^{\geq k} \subseteq \cD^{\geq -n}$). Such a $t$-structure will be called of \emph{type} $(n,k)$.
 \end{definition}
 
Intuitively in a pair of $t$-structures $(\cD,\cT)$ of type $(n,k)$, the shift $k$ permits to center the interval, while the gap $n$ gives the wideness of the interval:
\[\cD^{\leq -n}\subseteq \cT^{\leq k} \subseteq \cD^{\leq 0}\quad\text{or equivalently}\quad
\cD^{\geq 0}\subseteq \cT^{\geq k} \subseteq \cD^{\geq -n}.\]

 \begin{lemma}\label{heart-incl}
 If  $(\cD,\cT)$ is of type $(n,k)$, then
 the pair $(\cT,\cD)$ is of type $(n,-n-k)$. Moreover
 \[\cH_\cT[-k] \subseteq\cD^{[-n,0]}\text{ and }\cH_{\cD}[k] \subseteq\cT^{[0,n]}.\]
 In particular the possible non zero $\cD$-static objects in $\cH_\cT$ have degree between $-n-k$ and $-k$, while the possible non zero $\cT$-static objects in $\cH_\cD$ have degree between $k$ and $n+k$.
 \end{lemma}
 \begin{proof}
 The pair of $t$-structures $(\cD,\cT)$ is of type $(n,k)$ if $k$ is the maximal number and $n$ is the minimal number such that
 $\cT^{\leq k}\subseteq \cD^{\leq 0} \subseteq \cT^{\leq n+k}$.
 Applying the suspension functor $[k+n]$ we get 
$\cT^{\leq -n}\subseteq \cD^{\leq -k-n} \subseteq \cT^{\leq 0}$;
moreover $-k-n$ is the maximal number such that $\cD^{\leq -k-n} \subseteq \cT^{\leq 0}$ and $n$ is the minimal number such that $\cT^{\leq -n}\subseteq \cD^{\leq -k-n}$: this means that the pair $(\cT,\cD)$ has shift $-n-k$ and gap $n$. Moreover
 \[\cH_\cT[-k] =\cT^{\leq k}\cap \cT^{\geq k}\subseteq \cD^{\leq 0}\cap \cD^{\geq -n}=\cD^{[-n,0]}.\]
 The second inclusion follows analogously. Finally we get the last statement since $\cH_\cT\subseteq\cD^{[-n-k,-k]}$ and
 $\cH_\cD\subseteq\cT^{[k,n+k]}$.
 \end{proof}
 
The $t$-structure $\cT_{(\cX,\cY)}$ obtained by tilting $\cD$ with respect to a torsion pair $(\cX,\cY)$ in the heart $\cH_\cD$ of a $t$-structure $\cD$ in $\cC$ (see Proposition~\ref{prop:HRS}) satisfies
\[
\cD^{\leq -1}\subseteq \cT_{(\cX,\cY)}^{\leq 0}\subseteq \cD^{\leq 0}\quad\text{or equivalently}\quad
\cD^{\geq 0}\subseteq \cT_{(\cX,\cY)}^{\geq 0}\subseteq \cD^{\geq -1}.
\]
Polishchuk in \cite[Lemma 1.2.2]{MR2324559} proved that any pair of $t$-structures verifying the latter condition is obtained by tilting with respect to a torsion pair:
\begin{proposition}\cite[Lemma 1.2.2]{MR2324559}\label{polisch}
A pair $(\cD,\cT)$ of $t$-structures in a triangulated category $\cC$ verifies 
\[\cD^{\leq -1}\subseteq \cT^{\leq 0}\subseteq \cD^{\leq 0}\quad\text{or equivalently}\quad
\cD^{\geq 0}\subseteq \cT^{\geq 0}\subseteq \cD^{\geq -1}\]
if and only if $\cT$ is a $t$-structure obtained by tilting $\cD$ with respect to a torsion pair in 
$\cH_\cD$. In such a case the torsion pair one tilts by is 
\[(\cX,\cY):=(\cT^{\leq 0}\cap \cH_\cD, \cT^{\geq 1}\cap \cH_\cD).\]
\end{proposition}
\begin{remark}\label{rem:tipo000110}
A pair of $t$-structures $(\cD, \cT)$ satisfies
\[
\cD^{\leq -1}\subseteq \cT^{\leq 0}\subseteq \cD^{\leq 0}\quad\text{or equivalently}\quad
\cD^{\geq 0}\subseteq \cT^{\geq 0}\subseteq \cD^{\geq -1}
\]
if and only if it is of type $(0,0)$, $(0,1)$ or $(1,0)$. In the first two cases the torsion pairs we tilt by are the trivial ones: indeed 
the case of type $(0,0)$ corresponds to tilting with respect to the torsion pair $(\cH_{\cD}, 0)$, and that of type $(0,1)$ corresponds to tilting with respect to the torsion pair $(0, \cH_{\cD})$.
Moreover we note that in all the three cases the torsion class $\cH_\cD\cap \cT^{\leq 0}=\cH_\cD\cap\cH_\cT$ coincides with the class of objects in $\cH_\cD$ which are $\cT$-static of degree $0$, while the torsion-free class $\cH_\cD\cap\cT^{\geq 1}=\cH_\cD\cap\cH_\cT[-1]$ coincides with the class of objects in $\cH_\cD$ which are $\cT$-static of degree $1$ (see Definition~\ref{def:static}). Then each object in $\cH_\cD$ is an extension of $\cT$-static objects in $\cH_\cD$; we have the same phenomenon we encountered in the derived category of a Grothendieck category $\cG$ when $\cD$ is the natural $t$-structure and $\cT$ is the $t$-structure $\cT_T$ generated by a classical 1-tilting object $T$ in $\cG$ (see Section IV in Preliminaries).
\end{remark}

Our aim is to generalise the Polishchuk result describing for all $m\in\mathbb N$ all the pairs of $t$-structures $(\cD,\cT)$ satisfying
\[
\cD^{\leq -m}\subseteq \cT^{\leq 0}\subseteq \cD^{\leq 0}\quad\text{or equivalently}\quad
\cD^{\geq 0}\subseteq \cT^{\geq 0}\subseteq \cD^{\geq -m}.
\]
\begin{remark}\label{rem:tipo} Observe that a pair $(\cD,\cT)$ satisfies $\cD^{\geq 0}\subseteq \cT^{\geq 0}\subseteq \cD^{\geq -m}$ if and only if it is of type $(n,k)\in\mathbb N\times \mathbb N$ with $n+k\leq m$: indeed, the numbers $k$ and $n$ are the maximal and minimal respectively such that
\[
\cD^{\geq 0}\subseteq \cD^{\geq -k}\subseteq\cT^{\geq 0} \subseteq \cD^{\geq -k-n}\subseteq \cD^{\geq -m}.
\]
\end{remark}

To construct a pair of $t$-structures satisfying $\cD^{\geq 0}\subseteq \cT^{\geq 0}\subseteq \cD^{\geq -m}$ it is natural to iterate the procedure of tilting with respect to a torsion pair.

\begin{numero}\textbf{Iterated HRS procedure}.\label{num:HRSproc}
Let $\cE_0$ be a $t$-structure in a triangulated category $\cC$. Given a torsion pair $(\cW_0,\cZ_0)$ in the abelian category $\cH_{\cE_0}$, tilting $\cE_0$ with respect to $(\cW_0,\cZ_0)$ one gets in $\cC$ a new $t$-structure $\cE_1$. Consider now a torsion pair $(\cW_1,\cZ_1)$ in $\cH_{\cE_1}$; repeating the same procedure one obtains in $\cC$ the tilted $t$-structure $\cE_2$. Since
\[\cE_0^{\geq 0}\subseteq \cE_1^{\geq 0}\subseteq \cE_0^{\geq -1}\quad\text{and}\quad\cE_1^{\geq 0}\subseteq \cE_2^{\geq 0}\subseteq \cE_1^{\geq -1},\]
we have
\[\cE_0^{\geq 0}\subseteq \cE_1^{\geq 0}\subseteq \cE_2^{\geq 0}\subseteq \cE_1^{\geq -1}\subseteq \cE_0^{\geq -2}\]
and therefore $(\cE_0,\cE_2)$ satisfies
 \[\cE_0^{\geq 0}\subseteq \cE_2^{\geq 0}\subseteq \cE_0^{\geq -2}.\]
Iterating $m$-times this procedure with respect to torsion pairs $(\cW_i,\cZ_i)$ in $\cH_{\cE_i}$, $i=0,1,...,m-1$, we get a $t$-structure $\cE_m$ satisfying\[\cE_0^{\geq 0}\subseteq \cE_m^{\geq 0}\subseteq \cE_0^{\geq -m},\]
and hence of type $(n,k)\in\mathbb N\times\mathbb N$ with $n+k\leq m$. We say that the pair of $t$-structures $(\cE_0,\cE_m)$ has been obtained by an \emph{iterated HRS procedure of length $m$}.
\end{numero}

\begin{remark}
Let $\cD$ be a $t$-structure in a triangulated category $\cC$. Assume $(\cX,\cY)$ is  a torsion pair in the heart $\cH_\cD$. Let $\cT:=\cT_{(\cX,\cY)}$ be the $t$-structure obtained by tilting $\cD$ with respect to $(\cX,\cY)$. Now let us consider in $\cH_\cT=\{C\in\cC: H^0_\cD(C)\in\cX,\ H^{-1}_\cD(C)\in\cY,\  H^{i}_\cD(C)=0\ \forall i\not=-1,0\}$ the torsion pair $(\cY[1],\cX[0])$ (see \cite[Ch.~I, Corollary~2.2]{MR1327209}); tilting $\cT$ with respect to $(\cY[1],\cX[0])$ we get the $t$-structure $\cE_{(\cY[1],\cX[0])}$ defined by
\[\cE^{\leq 0}_{(\cY[1],\cX[0])}:=\{C\in\cC: H^0_\cT(C)\in\cY[1],\; H^{i}_\cT(C)=0\ \forall i>0\}.\]
Let us prove that $\cE^{\leq 0}_{(\cY[1],\cX[0])}=\cD^{\leq -1}$ and hence $\cE_{(\cY[1],\cX[0])}=\cD[1]$.
First let us see $\cE^{\leq 0}_{(\cY[1],\cX[0])}\subseteq\cD^{\leq -1}$.
By Proposition~\ref{polisch} we have
\[\cE^{\leq 0}_{(\cY[1],\cX[0])}\subseteq \cT^{\leq 0}\subseteq \cD^{\leq 0};\]
therefore we have to prove that if $E\in \cE^{\leq 0}_{(\cY[1],\cX[0])}$ then $H^0_\cD(E)=0$. 
Since $H^0_\cT(E)\in\cY[1]$, we have $H^0_\cD(H^0_\cT(E))=0$; then from the distinguished  triangle
\[\tau^{\leq -1}E\to \tau^{\leq 0}E=E\to H^0_\cT(E)\stackrel{+1}{\to}\]
and $\cT^{\leq -1}\subseteq\cD^{\leq -1}$ we get the exact sequence
\[0=H^0_\cD(\tau^{\leq-1} E)\to H^0_\cD(\tau^{\leq 0}E)=H^0_\cD(\E)\to H^0_\cD(H^0_\cT(E))=0\]
and hence $0=H^0_\cD(E)$.\\
Conversely, let us see $\cE^{\leq 0}_{(\cY[1],\cX[0])}\supseteq\cD^{\leq -1}$: if $C$ belongs to $\cD^{\leq -1}\subseteq\cT^{\leq 0}$ we have $H^i_\cT(C)=0$ for each $i>0$ and $H^0_\cT(C)=\tau^{\geq 0}C$. We have to prove that $H^0_\cT(C)$ belongs to $\cY[1]$, i.e., $H^{-1}_\cD(H^0_\cT(C))\in\cY$ and $H^0_\cD(H^0_\cT(C))=0$. Since $H^0_\cT(C)$ belongs to $\cT^{\geq 0}$, we immediately get $H^{-1}_\cD(H^0_\cT(C))\in\cY$. Finally from the distinguished triangle $\tau^{\leq -1}C\to C\to \tau^{\geq 0}C\stackrel{+1}\to$ we get the exact sequence
$0=H^0_\cD(C)\to H^0_\cD(\tau^{\geq 0}C)\to H^1_\cD(\tau^{\leq -1}C)=0$ and hence $H^0_\cD(H^0_\cT(C))=H^0_\cD(\tau^{\geq 0}C)=0$.\\
Observe that if both $\cX$ and $\cY$ are different from 0, then the pair $(\cD,\cT)$ and $(\cT, \cD[1])$ are of type $(1,0)$, but $(\cD, \cD[1])$ is of type $(0,1)$: therefore, iterating the tilting procedure, the gap, the shift, or their sum are not additive functions.
\end{remark}

Now we want to prove that, fixed a $t$-structure $\cD$ in $\cC$, an iterated HRS procedure permits to recover all the $t$-structures $\cT$ which are \emph{$\cD$-filterable} and satisfy condition $\cD^{\geq 0}\subseteq \cT^{\geq 0}\subseteq \cD^{\geq -m}$.

\begin{definition}\label{def:filterable}
Let $\cD$ be a fixed $t$-structure in $\cC$. We say that a $t$-structure $\cT$ in $\cC$ is 
\begin{itemize}
\item \emph{right $\cD$-filterable} if for any $i\in\mathbb Z$ the intersection $\cD^{\geq i}\cap \cT^{\geq 0}$ is a co-aisle;\item \emph{left $\cD$-filterable} if for any $i\in\mathbb Z$ the intersection $\cD^{\leq i}\cap \cT^{\leq 0}$ is an aisle.\end{itemize}
Both the right filterability and left filterability are symmetric notions: 
therefore we will say that the pair $(\cD,\cT)$ is \emph{right (left) filterable} if either
$\cT$ is right (left) $\cD$-filterable or equivalently $\cD$ is right (left) $\cT$-filterable. We call \emph{filterable} a pair of $t$-structures which is right or left filterable.
\end{definition}
\begin{remark}\label{rem:polisfilterable}
A pair $(\cD,\cT)$ of $t$-structures satisfying condition
\[
\cD^{\leq -1}\subseteq \cT^{\leq 0}\subseteq \cD^{\leq 0}\quad\text{or equivalently}\quad
\cD^{\geq 0}\subseteq \cT^{\geq 0}\subseteq \cD^{\geq -1}
\]
is both right and left filterable. Indeed, we have 
\[\cD^{\leq i}\cap\cT^{\leq 0}=\begin{cases}
    \cT^{\leq 0}  & \text{if }i\geq 0, \\
     \cD^{\leq i} & \text{otherwise}
\end{cases}\quad\text{and}\quad
\cD^{\geq i}\cap\cT^{\geq 0}=\begin{cases}
    \cD^{\geq i}  & \text{if }i\geq 0, \\
     \cT^{\geq 0} & \text{otherwise.}
\end{cases}
\]
\end{remark}

Observe that in general the intersection of the aisles or the co-aisles of two $t$-structures is not an aisle or a co-aisle (see e.g. \cite[Lemma~3]{Bondal}): the inclusion of the intersection of the two aisles (co-aisles) in the triangulated category $\cC$ could not admit a right (left) adjoint.

Nevertheless there is a wide class of interesting examples in which this pathology does not occur. 
The following is a (not exhaustive) list of sufficient conditions for a pair of $t$-structures to be filterable.

\begin{lemma}\label{lemma:esempi}
Let $(\cD,\cT)$ be a pair of $t$-structures in a triangulated category $\cC$.
Whenever one of the following conditions holds, the pair $(\cD,\cT)$ is right filterable:
\begin{enumerate}
\item[(1{\phantom{'}})] $\cC$ has countable direct sums and both the co-aisles $\cD^{\geq 0}$ and $\cT^{\geq 0}$ are closed under taking homotopy colimits in $\cC$; the aisle corresponding to $\cD^{\geq i}\cap\cT^{\geq 0}$ is the smallest subcategory of $\cC$ containing both $\cD^{\leq i}$ and $\cT^{\leq 0}$,
closed under suspension, extensions and direct summands.
\item[(2{\phantom{'}})] For each $i\in\mathbb Z$ one has $\tau^{\geq 0}(\cD^{\geq i})\subseteq \cD^{\geq i}$; the aisle corresponding to $\cD^{\geq i}\cap\cT^{\geq 0}$ is the subcategory of $\cC$ of extensions of $\cT^{\leq 0}$ by $\cD^{\leq i}$.
\item[(2')] For each $i\in\mathbb Z$ one has $\delta^{\geq 0}(\cT^{\geq i})\subseteq \cT^{\geq i}$; the aisle corresponding to $\cT^{\geq i}\cap \cD^{\geq 0}$ is the subcategory of $\cC$ of extensions of $\cD^{\leq 0}$ by $\cT^{\leq i}$.
\end{enumerate}
Dually, whenever one of the following conditions holds, the pair $(\cD,\cT)$ is left filterable:
\begin{enumerate}
\item[(i{\phantom{'}})] $\cC$ has countable direct products and both the aisles $\cD^{\leq 0}$ and $\cT^{\leq 0}$ are closed under taking homotopy limits in $\cC$; the co-aisle corresponding to $\cD^{\leq i}\cap\cT^{\leq 0}$  is the smallest subcategory of $\cC$ containing both $\cD^{\geq i}$ and $\cT^{\geq 0}$,
closed under suspension, extensions and direct summands.
\item[(ii{\phantom{'}})] For each $i\in\mathbb Z$ one has $\tau^{\leq 0}(\cD^{\leq i})\subseteq \cD^{\leq i}$; the co-aisle corresponding to $\cD^{\leq i}\cap\cT^{\leq 0}$ is the subcategory of $\cC$ of extensions of $\cD^{\geq i}$ by $\cT^{\geq 0}$.
\item[(ii')] For each $i\in\mathbb Z$ one has $\delta^{\leq 0}(\cT^{\leq i})\subseteq \cT^{\leq i}$;
the co-aisle corresponding to $\cT^{\leq i}\cap \cD^{\leq 0}$ is the subcategory of $\cC$ of extensions of $\cT^{\geq i}$ by $\cD^{\geq 0}$.
\end{enumerate}
\end{lemma}
\begin{proof}
(1) The claim is proved in \cite[Theorem~2.3]{broomhead2013averaging} under the stronger hypothesis that both $\cD$ and $\cT$ are compactly generated. Following the proof one realizes that it is sufficient to assume $\cD^{\geq 0}$ and $\cT^{\geq 0}$ are closed under taking homotopy colimits in $\cC$.\\
(2) The truncation functor $\sigma^{\geq 0}:=\tau^{\geq 0}\circ \delta^{\geq i}$ is left adjoint to the inclusion
$\cD^{\geq i}\cap \cT^{\geq 0}\hookrightarrow \cC$; since $\cD^{\geq i}\cap \cT^{\geq 0}$ is closed under cosuspension and extensions, it is a co-aisle (see \cite[\S 1]{keller1988aisles}). Next for any $M\in \cC$ we have the following diagram 
\[
\xymatrix{\delta^{\leq i}M\ar[r]\ar@{=}[d]&\sigma^{\leq 0}M\ar[d]\ar[r] & \tau^{\leq 0}(\delta^{\geq i+1}(M))\ar[d]\\
\delta^{\leq i} M\ar[r]\ar[d]&M\ar[r]\ar[d]&\delta^{\geq i+1} M\ar[d]\\
0\ar[r]&\sigma^{\geq 1}M\ar@{=}[r]&\tau^{\geq 1}(\delta^{\geq i+1}(M))}
\]
Therefore $\sigma^{\leq 0} M$ is an extension of the objects $\delta^{\leq i} M\in\cD^{\leq i}$ and $\tau^{\leq 0}(\delta^{\geq i+1}(M))\in\cT^{\leq 0}$.
(2') follows by (2) inverting the role of $\cD$ and $\cT$.\\
The dual part follows considering the previous statements in $\cC^\circ$.
\end{proof}

In the sequel we will analyze in detail the right filterable pairs of $t$-structures. We will collect the dual results for the left filterable pairs of $t$-structures in Section~\ref{Sec:left}.

\begin{definition}\label{def:D_i}
Let $(\cD,\cT)$ be a right filterable pair of $t$-structures. We denote by $\cD_i$ the $t$-structure whose co-aisle is $\cD^{\geq -i}\cap \cT^{\geq 0}$, by $\cH_i$ its heart, and by $(\cX_i,\cY_i)$ the torsion pair in $\cH_i$ defined by
$\cX_i:=\cD_{i+1}^{\leq 0}\cap\cH_i$, $\cY_i:= \cD_{i+1}^{\geq 1}\cap\cH_i$.
The $t$-structures $\cD_i$, the hearts $\cH_i$, the torsion classes $\cX_i$ and the torsion-free classes $\cY_i$ are called the \emph{right basic $t$-structures}, the \emph{right basic hearts}, the \emph{right basic torsion classes} and the \emph{right basic torsion-free classes} of $(\cD,\cT)$.
\end{definition}

\begin{lemma}\label{lemma:inclusioni}
Let $(\cD,\cT)$ be a right filterable pair of $t$-structures. Then, for any $i,\ell\in\mathbb Z$ we have the following inclusions
\[\cD_i^{\geq \ell}\subseteq \cD_{i+1}^{\geq \ell}, \quad
\cD_i^{\geq \ell}\subseteq \cD_{i-1}^{\geq \ell-1}\quad\text{and}\quad
\cD_i^{\leq \ell}\supseteq \cD_{i+1}^{\leq \ell}, \quad
\cD_i^{\leq \ell}\supseteq \cD_{i-1}^{\leq \ell-1}.
\]
In particular:
\begin{enumerate}
\item $(\cD_j,\cD_\ell)$ is a right filterable pair of $t$-structures 
for any $j\leq\ell\in\mathbb Z$;
\item $\cD_{i+1}$ is obtained by tilting $\cD_i$ with respect to the torsion pair $(\cX_i,\cY_i)$;
\item $\cX_i=\cD_{i+1}^{\leq 0}\cap \cD_i^{\geq 0}$ and $\cY_i=\cD_{i}^{\leq 0}\cap \cD_{i+1}^{\geq 1}$;
\item for each $m\geq 0$ we have $\cH_{i+m}\subseteq \cD_i^{[-m,0]}$, while $\cH_i\subseteq \cD_{i+m}^{[0,m]}$;
\item if $\cT^{\leq 0}\subseteq \cD^{\leq 0}$ and $i\geq 0$, then $\cY_i=\cT^{\geq 1}\cap\cH_i$.
\end{enumerate}
\end{lemma}
\begin{proof}
We have first
\[\cD_i^{\geq \ell}=\cD^{\geq -i+\ell}\cap \cT^{\geq \ell}\subseteq \cD^{\geq -i-1+\ell}\cap \cT^{\geq \ell}=\cD_{i+1}^{\geq \ell};\]
then, the second inclusion follows by
\[\cD_i^{\geq \ell}=\cD^{\geq -i+\ell}\cap \cT^{\geq \ell}\subseteq \cD^{\geq -i+\ell}\cap \cT^{\geq \ell-1}=
\cD^{\geq -i+1+\ell-1}\cap \cT^{\geq \ell-1}=\cD_{i-1}^{\geq \ell-1}.
\]
The other two inclusions are an easy consequence. 
Next, point 1 follows since for $j\leq\ell$ one has
\[\cD_j^{\geq i}\cap\cD_\ell^{\geq 0}=\left(\cD^{\geq i-j}\cap \cT^{\geq i}\right)\cap\left(\cD^{\geq -\ell}\cap \cT^{\geq 0}\right)=\begin{cases}
     \cD_j^{\geq i} & \text{if }i\geq 0, \\
     \cD_{j-i}^{\geq 0} & \text{if }j-\ell\leq i<0, \\
     \cD_\ell^{\geq 0} & \text{if }i<j-\ell.
\end{cases}
\]
Since $\cD_i^{\geq 0}\subseteq \cD_{i+1}^{\geq 0}\subseteq \cD_i^{\geq -1}$, point 2 is a consequence of Proposition~\ref{polisch}.
Next the equalities $\cX_i=\cD_{i+1}^{\leq 0}\cap \cD_i^{\geq 0}$ and $\cY_i=\cD_{i}^{\leq 0}\cap \cD_{i+1}^{\geq 1}$ in point 3 follow easily by
the inclusions proved in the first part. Let us prove point 4: by definition of right basic $t$-structure, one gets that for any $m\geq 0$
\[\cD_i^{\geq 0}\subseteq\cD_{i+m}^{\geq 0}\subseteq \cD_i^{\geq -m};\]
therefore $\cH_{i+m}=\cD_{i+m}^{\leq 0}\cap \cD_{i+m}^{\geq 0}\subseteq \cD_i^{\leq 0}\cap \cD_i^{\geq -m}=\cD^{[-m,0]}$. The other inclusion follows analogously.
Finally, point 5 can be deduced by
\[\cY_i=\cD_{i+1}^{\geq 1}\cap \cH_i=\cD^{\geq -i}\cap\cT^{\geq 1}\cap \cH_i=\cT^{\geq 1}\cap \cH_i.\]
\end{proof}

\begin{theorem}\label{teo:genPol1}
A right filterable pair $(\cD,\cT)$ of $t$-structures of type $(n,0)$
is obtained by an iterated HRS procedure of length $n$. 
\end{theorem}
\begin{proof}
Denote by $\cD_i$, $i\in\mathbb Z$, the right basic $t$-structures of the pair $(\cD,\cT)$: by Definition~\ref{def:D_i}, $\cD_i$ is the $t$-structure whose co-aisle is $\cD^{\geq -i}\cap\cT^{\geq 0}$.
It is $\cD_0=\cD$ and we have:
\[\xymatrix@-1,7pc{
\cD^{\geq 0}=:&\cD^{\geq 0}_0&\subseteq\cD_1^{\geq 0}\subseteq ...&\subseteq 
\cT^{\geq 0}.\\
&{}\ar[u]
}
\]
The co-aisle of the $t$-structure $\cD_{1}$ is
\[\cD_{1}^{\geq 0}:=
\cD^{\geq -1}\cap \cT^{\geq 0}
.\] By Lemma~\ref{lemma:inclusioni} we have 
$\cD_0^{\geq 0}\subseteq \cD_{1}^{\geq 0}\subseteq \cD_0^{\geq -1}$ and hence by Proposition~\ref{polisch} the $t$-structure $\cD_{1}$ is obtained by tilting $\cD_0$ with respect to the torsion pair
\[
(\cX_0,\cY_0):=(\cD_{1}^{\leq 0}\cap \cH_0, \cD_{1}^{\geq 1}\cap \cH_0)\]
on the heart $\cH_0$ of $\cD_0$:
\[\xymatrix@-1,7pc{
\cD^{\geq 0}=:&\cD^{\geq 0}_0\subseteq&\cD_1^{\geq 0}&\subseteq ...\subseteq 
\cT^{\geq 0}.\\
&&{}\ar[u]
}
\]
The co-aisle of the $t$-structure $\cD_{2}$ is
\[\cD_{2}^{\geq 0}:=
\cD^{\geq -2}\cap \cT^{\geq 0};\]
again by Lemma~\ref{lemma:inclusioni} one has $\cD_{1}^{\geq 0}\subseteq \cD_{2}^{\geq 0}\subseteq \cD_{1}^{\geq -1}$ and hence
the $t$-structure $\cD_{2}$ is obtained by tilting $\cD_{1}$ with respect to the torsion pair
\[
(\cX_{1},\cY_{1}):=(\cD_{2}^{\leq 0}\cap \cH_{1}, \cD_{2}^{\geq 1}\cap \cH_{1})\]
on the heart $\cH_{1}$ of $\cD_{1}$. At any step we get $\cD_{i}^{\geq 0}=\cD^{\geq -i}\cap \cT^{\geq 0}$; for $i=n$ we obtain $\cD_{n}^{\geq 0}=\cD^{\geq -n}\cap \cT^{\geq 0}=\cT^{\geq 0}$:
\[\xymatrix@-1,7pc{
\cD^{\geq 0}=:&\cD^{\geq 0}_0\subseteq\cD_1^{\geq 0}\subseteq ...\subseteq& 
\cD_{n}^{\geq 0}&=\cT^{\geq 0}.\\
&&{}\ar[u]
}
\]
\end{proof}

\begin{corollary}\label{cor:genPol1}
Let $(\cD,\cT)$ be a right filterable pair  of $t$-structures. The pair $(\cD,\cT)$ verifies
\[\cD^{\leq -m}\subseteq \cT^{\leq 0}\subseteq \cD^{\leq 0}\quad\text{or equivalently}\quad
\cD^{\geq 0}\subseteq \cT^{\geq 0}\subseteq \cD^{\geq -m}\]
if and only if $\cT$ is a $t$-structure obtained by $\cD$ with an iterated HRS procedure of length $m$. 
\end{corollary}
\begin{proof} 
Let us suppose $\cD^{\leq -m}\subseteq \cT^{\leq 0}\subseteq \cD^{\leq 0}$. As observed in Remark~\ref{rem:tipo}, $(\cD,\cT)$ is of type $(n,k)\in\mathbb N\times \mathbb N$ with $n+k\leq m$. Denote by $\cD_i$, $i\in\mathbb Z$, the right basic $t$-structures of the pair $(\cD,\cT)$. 
It is $\cD_0=\cD=\cD[0]$,..., $\cD_k=\cD[k]$. As we have observed in Remark~\ref{rem:tipo000110}, tilting $\cD_0$ with respect to the torsion pair $(0,\cH_0:=\cH_\cD)$ we get the $t$-structure $\cD[1]$; next tilting $\cD[1]$ with respect to the torsion pair $(0,\cH_1=\cH_\cD[1])$ we get the  $t$-structure $\cD[2]$. Therefore, iterating this procedure $k$ times we get  $\cD_k=\cD[k]$:
\[\xymatrix@-1,7pc{
\cD^{\geq 0}\subseteq...\subseteq &\cD_k^{\geq 0}&=\cD^{\geq -k}\subseteq
\cT^{\geq 0}\subseteq\cD^{\geq -n-k}\subseteq \cD^{\geq -m}.\\
&{}\ar[u]
}
\]
Now $(\cD_k,\cT)$ is a pair of $t$-structures of type $(n,0)$. By Theorem~\ref{teo:genPol1} with an iterated HRS procedure of length $n$ 
we get $\cD_{k+n}=\cT$.
We have already recovered the $t$-structure $\cT$ via an iterated HRS procedure of length $n+k\leq m$. If one wants to obtain the iterated HRS procedure of length exactly $m$, another step is necessary.
Indeed, tilting $m-n-k$ times the $t$-structure $\cD_{n+k}=\cT$ with respect to the trivial torsion pair $(\cH_\cT,0)$ in the heart $\cH_\cT$ of $\cT$ we get the $t$-structures $\cD_{n+k}=...=\cD_m=\cT$:
\[\xymatrix@-1,7pc{{\phantom{AAAA}}
\cD^{\geq 0}\subseteq...\subseteq \cD^{\geq -k}\subseteq ...\subseteq
\cT^{\geq 0}=\cD_{n+k}^{\geq 0}=...=&\cD_m^{\geq 0}&\subseteq\cD^{\geq -n-k}\subseteq\cD^{\geq -m}.\\
&{}\ar[u]
}
\]
The other direction follows by the construction of the iterated HRS procedure (see \ref{num:HRSproc} and the proof of Theorem~\ref{teo:genPol1}).
\end{proof}

\begin{remark}\label{rem:steps}
Consider a right filterable pair $(\cD,\cT)$ of $t$-structures of type $(n,k)$ satisfying
$\cD^{\geq 0}\subseteq \cT^{\geq 0}\subseteq \cD^{\geq -m}$
as in Corollary~\ref{cor:genPol1}; then
 the basic right $t$-structures $\cD_i$ (see Definition~\ref{def:D_i}) of $(\cD,\cT)$ satisfy the following properties:
\begin{enumerate}
\item For $1\leq i\leq k$, the pairs $(\cD_{i-1},\cD_i)$ are of type $(0,1)$; the pair
$(\cD_0:=\cD,\cD_i)$ is of type $(0,i)$, while the pair $(\cD_i,\cD_m:=\cT)$ is of type $(n,k-i)$.
\item For $k+1\leq i\leq n+k$, the pairs $(\cD_{i-1},\cD_i)$ are of type $(1,0)$; the pair
$(\cD_0:=\cD,\cD_i)$ is of type $(i-k,k)$, while the pair $(\cD_i,\cD_m:=\cT)$ is of type $(n+k-i,0)$. In particular $\cD_{n+k}=\cT$.
\item For $n+k+1\leq i\leq m$, the pairs $(\cD_{i-1},\cD_i)$ are of type $(0,0)$, i.e., $\cD_{n+k}=...=\cD_{m}=\cT$.
\end{enumerate}
\end{remark}

\section{$t$-tree}\label{Sec:ttree}

Let $(\cD, \cT)$ be a right filterable pair of $t$-structures of type $(n,0)$. We have 
\[\cD^{\geq 0}\subseteq\cT^{\geq 0}\subseteq \cD^{\geq -n},\]
and therefore by Theorem~\ref{teo:genPol1} the pair $(\cD, \cT)$ is obtained by an iterated HRS procedure of length $n$. 
Denoted by $\cD_0=\cD$, $\cD_1$,..., $\cD_n=\cT$ the right basic $t$-structures of $(\cD, \cT)$, we recall that for each $i=0,...,n-1$,
any pair $(\cD_{i}, \cD_{i+1})$ is of type $(1,0)$, i.e., $\cD_{i+1}$ is obtained by tilting $\cD_{i}$ with respect to the non trivial torsion pair $(\cX_i,\cY_i):=(\cD_{i+1}^{\leq 0}\cap\cH_i, \cD_{i+1}^{\geq 1}\cap \cH_i)$.
Observe that since $\cX_i=\cD_{i+1}^{\leq 0}\cap\cH_i\subseteq\cH_{i+1}$ and $\cY_i=\cD_{i+1}^{\geq 1}\cap \cH_i\subseteq \cH_{i+1}[-1]$, the torsion class $\cX_i$ is contained in both the hearts $\cH_i$ and $\cH_{i+1}$
while the torsion free class $\cY_i$ is contained in both $\cH_i$ and $\cH_{i+1}[-1]$. 

Starting with an object in $\cH_0$, it first decomposes with respect to the torsion pair $(\cX_0,\cY_0)$, producing a short exact sequence in $\cH_0$. The first term of this short exact sequence, i.e., the torsion part, belongs also to $\cH_{1}$; therefore it decomposes with respect to the torsion pair $(\cX_{1},\cY_{1})$, producing a new short exact sequence in $\cH_{1}$. Analogously, the third term, i.e., the torsion-free part, belongs also to $\cH_{1}[-1]$; therefore it decomposes with respect to the torsion pair $(\cX_{1}[-1],\cY_{1}[-1])$, producing a new short exact sequence in $\cH_{1}[-1]$. Iterating this procedure $n$-times we will get a tree of short exact sequences in the right basic hearts $\cH_0=\cH_\cD$, ..., $\cH_n=\cH_\cT$, that we call the \emph{right $t$-tree} associated to the right filterable pair $(\cD,\cT)$ of $t$-structures.

\begin{theorem}\label{tree}
Let $(\cD, \cT)$ be a right filterable pair of $t$-structures of type $(n,0)$ in $\cC$.
For any object $X$ in $\cH_\cD$ one can functorially construct its \emph{right $t$-tree}
{\tiny
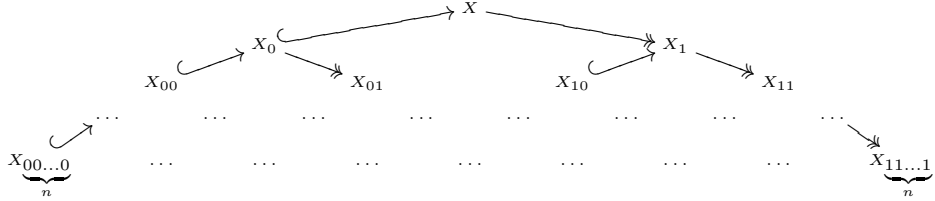
{\[\xymatrix@-1.75pc{
&&&&&&&&X\ar@{->>}[drrrr]
\\
&&&&X_0\ar@{^(->}[urrrr]\ar@{->>}[drr]&&&&&&&&X_1\ar@{->>}[drr]
\\
&&X_{00}
\ar@{^(->}[urr]&&&&
X_{01}
&&&&X_{10}
\ar@{^(->}[urr]&&&&X_{11}
\\
&\hdots&&\hdots&&\hdots&&\hdots&&\hdots&&\hdots&&\hdots&&\hdots\ar@{->>}[dr]
\\
X_{\underbrace{00...0}_n}\ar@{^(->}[ur]&&\hdots&&\hdots&&\hdots&&\hdots&&\hdots&&\hdots&&\hdots&&X_{\underbrace{11...1}_n}
}\]}}\\
whose branches have $n+1$ vertices and where for each $\ell=0,...,n-1$ the sequence
\[0\to X_{i_1...i_{\ell} 0}\to X_{i_1...i_\ell}\to X_{i_1...i_{\ell} 1}\to 0\]
is a short exact sequence in the shifted right basic heart $\cH_{\ell}[-(i_1+\cdots+i_\ell)]$ with  $X_{i_1...i_{\ell} 0}$ belonging to the torsion class $\cX_\ell[-(i_1+\cdots+i_{\ell})]$ and $X_{i_1...i_{\ell} 1}$ belonging to the torsion-free class $\cY_\ell[-(i_1+\cdots+i_{\ell})]$. 
\end{theorem}
\begin{proof}
Denote by $\cD_i$, $\cH_i$, $i=0,1,...,n-1,n$ the right basic $t$-structures and the right basic hearts of $(\cD,\cT)$. Let us consider the torsion pairs $(\cX_i,\cY_i):=(\cD_{i+1}^{\leq 0}\cap\cH_i, \cD_{i+1}^{\geq 1}\cap \cH_i)$ in $\cH_i$, $i=0,1,...,n-1$.
Take an object $X$ in $\cH_0=\cH_\cD$; we denote by $X_0$ and $X_1$ its torsion and torsion-free parts with respect to the torsion pair 
$(\cX_0,\cY_0)$ in $\cH_0$:
\[0\to X_0\to X\to X_1\to 0\quad\text{in }\cH_0.\]
The object $X_0$ belongs also to $\cH_1$, while $X_1$ belongs also to $\cH_1[-1]$. Therefore $X_0$ has a torsion part $X_{00}$ and a torsion free part $X_{01}$ with respect to the torsion pair 
$(\cX_1,\cY_1)$ in $\cH_1$; analogously $X_1$ has a torsion part $X_{10}$ and a torsion free part $X_{11}$ with respect to the torsion pair 
$(\cX_1[-1],\cY_1[-1])$ in $\cH_1[-1]$.
Let $1\leq \ell\leq n-1$; suppose we have obtained the object $X_{i_1...i_\ell}$ in $\cH_{\ell}[-(i_1+\cdots +i_\ell)]$, where $i_1,..., i_\ell\in\{0,1\}$.
We denote by $X_{i_1...i_\ell 0}$ and $X_{i_1...i_\ell 1}$ its torsion and torsion free parts with respect to the torsion pair $(\cX_\ell,\cY_\ell)[-(i_1+\cdots +i_\ell)]$ in $\cH_\ell[-(i_1+\cdots +i_\ell)]$:
\[0\to X_{i_1...i_{\ell} 0}\to X_{i_1...i_\ell}\to X_{i_1...i_{\ell} 1}\to 0\quad\text{in }\cH_{\ell}[-(i_1+\cdots+i_\ell)].\]
The object $X_{i_1...i_\ell 0}$ belongs also to $\cH_{\ell+1}[-(i_1+\cdots +i_{\ell})] $, while $X_{i_1...i_\ell 1}$ belongs also to $\cH_{\ell+1}[-(i_1+\cdots +i_{\ell}+1)] $. This permits to iterate the procedure  untill $\ell=n-1$ obtaining the wished $t$-tree. 
\end{proof}
\begin{definition}
Let $(\cD, \cT)$ be a right filterable pair of $t$-structures of type $(n,0)$ and $X$ an object of $\cH_\cD$. We call the \emph{degree of the vertex} $X_{i_1...i_\ell}$ in the right $t$-tree of  $X$ the sum $i_1+\cdots+i_\ell$. The vertices 
$X_{i_1...i_n}$ are called \emph{right $t$-leaves of the $t$-tree}, and the right $t$-leaf $X_{\underbrace{\scriptstyle{1...1}}_d\underbrace{\scriptstyle{0...0}}_{n-d}}$ is called the \emph{right leading $t$-leaf of degree $d$}.
\end{definition}
\begin{remark}\label{rem:ttree}
The $2^n$ \emph{right $t$-leaves} $X_{i_1...i_{n}}$ in $\cH_{n}[-(i_1+\cdots+i_n)]=\cH_\cT[-(i_1+\cdots+i_n)]$ produced in the last step of the construction of  a right $t$-tree are $\cT$-static objects in $\cC$ of degree $i_1+\cdots+i_n$ (see Definition~\ref{def:static}). Therefore we have got a decomposition of the objects in the heart $\cH_\cD$ in $\cT$-static pieces.
\end{remark}
Let us study the right $t$-tree associated to an object and the information we can obtain from it.

\begin{definition}\label{def:generatedtree}
Let $(\cD, \cT)$ be a right filterable pair of $t$-structures of type $(n,0)$ in $\cC$. Given the right $t$-tree of an object $X\in\cH_\cD$, we define the \emph{subtree generated by the term $X_{i_1...i_\ell}$} to be the subtree of the right $t$-tree of $X$ which has $X_{i_1...i_\ell}$ as root.  
\end{definition}

\begin{remark}\label{rem:subtree}
Let $X\in\cH_\cD$ and $X_{i_1...i_\ell}$ be a vertex of its right $t$-tree. The object $X_{i_1...i_\ell}[i_1+\cdots+i_\ell]$ belongs to  $\cH_\ell$; since $(\cD_\ell,\cT)$ is a right filterable pair of $t$-structures of type $(n-\ell,0)$, we can construct the right $t$-tree of $X_{i_1...i_\ell}[i_1+\cdots+i_\ell]$: this right $t$-tree coincides with the $(i_1+\cdots+i_\ell)$-shift of the subtree of the right $t$-tree of $X$ generated by $X_{i_1...i_\ell}$. The leaves of this subtree are the leaves of the right $t$-tree of $X$ whose index starts with $i_1...i_\ell$.
\end{remark}

In the following proposition we give some cohomological properties of the vertices in the right $t$-tree of an object $X\in\cH_\cD$.

\begin{proposition}\label{cohX}
Let $X\in\cH_\cD$. For each $0\leq \ell\leq n$, the vertex $X_{i_1...i_\ell}$ in the right $t$-tree of $X$ satisfies the following properties:
\begin{enumerate}
\item $X_{i_1\dots i_\ell}$ belongs to $\cT^{[i_1+\cdots+i_\ell,n-\ell+i_1+\cdots+i_\ell]}\subseteq\cT^{[0,n]}$;
\item $H^{i_1+\cdots+i_\ell}_\cT(X_{i_1\dots i_\ell})=X_{i_1...i_\ell\underbrace{\scriptstyle{0...0}}_{n-\ell}}[i_1+\cdots+i_\ell]$;
\item$H^{n-\ell+i_1+\cdots+i_\ell}_\cT(X_{i_1\dots i_\ell})=X_{i_1...i_\ell\underbrace{\scriptstyle{1...1}}_{n-\ell}}[n-\ell+i_1+\cdots+i_\ell]$.
\end{enumerate}
\end{proposition}
\begin{proof}
(1) In Theorem~\ref{tree} we have proved that $X_{i_1...i_\ell}$ belongs to $\cH_{\ell}[-(i_1+\cdots+i_\ell)]$. Since the pair of $t$-structures $(\cD_\ell,\cT)$ is of type $(n-\ell, 0)$ it follows by Lemma~\ref{heart-incl} that
$\cH_{\ell}\subseteq \cT^{[0,n-\ell]}$; therefore the assertion is proved.\\
(2) For any $\ell=0,\dots,n-1$ the short exact sequence
\[
0\to X_{i_1...i_\ell0}\to X_{i_1...i_\ell}\to X_{i_1...i_\ell1}\to 0
\]
in the heart $\cH_\ell[-(i_1+\cdots+i_\ell)]$ provides 
a distinguished triangle in $\cC$. Considering the long exact sequence of $\cT$-cohomology associated to this distinguished triangle, by point 1 we get first
\[0\to H^{i_1+\cdots+i_\ell}_\cT(X_{i_1\dots i_\ell0})\to H^{i_1+\cdots+i_\ell}_\cT(X_{i_1\dots i_\ell})\to
H^{i_1+\cdots+i_\ell}_\cT(X_{i_1\dots i_\ell1})=0.\]
Iterating we have
\[H^{i_1+\cdots+i_\ell}_\cT(X_{i_1\dots i_\ell})\cong H^{i_1+\cdots+i_\ell}_\cT(X_{i_1\dots i_\ell0})\cong ...\]
\[...\cong
H^{i_1+\cdots+i_\ell}_\cT(X_{i_1...i_\ell\underbrace{\scriptstyle{0...0}}_{n-\ell}})=
X_{i_1...i_\ell\underbrace{\scriptstyle{0...0}}_{n-\ell}}[i_1+\cdots+i_\ell].\]
(3) The same long exact sequence of $\cT$-cohomology considered in point 2 gives
\[0=H^{n-\ell+i_1+\cdots+i_\ell}_\cT(X_{i_1\dots i_\ell0})\to H^{n-\ell+i_1+\cdots+i_\ell}_\cT(X_{i_1\dots i_\ell})\to
H^{n-\ell+i_1+\cdots+i_\ell}_\cT(X_{i_1\dots i_\ell1})\to 0.\]
Iterating we have
\[H^{n-\ell+i_1+\cdots+i_\ell}_\cT(X_{i_1\dots i_\ell})\cong H^{n-\ell+i_1+\cdots+i_\ell}_\cT(X_{i_1\dots i_\ell1})\cong ...\]
\[...\cong
H^{n-\ell+i_1+\cdots+i_\ell}_\cT(X_{i_1...i_\ell\underbrace{\scriptstyle{1...1}}_{n-\ell}})=
X_{i_1...i_\ell\underbrace{\scriptstyle{1...1}}_{n-\ell}}[n-\ell+i_1+\cdots+i_\ell].\]
\end{proof}
\begin{remark}
Points 2 and 3 of Proposition~\ref{cohX} give the invariance of the $\cT$-cohomology of degree $i_1+\cdots+i_\ell$ on the \emph{left branch} passing through the vertex $X_{i_1...i_\ell}$ and of the $\cT$-cohomology of degree $n-\ell+i_1+\cdots+i_\ell$ on the \emph{right branch} passing through the vertex $X_{i_1...i_\ell}$:
{\tiny{\[\xymatrix@-1.75pc{
&&&&&&&X_{i_1...i_\ell}\ar@{->>}[drrrr]
\\
&&&X_{i_1...i_\ell0}\ar@{^(->}[urrrr]
&&&&&&&&X_{i_1...i_\ell1}\ar@{->>}[drr]
\\
&\hdots
\ar@{^(->}[urr]
&&&&
&&&&
&&&&\hdots\ar@{->>}[dr]
\\
X_{i_1...i_\ell\underbrace{\scriptstyle{00...0}}_{n-\ell}}\ar@{^(->}[ur]&&&&&&&&&&&&&&X_{i_1...i_\ell\underbrace{\scriptstyle{11...1}}_{n-\ell}}
}\]}}

\end{remark}

Let us analyze the objects in $\cH_\cD$ with a particularly simple right $t$-tree. The following lemma
will be useful in the sequel.

\begin{lemma}\label{heart-int}
Let $(\cD, \cT)$ be a right filterable pair of $t$-structures of type $(n,0)$ in $\cC$.
Then $\cH_\cD\cap\cH_{\cT}[-d]$ is equal to:
\[
\begin{cases}
      0& \text{if }d<0\text{ or }d>n; \\
      \displaystyle\bigcap_{i=0}^{n-1}\cX_i=\bigcap_{i=0}^{n}\cH_i& \text{if }d=0; \\
      \displaystyle\left(\bigcap_{i=0}^{d-1}\cY_i[-i]\right)\!\!\cap\!\!\left(\bigcap_{j=d}^{n-1}\cX_j[-d]\right)\!\!\!=\!\!\!
\left(\bigcap_{i=0}^{d}\cH_i[-i]\right)\!\!\cap\!\!\left(\bigcap_{j=d}^{n}\cH_j[-d]\right)& \text{if }0<d< n; \\
      \displaystyle\bigcap_{i=0}^{n-1}\cY_i[-i]=\bigcap_{i=0}^{n}\cH_i[-i]& \text{if }d=n.
\end{cases}
\]
\end{lemma}
\begin{proof}
Since the pair $(\cD,\cT)$ is of type $(n,0)$, by Lemma~\ref{heart-incl} the heart
$\cH_\cT$ is contained in $\cD^{[-n,0]}$; therefore
$\cH_\cD\cap\cH_{\cT}[-d]=0$  whenever $d<0$ or $d>n$.\\
By Lemma~\ref{lemma:inclusioni}, we have for $i=0,..., n-1$
\[
\cH_\cD\cap\cH_\cT=\cD_0^{\geq 0}\cap\cD_n^{\leq 0}\subseteq \cD_i^{\geq 0}\cap\cD_{i+1}^{\leq 0}=\cD_{i+1}^{\leq 0}\cap \cH_i=\cX_i\subseteq \cH_i\cap\cH_{i+1}\]
which
proves that 
\[\cH_\cD\cap\cH_\cT\subseteq 
\bigcap_{i=0}^{n-1}\cX_i\subseteq \bigcap_{i=0}^{n}\cH_i\subseteq
\cH_0\cap\cH_n=\cH_\cD\cap\cH_{\cT}
\]
and hence $\cH_\cD\cap\cH_\cT=\displaystyle\bigcap_{i=0}^{n-1}\cX_i= \bigcap_{i=0}^{n}\cH_i$.\\
If $0<d\leq n$,  by Lemma~\ref{lemma:inclusioni} we have that for each $0\leq i\leq d-1$
\[\cH_\cD\cap\cH_\cT[-d]\subseteq 
\cD^{\leq 0}\cap \left(\cD^{\geq 0}\cap\cT^{\geq d} \right)=
\cD_0^{\leq 0}\cap\cD_d^{\geq d}\subseteq \]
\[\subseteq   \cD_i^{\leq i} \cap\cD_{i+1}^{\geq i+1} = \cH_i[-i]\cap\cD_{i+1}^{\geq i+1}=\cY_i[-i],\]
while for each $d\leq i\leq n-1$ we have
\[\cH_\cD\cap\cH_\cT[-d]\subseteq
\left(\cD^{\geq 0}\cap\cT^{\geq d} \right)\cap\cT^{\leq d}=
\cD_d^{\geq d}\cap\cD_n^{\leq d}\subseteq \]
\[\subseteq \cD_i^{\geq d} \cap \cD_{i+1}^{\leq d}  =\cH_i[-d]\cap\cD_{i+1}^{\leq d} =
\cX_i[-d].\]
This proves that for $0<d\leq n$
\[\cH_\cD\cap\cH_\cT[-d]\subseteq 
\left(\bigcap_{i=0}^{d-1}\cY_i[-i]\right)\cap
\left(\bigcap_{j=d}^{n-1}\cX_j[-d]\right)\subseteq
\]
in view of $\cY_i\subseteq \cH_i\cap\cH_{i+1}[-1]$ and $\cX_i\subseteq \cH_i\cap\cH_{i+1}$
\[\subseteq\left(\bigcap_{i=0}^{d}\cH_i[-i]\right)\cap
\left(\bigcap_{j=d}^{n}\cH_j[-d]\right)\subseteq 
\cH_0\cap\cH_n[-d]=\cH_\cD\cap\cH_\cT[-d].\]
Thus we get the wanted equalities for $0<d<n$ and for $d=n$ we obtain: 
\[\cH_\cD\cap\cH_{\cT}[-n]=\bigcap_{i=0}^{n-1}\cY_i[-i]=\bigcap_{i=0}^{n}\cH_i[-i].\]
\end{proof}

Let us start by studying the case of a right $t$-tree degenerating in a single branch.

\begin{proposition}\label{cor:torsono}
Let $(\cD, \cT)$ be a right filterable pair of $t$-structures of type $(n,0)$ in $\cC$. 
An object $0\not=X\in\cH_0=\cH_\cD$ has a right $t$-tree with a unique non zero branch if and only if $X$  is $\cT$-static. In such a case, if $X$  is $\cT$-static of degree $d$,  the unique non zero leaf is the right leading leaf of degree $d$.
\end{proposition}
\begin{proof}
If the right $t$-tree of $X$ has a unique non zero branch, necessarily all the maps along this branch are the identity map; therefore $X$ coincides with one of its leaves, and by Remark~\ref{rem:ttree} it is a $\cT$-static object.\\
 Conversely, assume $X$ is $\cT$-static of degree $d$ (i.e.,  $X\in\cH_\cT[-d]$); 
then $X$ belongs to $\cH_\cD\cap\cH_\cT[-d]$.
If $d=0$ by Lemma~\ref{heart-int}  we have
$X\in\bigcap_{i=0}^{n-1}\cX_i$ which proves that 
$X=X_{\underbrace{\scriptstyle{0...0}}_{n}}$. 
If $0<d\leq n$ by Lemma~\ref{heart-int}  we have
$X\in\left(\bigcap_{i=0}^{d-1}\cY_i[-i]\right)\cap\left(\bigcap_{j=d}^{n-1}\cX_j[-d]\right)$
which proves that $X=X_1=\cdots=X_{\underbrace{\scriptstyle{1...1}}_d}$ and then in the remaining $n-d$ steps we have
$X=X_{\underbrace{\scriptstyle{1...1}}_d0}=\cdots=X_{\underbrace{\scriptstyle{1...1}}_d\underbrace{\scriptstyle{0...0}}_{n-d}}$, which is the right leading leaf of degree $d$.
\end{proof}

The previous proposition characterises the case in which the right $t$-tree admits only one non zero leading leaf in $\cH_\cD$. The next result generalises to the case of right $t$-trees whose non zero leaves are leading leaves in $\cH_\cD$.

\begin{proposition}\label{prop:static}
Let $(\cD, \cT)$ be a right filterable pair of $t$-structures of type $(n,0)$ in $\cC$. 
An object $0\not=X\in\cH_0=\cH_\cD$ has a right $t$-tree whose non zero leaves are leading leaves in $\cH_\cD$ if and only if  the $\cT$-cohomologies $H^i_\cT X$ are $\cD$-static of degree $-i$ for each $0\leq i\leq n$. In this case all the vertices of the right $t$-tree of $X$ belong to $\cH_\cD$, and moreover for each $0\leq i\leq n$
\begin{itemize}
\item $X_{\underbrace{\scriptstyle{1...1}}_{i}}=\tau^{\geq i} X$,
\item $X_{\underbrace{\scriptstyle{1...1}}_{i}0}=...=X_{\underbrace{\scriptstyle{1...1}}_{i}\underbrace{\scriptstyle{0...0}}_{n-i}}=H^i_\cT X[-i]$,
\item $X_{\underbrace{\scriptstyle{1...1}}_{i}0j_{i+2}...j_\ell}=0$ for each $j_{i+2}...j_\ell\not=\underbrace{0...0}_{\ell-i-1}$, with $i+2\leq\ell\leq n$.
\end{itemize}
\end{proposition}

\begin{proof}
If $n=0$, the statement is clearly true since $\cD=\cT$. Let $n\geq 1$ and assume the right $t$-tree of $X$ has the leading leaves in $\cH_\cD$ and the non leading ones equal to zero. For $\ell=n-1, n-2, ...,0$, the short exact sequences $0\to X_{j_1...j_\ell0}\to X_{j_1...j_\ell}\to X_{j_1...j_\ell1}\to 0$ in the shifted right basic hearts $\cH_\ell[-(j_1+\cdots+j_\ell)]$
induce distinguished triangles $X_{j_1...j_\ell0}\to X_{j_1...j_\ell}\to X_{j_1...j_\ell1}\stackrel {+1}\to$ in the triangulated category $\cC$.
Starting with $\ell=n-1$, and proceeding along the right $t$-tree from the leaves towards the root, we easily get that all the vertices of the $t$-tree live in $\cH_\cD$.
We prove now that $H^i_\cT X=X_{\underbrace{\scriptstyle{1...1}}_i\underbrace{\scriptstyle{0...0}}_{n-i}}[i]$ for $0\leq i\leq n$; then we conclude as the leading leaves are in $\cH_\cD$.
First of all, forasmuch as the non leading leaves vanish, we have \[
X_{\underbrace{\scriptstyle{1...1}}_i0}=...=X_{\underbrace{\scriptstyle{1...1}}_i\underbrace{\scriptstyle{0...0}}_{n-i}}\in\cH_\cD\cap\cH_\cT[-i],\quad \text{for }0\leq i\leq n-1,\]
\[\text{and}\quad
X_{\underbrace{\scriptstyle{1...1}}_n}\in\cH_\cD\cap\cH_\cT[-n].
{\phantom{AAAA}}
\]
Since $X_{\underbrace{\scriptstyle{1...1}}_i}\in\cY_{i-1}[1-i]\subseteq \cT^{\geq 1}[1-i]=\cT^{\geq i}$,
for $1\leq i\leq n$ the distinguished triangle 
 \[X_{\underbrace{\scriptstyle{1...1}}_{i-1}0}\to X_{\underbrace{\scriptstyle{1...1}}_{i-1}}\to X_{\underbrace{\scriptstyle{1...1}}_i}\stackrel {+1}\to\]
coincides with the approximating triangle of $X_{\underbrace{\scriptstyle{1...1}}_{i-1}}$ with respect to the $t$-structure $\cT[-i+1]$: indeed  $X_{\underbrace{\scriptstyle{1...1}}_{i-1}0}$ belongs to $\cT^{\leq i-1}$ and $X_{\underbrace{\scriptstyle{1...1}}_i}$ belongs to $\cT^{\geq i}$.
In particular 
$X_1=\tau^{\geq 1}X$, $X_{11}=\tau^{\geq 2}X_1=\tau^{\geq 2}(\tau^{\geq 1}X)=\tau^{\geq 2}X$, and hence
\[X_{\underbrace{\scriptstyle{1...1}}_i}=\tau^{\geq i}X, \quad \text{for }i=0,1,...,n.\]
Next $H^0_\cT(X)=X_0=X_{\underbrace{\scriptstyle{0...0}}_n}$, $H^1_\cT(X) =H^1_\cT(\tau^{\geq 1}X)=H^1_\cT(X_1)=X_{10}[1]=X_{1\underbrace{\scriptstyle{0...0}}_{n-1}}[1]$, and hence for $i=0,1,...,n$
\[H^i_\cT(X) =H^i_\cT(\tau^{\geq i}X)=H^i_\cT(X_{\underbrace{\scriptstyle{1...1}}_i})=
X_{\underbrace{\scriptstyle{1...1}}_i0}[i]=
X_{\underbrace{\scriptstyle{1..1}}_{i}\underbrace{\scriptstyle{0...0}}_{n-i}}[i]\in \cH_\cD[i].\]

Conversely, assume the $\cT$-cohomologies $H^i_\cT X$ are $\cD$-static of degree $-i$, i.e., $H^i_\cT X\in\cH_\cD[i]$, for each $0\leq i\leq n$. 

Consider the distinguished triangle
\[H^{n-1}_\cT (X)[-n+1]=\tau^{\leq n-1}(\tau^{\geq n-1} X)\to \tau^{\geq n-1} X\to \tau^{\geq n} X=H^{n}_\cT (X)[-n]
\stackrel {+1}\to.
\]
Since the terms $H^{n-1}_\cT (X)[-n+1]$ and $H^{n}_\cT (X)[-n]$ belong to $\cH_\cD$, also the middle term $\tau^{\geq n-1} X$ belongs to $\cH_\cD$ and hence
\[0\to H^{n-1}_\cT (X)[-n+1]\to \tau^{\geq n-1} X\to \tau^{\geq n} X\to 0\]
is a short exact sequence in $\cH_\cD$.
From the distinguished triangle\[H^{n-2}_\cT (X)[-n+2]=\tau^{\leq n-2}(\tau^{\geq n-2} X)\to \tau^{\geq n-2} X\to \tau^{\geq n-1} X
\stackrel {+1}\to,
\]
considered that $H^{n-2}_\cT (X)[-n+2]$ and $\tau^{\geq n-1} X$ belong to $\cH_\cD$,
one gets that also $\tau^{\geq n-2} X\in\cH_\cD$ and hence
\[0\to H^{n-2}_\cT (X)[-n+2]\to \tau^{\geq n-2} X\to \tau^{\geq n-1} X\to 0\]
is a short exact sequence in $\cH_\cD$.
Iterating the same argument, one proves that
\begin{equation}\label{gtroni}
0\to H^{i}_\cT (X)[-i]\to \tau^{\geq i} X\to \tau^{\geq i+1} X\to 0
\end{equation}
is a short exact sequence in $\cH_\cD$ for any $0\leq i\leq n-1$. 
Moreover applying 
Lemma~\ref{heart-int}  we have:
\[H^{i}_\cT(X)[-i]\in\cH_\cD\cap\cH_\cT[-i]\subseteq
\bigcap_{j=i}^{n-1}\cX_j[-i]\subseteq  \cX_i[-i]\subseteq \cH_i[-i]\]
and by 
Lemma~\ref{lemma:inclusioni} 
\[\tau^{\geq i+1} X\in\cH_\cD\cap \cT^{\geq i+1}\subseteq
\cD^{\leq 0}\cap (\cD^{\geq 0}\cap \cT^{\geq i+1})\subseteq
\cD_i^{\leq i}\cap\cD_{i+1}^{\geq i+1}=\cY_{i}[-i]\subseteq\cH_i[-i]\]
which proves that also the middle term of
$\tau^{\geq i} X\in\cH_i[-i]$ and the sequence 
(\ref{gtroni}) is the short exact sequence  in $\cH_i[-i]$ 
associated to the torsion pair $(\cX_i[-i],\cY_i[-i])$
for any
$0\leq i\leq n-1$. Thus for $i=0$ we have the short exact sequence in $\cH_\cD$:
\[0\to H^{0}_\cT (X)\to \tau^{\geq 0} X=X\to \tau^{\geq 1} X\to 0\]
and hence $\tau^{\geq 1} X=X_1$ and $H^{0}_\cT (X)=X_0=X_{00}=...=X_{\underbrace{\scriptstyle{0...0}}_n}$. In particular the vertices $X_{0j_2...j_\ell}$ vanish for each $j_2...j_\ell\not=\underbrace{0...0}_{\ell-1}$, with $2\leq\ell\leq n$. Next for $i=1$ we have the short exact sequence in $\cH_\cD\cap\cH_1[-1]$:
\[0\to H^{1}_\cT (X)[-1]\to \tau^{\geq 1} X=X_1\to \tau^{\geq 2} X\to 0\]
and hence $\tau^{\geq 2} X=X_{11}$ and $H^{1}_\cT (X)[-1]=X_{10}=X_{100}=...=X_{1\underbrace{\scriptstyle{0...0}}_{n-1}}$. In particular the vertices $X_{10j_3...j_\ell}$ vanish for each $j_3...j_\ell\not=\underbrace{0...0}_{\ell-2}$, with $3\leq\ell\leq n$.
Repeating the same argument we get that for any $0\leq i\leq n-1$ we have the short exact sequence in $\cH_\cD\cap\cH_{i}[-i]$:
\[0\to H^{i}_\cT (X)[-i]\to \tau^{\geq i} X=X_{\underbrace{\scriptstyle{1...1}}_{i}}
\to \tau^{\geq i+1} X\to 0\]
and hence $\tau^{\geq i+1} X=X_{\underbrace{\scriptstyle{1...1}}_{i}1}$
and $H^{i}_\cT (X)[-i]=X_{\underbrace{\scriptstyle{1...1}}_{i}0}=...=X_{\underbrace{\scriptstyle{1...1}}_{i}\underbrace{\scriptstyle{0...0}}_{n-i}}$. In particular the vertices $X_{\underbrace{\scriptstyle{1...1}}_{i}0j_{i+2}...j_\ell}$ vanish for each $j_{i+2}...j_\ell\not=\underbrace{0...0}_{\ell-i-1}$, with $i+2\leq\ell\leq n$.
\end{proof}

\begin{remark}
Let  $X\in\cH_\cD\subset \cT^{[0,n]}$. The $\cT$-truncation functors applied to $X$  perform in the so called  Postnikov tower of $X$ :
\[
\xymatrix@-0.4pc@C-2pc{
H^{0}_\cT(X) \ar[rr]& & \tau^{\leq 1}(X) \ar[rr]\ar[ld] & &\cdots\ar[rr] \ar[ld]&& \tau^{\leq n}(X)=X  \ar[ld]\\
		& H^{1}_\cT(X)[-1]\ar@{-->}^{+1}[lu] & & H^{2}_\cT(X)[-2]\ar@{-->}^{+1}[lu] &&H^n_\cT(X)[-n]\ar@{-->}^{+1}[lu]  & \
}
\]
where all the triangle diagrams are distinguished triangles in the triangulated category $\cC$.
Whenever $X$ satisfies the equivalent conditions of Proposition~\ref{prop:static}
(i.e., $H^i_\cT X[-i]\in\cH_\cD$)
any object written in the previous tower belongs to $\cH_\cD$ and so the distinguished triangles can be regarded as short exact sequences in $\cH_\cD$ and this permits to interpret the Postnikov tower as a filtration of $X$:
\[
\xymatrix@-0.4pc@C-2pc{
H^{0}_\cT(X)0 \ar@{^(->}[rr]& & \tau^{\leq 1}(X) \ar@{^(->}[rr]\ar@{->>}[ld] & &\cdots\ar@{^(->}[rr] \ar@{->>}[ld]&& \tau^{\leq n}(X)=X  \ar@{->>}[ld]\\
		& H^{1}_\cT(X)[-1]
		& & H^{2}_\cT(X)[-2]
		&&H^n_\cT(X)[-n]
		& \
}
\]
whose graded pieces are $\cD$-static.
\end{remark}

\begin{remark}
If $(\cD, \cT)$ is a right filterable pair of $t$-structures of type $(n,k)$, then we can repeat the construction of the right $t$-tree for each $X$ in $\cH_\cD$; the result will be a unique branch with $k+1$ vertices followed by a tree whose branches have $n+1$ vertices:

{\tiny
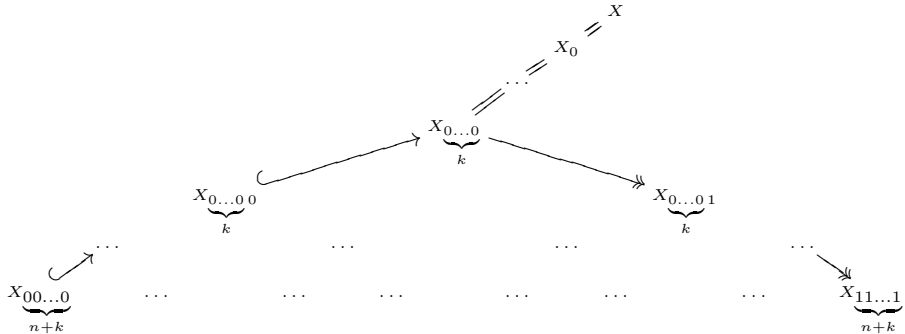
{\[\xymatrix@-1.75pc{
&&&&&&&&&&X\ar@{=}[dl]
\\
&&&&&&&&&X_0\ar@{=}[dl]
\\
&&&&&&&&\hdots\ar@{=}[dl]
\\
&&&&&&&X_{\underbrace{\scriptstyle{0...0}}_k}\ar@{->>}[drrrr]
\\
&&&X_{\underbrace{\scriptstyle{0...0}}_k0}\ar@{^(->}[urrrr]
&&&&&&&&X_{\underbrace{\scriptstyle{0...0}}_k1}
\\
&\hdots
&&&&
\hdots
&&&&\hdots
&&&&\hdots\ar@{->>}[dr]
\\
X_{\underbrace{00...0}_{n+k}}\ar@{^(->}[ur]&&\hdots&&\hdots&&\hdots&&\hdots&&\hdots&&\hdots&&X_{\underbrace{11...1}_{n+k}}
}\]}}
\end{remark}

\section{Left filterable case}\label{Sec:left}

In the previous sections we have studied in detail the theory of  right filterable pairs of $t$-structures $(\cD,\cT)$ satisfying the condition 
$\cD^{\leq -m}\subseteq \cT^{\leq 0} \subseteq \cD^{\leq 0}$ with $m\in \Bbb N$.
This section is devoted to collect the same results in the dual hypothesis of left filterability.
First of all we remark that this duality is obtained once we have fixed the wideness of the involved pair $(\cD,\cT)$, i.e.,  the natural number $m$ such that
\[\cD^{\leq -m}\subseteq \cT^{\leq 0} \subseteq \cD^{\leq 0};\]
in this case  the natural number $m$ plays a role analogous of the dimension in the Serre duality for sheaves. 

\begin{definition}\label{def:{}_iD}
Let $(\cD,\cT)$ be a left filterable pair of $t$-structures with
\[\cD^{\leq -m}\subseteq \cT^{\leq 0} \subseteq \cD^{\leq 0};\]
We indicate with ${}_i\cD$ the $t$-structure whose aisle is ${}_i\cD^{\leq 0}=\cD^{\leq 0}\cap \cT^{\leq m-i}$, by ${}_i\cH$ its heart, and by $({}_i\cX,{}_i\cY)$ the torsion pair in ${}_i\cH$ defined by
${}_i\cX:={}_{i+1}\cD^{\leq 0}\cap{}_i\cH$, ${}_i\cY:= {}_{i+1}\cD^{\geq 1}\cap{}_i\cH$.
The $t$-structures ${}_i\cD$, the hearts ${}_i\cH$, the torsion classes ${}_i\cX$ and the torsion-free classes ${}_i\cY$ are called the \emph{left basic $t$-structures}, the \emph{left basic hearts}, the \emph{left basic torsion classes} and the \emph{left basic torsion-free classes} of $(\cD,\cT)$.
\end{definition}

\begin{lemma}\label{lemma:inclusionileft}
Let $(\cD,\cT)$ be a left filterable pair of $t$-structures. 
Then, for any $i,\ell\in\mathbb Z$ we have the following inclusions
\[
{}_i\cD^{\leq \ell}\supseteq {}_{i+1}\cD^{\leq \ell}, \quad
{}_i\cD^{\leq \ell}\supseteq {}_{i-1}\cD^{\leq \ell-1} \quad\text{and}\quad
{}_i\cD^{\geq \ell}\subseteq {}_{i+1}\cD^{\geq \ell}, \quad
{}_i\cD^{\geq \ell}\subseteq {}_{i-1}\cD^{\geq \ell-1}.
\]
In particular
\begin{enumerate}
\item $({}_j\cD,{}_\ell\cD)$ is a left filterable pair of $t$-structures for any 
$j\leq\ell\in\mathbb Z$;
\item ${}_{i+1}\cD$ is obtained by tilting ${}_i\cD$ with respect to the torsion pair 
$({}_i\cX,{}_i\cY)$;
\item ${}_i\cX={}_{i+1}\cD^{\leq 0}\cap {}_i\cD^{\geq 0}$ and 
${}_i\cY={}_i\cD^{\leq 0}\cap {}_{i+1}\cD^{\geq 1}$;
\item for each $m\geq 0$ we have ${}_{i+m}\cH\subseteq {}_i\cD^{[-m,0]}$, while ${}_i\cH\subseteq {}_{i+m}\cD^{[0,m]}$;
\item if $\cT^{\leq 0}\subseteq \cD^{\leq 0}$ and $i\geq 0$, then ${}_i\cX[m-i]=\cT^{\leq -1}\cap{}_i\cH[m-i]$.
\end{enumerate}
\end{lemma}

The following are the left version of Theorem~\ref{teo:genPol1} and of Corollary~\ref{cor:genPol1}.
We prefer to provide the proof of both these  dual statements since they clarify  
the dual approach.

\begin{theorem}\label{teo:genPolleft}
A left filterable pair $(\cD,\cT)$ of $t$-structures of type $(n,0)$
is obtained by an iterated HRS procedure of length $n$. 
\end{theorem}
\begin{proof}
Denote by ${}_{i}\cD$, $i\in\mathbb Z$, the left basic $t$-structures of the pair $(\cD,\cT)$: by Definition~\ref{def:{}_iD}, ${}_{i}\cD$ is the $t$-structure whose aisle is $\cD^{\leq 0}\cap\cT^{\leq n-i}$. It is ${}_0\cD=\cD$ and we have:
\[\xymatrix@-1,7pc{
\cT^{\leq 0}\subseteq&...\subseteq{}_1\cD^{\leq 0}\subseteq&{}_0\cD^{\leq 0}&=:\cD^{\leq 0}\\
&&{}\ar[u]
}
\]
The aisle of the $t$-structure ${}_{1}\cD$ is
\[{}_{1}\cD^{\leq 0}:=
\cD^{\leq 0}\cap \cT^{\leq n-1}
.\] 
By Lemma~\ref{lemma:inclusionileft} we have
${}_{1}\cD^{\leq -1}\subseteq {}_{0}\cD^{\leq 0}\subseteq {}_{1}\cD^{\leq 0}$ and hence by Proposition~\ref{polisch} the $t$-structure ${}_{1}\cD$ is obtained by tilting 
${}_{0}\cD$ with respect to the torsion pair
\[
({}_0\cX,{}_0\cY):=({}_{1}\cD^{\leq 0}\cap {}_0\cH, {}_{1}\cD^{\geq 1}\cap {}_0\cH)\]
on the heart ${}_0\cH$ of ${}_0\cD$:
\[\xymatrix@-1,7pc{
\cD^{\leq -n}\subseteq
\cT^{\leq 0}\subseteq  ...\subseteq
& {}_{1}\cD^{\leq 0}&\subseteq {}_0\cD^{\leq 0}
= \cD^{\leq 0}.\\
&{}\ar[u]
}
\]
The aisle of the $t$-structure ${}_{2}\cD$ is
\[{}_{2}\cD^{\leq 0}:=
\cD^{\leq 0}\cap \cT^{\leq n-2}.\]
By Lemma~\ref{lemma:inclusionileft} we have ${}_{1}\cD^{\leq -1}\subseteq {}_{2}\cD^{\leq 0}\subseteq {}_{1}\cD^{\leq 0}$ and hence
the $t$-structure ${}_{2}\cD$ is obtained by tilting ${}_{1}\cD$ with respect to the torsion pair
\[
({}_{1}\cX,{}_{1}\cY):=({}_{2}\cD^{\leq 0}\cap {}_{1}\cH, {}_{2}\cD^{\geq 1}\cap {}_{1}\cH)\]
on the heart ${}_1\cH$ of $\cD_{1}$. At any step we get 
${}_{i}\cD^{\leq 0}=\cD^{\leq 0}\cap \cT^{\leq n-i}$; for $i=n$ we obtain 
${}_{n}\cD^{\leq 0}=
\cD^{\leq 0}\cap \cT^{\leq 0}=
\cT^{\leq 0}$:
\[\xymatrix@-1,7pc{
\cD^{\leq -n}\subseteq& \cT^{\leq 0}=&{}_{n}\cD^{\leq 0}& 
\subseteq  \cdots\subseteq  {}_{1}\cD^{\leq 0}\subseteq {}_0\cD^{\leq 0}
= \cD^{\leq 0} .\\
&&{}\ar[u]
}
\]
\end{proof}
\begin{corollary}\label{cor:leftgenPol1}
Let $(\cD,\cT)$ be a left filterable pair of $t$-structures. The pair $(\cD,\cT)$ verifies
\[\cD^{\leq -m}\subseteq \cT^{\leq 0}\subseteq \cD^{\leq 0}\quad\text{or equivalently}\quad
\cD^{\geq 0}\subseteq \cT^{\geq 0}\subseteq \cD^{\geq -m}\]
if and only if $\cT$ is a $t$-structure obtained by $\cD$ with an iterated HRS procedure of length $m$. 
\end{corollary}
\begin{proof}
Let us suppose $\cD^{\leq -m}\subseteq \cT^{\leq 0}\subseteq \cD^{\leq 0}$. The pair $(\cD,\cT)$ is of type $(n,k)\in\mathbb N\times \mathbb N$ with $n+k\leq m$. Denote by ${}_{i}\cD$, $i\in\mathbb Z$, the left basic $t$-structures of the pair $(\cD,\cT)$. 
Dually to the right case we have 
${}_i\cD=\cD$ for any $0\leq i \leq m-(n+k)$. 
Tilting $m-n-k$ times the $t$-structure $\cD$ with respect to the trivial torsion pair $(\cH_\cD,0)$ in the heart $\cH_\cD$ of $\cD$ we get the $t$-structures 
${}_0\cD=\dots={}_{m-n-k}\cD$:
\[\xymatrix@-1,7pc{
\cD^{\leq -m}\subseteq \cT^{\leq 0}\subseteq \cT^{\leq k}\subseteq ...\subseteq
&{}_{m-n-k}\cD^{\leq 0}&
= \dots=\cD_0^{\leq 0}= \cD^{\leq 0} .\\
&{}\ar[u]
}
\]
Now $({}_{m-n-k}\cD, \cT[-k])=({}_{0}\cD, \cT[-k])$ is a pair of $t$-structures of type $(n,0)$. By Theorem~\ref{teo:genPolleft} with an iterated HRS procedure of length $n$ we get ${}_{m-k}\cD=\cT[-k]$.
Finally we have to adjust the shift. 
Tilting ${}_{m-k}\cD=\cT[-k]$ with respect to the torsion pair $(0,\cH_{\cT}[-k])$ we get the $t$-structure ${}_{m-k+1}\cD=\cT[-k+1]$.
Next, tilting ${}_{m-k+1}\cD$ with respect to the torsion pair $(0,\cH_{\cT}[-k+1])$ we get the  $t$-structure ${}_{m-k+2}\cD=\cT[-k+2]$.
Iterating this procedure $k$ times we get  ${}_m\cD=\cT$:
\[\xymatrix@-1.7pc{
\cD^{\leq -m}\subseteq \cT^{\leq 0}=&\!\!\!{}_{m}\cD^{\leq 0}&\!\!\!\subseteq ...\subseteq \cT^{\leq k}={}_{m-k}\cD^{\leq 0}\subseteq ...\subseteq
&\!\!\!{}_{m-n-k}\cD^{\leq 0}&
\!\!\!= \dots= \cD^{\leq 0} .\\
&{}\ar[u]
}
\]
The other direction follows by the construction of the iterated HRS procedure (see the proof of Theorem~\ref{teo:genPolleft}).
\end{proof}
\begin{remark}\label{rem:stepsleft}
Consider a left filterable pair $(\cD,\cT)$ of $t$-structures of type $(n,k)$ satisfying
$\cD^{\leq -m}\subseteq \cT^{\leq 0}\subseteq \cD^{\leq 0}$
as in Corollary~\ref{cor:leftgenPol1}; then
 the basic left $t$-structures ${}_i\cD$ (see Definition~\ref{def:{}_iD}) of $(\cD,\cT)$ satisfy the following properties:
\begin{enumerate}
\item For $1\leq i\leq m-n-k$, the pairs $({}_{i-1}\cD,{}_i\cD)$ are of type $(0,0)$  i.e., $\cD={}_0\cD=...={}_{m-n-k}\cD$.
\item For $m-n-k+1\leq i\leq m-k$, the pairs $({}_{i-1}\cD,{}_i\cD)$ are of type $(1,0)$; the pair
$({}_0\cD:=\cD,{}_i\cD)$ is of type $(i-m+n+k,0)$, while the pair $({}_i\cD,{}_m\cD=\cT)$ is of type $(m-k-i,k)$. In particular ${}_{m-k}\cD=\cT[-k]$.
\item For $m-k+1\leq i\leq m$, the pairs $({}_{i-1}\cD,{}_i\cD)$ are of type $(0,1)$, i.e., ${}_{m-k+1}\cD=\cT[-k+1]$, $_{m-k+2}\cD=\cT[-k+2]$, $\dots$
${}_m\cD=\cT$.
\end{enumerate}
\end{remark}

\bigskip

Let $(\cD, \cT)$ be a left filterable pair of $t$-structures of type $(n,0)$. 
We have 
\[\cD^{\leq -n}\subseteq\cT^{\leq 0}\subseteq \cD^{\leq 0},\]
and therefore by Theorem~\ref{teo:genPolleft} the pair $(\cD, \cT)$ is obtained by an iterated HRS procedure of length $n$. Denoted by ${}_0\cD=\cD$, ${}_1\cD$,..., ${}_n\cD=\cT$ the left basic $t$-structures of $(\cD, \cT)$, we recall that for each $i=0,...,n-1$,
any pair $({}_{i}\cD, {}_{i+1}\cD)$ is of type $(1,0)$, i.e., ${}_{i+1}\cD$ is obtained by tilting ${}_{i}\cD$ with respect to the non trivial torsion pair
$({}_i\cX,{}_i\cY):=({}_{i+1}\cD^{\leq 0}\cap{}_i\cH,\; {}_{i+1}\cD^{\geq 1}\cap {}_i\cH)$.

Let us briefly summarise the left $t$-tree.

\begin{theorem}\label{ltree}
Let $(\cD, \cT)$ be a left filterable pair of $t$-structures of type $(n,0)$ in $\cC$.
For any object $X$ in $\cH_\cD$ one can functorially construct its \emph{left $t$-tree}
{\tiny
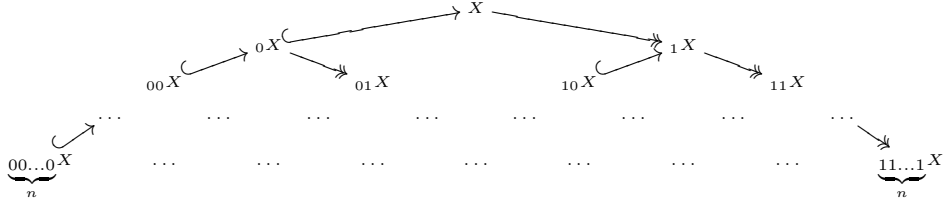
{\[\xymatrix@-1.75pc{
&&&&&&&&X\ar@{->>}[drrrr]
\\
&&&&{}_0X\ar@{^(->}[urrrr]\ar@{->>}[drr]&&&&&&&&{}_1X\ar@{->>}[drr]
\\
&&{}_{00}X
\ar@{^(->}[urr]&&&&
{}_{01}X
&&&&{}_{10}X
\ar@{^(->}[urr]&&&&{}_{11}X
\\
&\hdots&&\hdots&&\hdots&&\hdots&&\hdots&&\hdots&&\hdots&&\hdots\ar@{->>}[dr]
\\
{}_{\underbrace{00...0}_n}X\ar@{^(->}[ur]&&\hdots&&\hdots&&\hdots&&\hdots&&\hdots&&\hdots&&\hdots&&{}_{\underbrace{11...1}_n}X
}\]}}
whose branches have $n+1$ vertices and where for each $\ell=0,...,n-1$ the sequence
\[0\to {}_{i_1...i_{\ell} 0}X\to {}_{i_1...i_\ell}X\to {}_{i_1...i_{\ell} 1}X\to 0\]
is a short exact sequence in the heart ${}_{\ell}\cH[-(i_1+\cdots+i_\ell)]$ with  
${}_{i_1...i_{\ell}0}X$ belonging to the torsion class ${}_{\ell}\cX[-(i_1+\cdots+i_{\ell})]$ and ${}_{i_1...i_{\ell} 1}X$ belonging to the torsion-free class ${}_{\ell}\cY[-(i_1+\cdots+i_{\ell})]$. 
\end{theorem}
\begin{definition}
Let $(\cD, \cT)$ be a left filterable pair of $t$-structures of type $(n,0)$ and $X$ an object of $\cH_\cD$. We call the \emph{degree of the vertex} ${}_{i_1...i_\ell}X$ in the left $t$-tree of  $X$ the sum $i_1+\cdots+i_\ell$. The vertices 
${}_{i_1...i_n}X$ are called \emph{left $t$-leaves of the $t$-tree}, and the left $t$-leaf ${}_{\underbrace{\scriptstyle{0...0}}_{n-d}\underbrace{\scriptstyle{1...1}}_d}X$ is called the \emph{left leading $t$-leaf of degree $d$}.
\end{definition}
\begin{remark}
The $2^n$ \emph{left $t$-leaves} ${}_{i_1...i_{n}}X$ in ${}_n\cH[-(i_1+\cdots+i_n)]=\cH_\cT[-(i_1+\cdots+i_n)]$ produced in the last step of the construction of  a left $t$-tree are $\cT$-static objects in $\cC$ of degree $i_1+\cdots+i_n$ (see Definition~\ref{def:static}). Therefore we have got a decomposition of the objects in the heart $\cH_\cD$ in $\cT$-static pieces.
\end{remark}

\begin{definition}\label{def:generatedtreeleft}
Let $(\cD, \cT)$ be a left filterable pair of $t$-structures of type $(n,0)$ in $\cC$. Given the left $t$-tree of an object $X\in\cH_\cD$, we define \emph{subtree generated by the term ${}_{i_1...i_\ell}X$} the subtree of the left $t$-tree of $X$ which has ${}_{i_1...i_\ell}X$ as root.  
\end{definition}

\begin{proposition}\label{lcohX}
Let $X\in\cH_\cD$. For each $0\leq \ell\leq n$, the vertex ${}_{i_1...i_\ell}X$ in the $t$-tree of $X$ satisfies the following properties:
\begin{enumerate}
\item ${}_{i_1\dots i_\ell}X$ belongs to $\cT^{[i_1+\cdots+i_\ell,n-\ell+i_1+\cdots+i_\ell]}\subseteq\cT^{[0,n]}$;
\item $H^{i_1+\cdots+i_\ell}_\cT({}_{i_1\dots i_\ell}X)=
{}_{i_1...i_\ell\underbrace{\scriptstyle{0...0}}_{n-\ell}}X[i_1+\cdots+i_\ell]$;
\item
$H^{n-\ell+i_1+\cdots+i_\ell}_\cT({}_{i_1\dots i_\ell}X)=
{}_{i_1...i_\ell\underbrace{\scriptstyle{1...1}}_{n-\ell}}X[n-\ell+i_1+\cdots+i_\ell]$.
\end{enumerate}
\end{proposition}

\begin{lemma}\label{heart-intleft}
Let $(\cD, \cT)$ be a left filterable pair of $t$-structures of type $(n,0)$ in $\cC$.
Then $\cH_\cD\cap\cH_{\cT}[-d]$ is equal to:

\[
\begin{cases}
      0& \text{if }d<0\text{ or }d>n; \\
      \displaystyle\bigcap_{i=0}^{n-1}{}_i\cX=\bigcap_{i=0}^{n}{}_i\cH& \text{if }d=0; \\
       \left(\displaystyle\bigcap_{i=0}^{n-d-1}{}_i\cX\right)\!\cap\!
     \displaystyle\left(\bigcap_{j=n-d}^{n-1}{}_j\cY[-j+n-d]\right)=\\
  {\phantom{AAAAAA}}   =
\displaystyle\left(\bigcap_{i=0}^{n-d}{}_i\cH\right)\!\cap\!\left(\bigcap_{j=n-d}^{n}{}_j\cH[-j+n-d]\right)& \text{if }0<d< n; 
\\
      \displaystyle\bigcap_{i=0}^{n-1}{}_i\cY[-i]=\bigcap_{i=0}^{n}{}_i\cH[-i]& \text{if }d=n.
\end{cases}
\]
\end{lemma}

\begin{proposition}\label{cor:torsonoleft}
Let $(\cD, \cT)$ be a left filterable pair of $t$-structures of type $(n,0)$ in $\cC$. 
An object $0\not=X\in{}_0\cH=\cH_\cD$ has a left $t$-tree with a unique non zero branch if and only if $X$  is $\cT$-static. In such a case, if $X$ is $\cT$-static of degree $d$, the unique non zero leaf is the left leading leaf of degree $d$.
\end{proposition}

\begin{proposition}
Let $(\cD, \cT)$ be a left filterable pair of $t$-structures of type $(n,0)$ in $\cC$. 
An object $0\not=X\in{}_0\cH=\cH_\cD$ has a $t$-tree whose non zero leaves are leading leaves in $\cH_\cD$ if and only if  the $\cT$-cohomologies $H^i_\cT X$ are $\cD$-static of degree $-i$ for each $0\leq i\leq n$. In this case all the vertices of the left  $t$-tree of $X$ belong to $\cH_\cD$, and moreover for each $0\leq i\leq n$
\begin{itemize}
\item ${}_{\underbrace{\scriptstyle{0...0}}_{n-i}}X=\tau^{\leq i} X$,
\item ${}_{\underbrace{\scriptstyle{0...0}}_{n-i}1}X=...={}_{\underbrace{\scriptstyle{0...0}}_{n-i}\underbrace{\scriptstyle{1...1}}_{i}}X=H^{i}_\cT X[-i]$,
\item ${}_{\underbrace{\scriptstyle{1...1}}_{i}0j_{i+2}...j_\ell}X=0$ for each $j_{i+2}...j_\ell\not=\underbrace{0...0}_{\ell-i-1}$, with $i+2\leq\ell\leq n$.
\end{itemize}
\end{proposition}

\section{Tilting $t$-structures}\label{TiltTstr}

This paragraph is devoted to a detailed study of the so called
\emph{$n$-tilting} $t$-structures. The motivating example is the $t$-structure on the derived category of left
$R$-modules over a ring $R$ generated by a $n$-tilting
module ${}_RT$.

\begin{definition}
We say that a full subcategory $\cS$ of an abelian category $\cA$ \emph{cogenerates} (resp. \emph{generates}) $\cA$ if any object of $\cA$ embeds in an object of $\cS$ 
(resp. any object of $\cA$ is a quotient of an object in $\cS$).
\end{definition}

Generalising the notion of tilting (cotilting) torsion class introduced in \cite[Ch.~I, \S3]{MR1327209}, we give the following definition.
\begin{definition}\label{Tiltorpair}
A pair $(\cD,\cT)$ of $t$-structures in a triangulated category $\cC$ is \emph{$n$-tilting} (resp. \emph{$n$-cotilting}) if:
\begin{enumerate}
\item $(\cD,\cT)$ is filterable of type $(n,0)$, and
\item the full subcategory $\cH_\cD \cap \cH_\cT$ of $\cH_\cD$ cogenerates $\cH_\cD$
(resp. the full subcategory $\cH_\cD \cap \cH_\cT[-n]$ of $\cH_\cD$ generates $\cH_\cD$).
\end{enumerate}
The pair $(\cD,\cT)$ is \emph{right $n$-tilting} (resp. \emph{right $n$-cotilting}) if it is $n$-tilting (resp. $n$-cotilting) and right filterable; \emph{left $n$-tilting} and \emph{left $n$-cotilting} pair of $t$-structures are similarly defined.
\end{definition}
 
\begin{remark}
Consider an abelian category $\cA$ and  a non trivial torsion pair $(\cX,\cY)$ on $\cA$.
Denote by $\cD$ the natural $t$-structure on $D(\cA)$ and by $\cT_{(\cX,\cY)}$ the $t$-structure obtained by tilting $\cD$ with respect to the torsion pair $(\cX,\cY)$. By Remark~\ref{rem:polisfilterable}, $(\cD,\cT_{(\cX,\cY)})$ is both right and left filterable.
The torsion pair $(\cX,\cY)$ is tilting (resp. cotilting) in the sense of \cite[Ch.~I , \S3]{MR1327209}
if and only if the pair
the $t$-structures $(\cD,\cT_{(\cX,\cY)})$ is $1$-tilting (resp. $1$-cotilting).
\end{remark}
Following our terminology, Theorem~\ref{thm:HRSfundamental} proved by Happel, Reiten, Smal\o \ becomes:
\begin{theorem}\label{heartderiv}
Let $\cA$ be an abelian category and $\cD$ be the natural $t$-structure on $D(\cA)$. Suppose that $(\cD,\cT)$ is a $1$-tilting pair
(resp.
$1$-cotilting pair)
of $t$-structures; then there is a triangle equivalence
\[
\xymatrix{D(\cH_\cT) \ar[r]^-\simeq& D(\cH_\cD)=D(\cA)}
\]
which extends the natural inclusion $\cH_\cT \subseteq D(\cA)$.
\end{theorem}

We want to extend this result to $n$-tilting $t$-structures.

\begin{lemma}\label{resolutions}\cite[Lemma~13.2.1]{kashiwara2006categories}.
Given a cogenerating (resp. generating) full subcategory $\cS$ of an abelian category $\cA$, 
any complex  $X^\bullet$ in $D^b(\cA)$ is quasi-isomorphic to a complex:
\[
S^\bullet = \cdots \to 0 \to S^{i} \to S^{i+1} \to\cdots\quad ({\rm{resp.}} \quad S^\bullet = \cdots\to S^{i-1} \to S^i \to 0 \to \cdots )
\]
where $S^j \in \cS$ for every $j\geq i$ (resp. for every $j\leq i$) and 
$i=\rm{min}\left\{k\in \bZ \,|\, H^k(X^\bullet)\neq0\right\}$
(resp. $i=\rm{max}\left\{k\in \bZ \,|\, H^k(X^\bullet)\neq0\right\}$).
\end{lemma}

\begin{lemma}\label{lemma:tiltingpairsintermedie}
Let $(\cD,\cT)$ be a right $n$-tilting pair of $t$-structures on $D(\cA)$ with $\cD$ the natural $t$-structure. Consider the right basic $t$-structures $\cD_i$, $i=0,..., n$, associated to the pair $(\cD,\cT)$; then the pair $(\cD_i,\cD_{i+j})$ is right $j$-tilting for each $0\leq i\leq n$ and $0\leq j\leq n-i$. Analogous result holds for left $n$-tilting pairs and left/right $n$-cotilting pairs of $t$-structures.
\end{lemma}
\begin{proof}
It is sufficient to prove that both $(\cD,\cD_{n-1})$ and $(\cD_1,\cT)$ are right $(n-1)$-tilting pairs of $t$-structures.
By Lemma~\ref{lemma:inclusioni} and Remark~\ref{rem:steps} the pair $(\cD,\cD_{n-1})$ is right filterable of type $(n-1,0)$. Then by Lemma~\ref{heart-int} the full subcategory $\cH_\cD\cap \cH_{n-1}$ is equal to $\bigcap_{\ell=0}^{n-2}\cH_\ell$ which contains 
$\bigcap_{\ell=0}^{n-1}\cH_\ell=\cH_\cD\cap \cH_\cT$; since the latter cogenerates $\cH_\cD$, also $\cH_\cD\cap \cH_{n-1}$ cogenerates $\cH_\cD$.\\
Let us prove that $(\cD_1,\cT)$ is a right $(n-1)$-tilting pair of $t$-structures.
By Lemma~\ref{lemma:inclusioni} $(\cD_1,\cT)$ is right filterable; let us prove that
the subcategory $\cS_1:=\cH_1 \cap \cH_\cT$ cogenerates 
$\cH_1$. Consider an object  $X$ of $\cH_1\subseteq \cD^{[-1,0]}$. 
By Lemma~\ref{resolutions} applied to the full subcategory $\cS:=\cH_\cD\cap\cH_\cT$ of $\cH_\cD$, we may assume that the complex $X$ is represented by:
\[\xymatrix{
X= &\dots \ar[r] & 0\ar[r] & S^{-1}\ar[r]^{d_X^{-1}}& S^0\ar[r]^{d_X^{0}} &S^1\ar[r]^{d_X^{1}}   & \dots \\ }
\] 
with $S^j\in \cS$ and  $j\geq -1$.
Let us define $W$ to be the complex
\[\xymatrix{W:=&\dots \ar[r] & 0\ar[r] & S^{-1}\ar[r]^{d_X^{-1}}& S^0\ar[r] &0\ar[r]  & \dots \\ }
\] 
The following exact sequence of complexes  $0\to Z\to X\to W\to 0$ gives rise to a distinguished triangle in $D(\cA)$: 
\[
\xymatrix{
Z:=\ar[d]^q  &  \dots \ar[r] \ar[d] & 0\ar[r]\ar[d] & 0 \ar[r]\ar[d]& 0\ar[r]  \ar[d]   &  
S^1\ar[r]^{d_X^{1}}\ar[d]^{\rm{id}_{S^{1}}} & S^2\ar[r]^{d_X^{2}}\ar[d]^{\rm{id}_{S^{2}}} & \dots\ar[d] \\
X=\ar[d]^i &  \dots \ar[r] \ar[d]  & 0\ar[r]\ar[d] & 
S^{-1}\ar[r]^{d_X^{-1}}\ar[d]_{\rm{id}_{S^{-1}}} & 
S^0\ar[r]^{d_X^{0}}   \ar[d]^{\rm{id}_{S^{0}}} & S^1\ar[r]^{d_X^{1}}\ar[d] & S^2\ar[r]^{d_X^{2}}\ar[d] & \dots\\
W=&\dots \ar[r] & 0\ar[r]  & S^{-1}\ar[r]^{d_X^{-1}}& S^0\ar[r] &0\ar[r]  & 0\ar[r] & \dots  \\ 
}
\]
To conclude, we shall show that $W\in \cS_1:=\cH_1\cap\cH_\cT
$ and that $i$ is a monomorphism in $\cH_1$.
First notice that $W\in\cD_1^{\geq 0}$ since $H^i(W)=0$ for any $i\leq -2$
while $H^{-1}(W)\cong H^{-1}(X) \in\cY_0$.
Moreover $W\in\cT^{\leq 0}$ since it is the mapping cone of the morphism
$d^{-1}_X:S^{-1}\to S^0$ between objects in $\cH_\cT$.
Therefore $W$ belongs to $\cD_1^{\geq 0}\cap\cT^{\leq 0}$ which is equal to $\cH_1\cap\cH_\cT$ by Lemma~\ref{lemma:inclusioni}.
In order to prove that $i$ is a monomorphism in $\cH_1$ we have to prove that
its mapping cone $Z[1]$ lies in $\cH_1$, i.e., $H^{-1}(Z[1])$ belongs to $\cY_0$,  $H^{0}(Z[1])$ belongs to $\cX_0$, and $H^{i}(Z[1])=0$ for each $i\not=-1,0$. 
The long exact sequence of $\cD$-cohomology of the distinguished triangle $Z\to X \to W \overset{+1}\to$ 
proves that 
\[\xymatrix{ 0\ar[r] &H^0(X)\ar[r] & H^0(W)\ar[r] & H^1(Z)\ar[r] & H^1(X)=0}\text{ and}\]
$H^i(Z)=0$ for all $i\not= 1$; therefore $Z[1]$ lies in $\cH_1$ since $H^0(Z[1])=H^1(Z)\in \cX_0$: indeed it is a quotient of $H^0(W)$ which is a quotient of $S^0\in\cH_\cD\cap\cH_\cT$ which is contained in $\cX_0$ by Lemma~\ref{heart-int}.
\end{proof}

\begin{theorem}\label{genheartderiv}
Let $\cA$ be an abelian category and $\cD$ be the natural $t$-structure in $D(\cA)$.
Suppose that $(\cD,\cT)$ is a 
$n$-tilting pair
(resp. 
$n$-cotilting pair)
of $t$-structures; then there is a triangle equivalence
\[
\xymatrix{D(\cH_\cT) \ar[r]^-\simeq& D(\cH_\cD)=D(\cA)}
\]
which extends the natural inclusion $\cH_\cT \subseteq D(\cA)$.
\end{theorem}
\begin{proof}
We proceed by induction on the gap $n$. 
In both the tilting and the cotilting cases, for $n=0$ there is nothing to prove and for $n=1$ the result is just Theorem~\ref{heartderiv}.
Assume $(\cD,\cT)$ is right filterable;
suppose $n>1$ and that the statement holds for $n-1$. 
Consider the right basic $t$-structures $\cD_i$, $i=0,..., n$, associated to the pair $(\cD,\cT)$.
Let us assume $(\cD,\cT)$ is $n$-tilting; if $(\cD,\cT)$ is $n$-cotilting one concludes simply dualising the sequel. 

By Lemma~\ref{lemma:tiltingpairsintermedie} $(\cD,\cD_1)$ is right $1$-tilting. Therefore
there exists a triangle equivalence $E:D(\cH_1)\overset{\simeq}\to D(\cA)$ extending the inclusion $\cH_1 \subseteq D(\cA)$.
Via the equivalence
$E:D(\cH_1)\overset{\simeq}\to D(\cA)$, $\cD_1$ is the trivial $t$-structure in $\cD(\cH_1)$, and $\cT$ can be regarded as a $t$-structure in $\cD(\cH_1)$.
Since $(\cD_1,\cT)$ is $(n-1)$-tilting, by inductive hypothesis there is a triangle equivalence $D(\cH_\cT)\to D(\cH_1)$ which extends the inclusion of $\cH_\cT$ in $D(\cH_1)$; composing with the triangle equivalence $E:D(\cH_1)\overset{\simeq}\to D(\cA)$ we get a triangle equivalence $D(\cH_\cT)\overset{\simeq}\to D(\cA)$ which extends the inclusion of $\cH_\cT$ in $D(\cA)$.
{\tiny
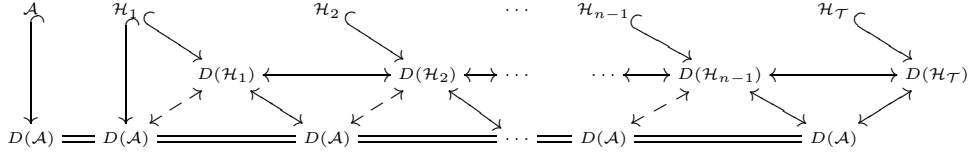
{\[\xymatrix@-1pc{{}\cA\ar@{^(->}[dd]&
\cH_1\ar@{^(->}[dd]\ar@{^(->}[dr]&&
\cH_2
\ar@{^(->}[dr]&&\cdots&
\cH_{n-1}
\ar@{^(->}[dr]&&\cH_{\cT}\ar@{^(->}[dr]
\\
{}&&D(\cH_1)\ar@{<->}[rr]\ar@{<-->}[ld]\ar@{<->}[rd]&&D(\cH_2)\ar@{<->}[r]\ar@{<-->}[ld]\ar@{<->}[rd]&\cdots&\cdots&D(\cH_{n-1})\ar@{<->}[l]\ar@{<-->}[ld]\ar@{<->}[rd]&&D(\cH_{\cT})\ar@{<->}[ll]\ar@{<->}[ld]
\\
D(\cA)\ar@{=}[r]&D(\cA)\ar@{=}[rr]&&D(\cA)\ar@{=}[rr]&&\cdots\ar@{=}[r]&D(\cA)\ar@{=}[rr]&&D(\cA)
}\]}}
If $(\cD,\cT)$ is left filterable, one repeats for both the tilting and cotilting cases analogous arguments, using the associated left basic $t$-structures ${}_i\cD$.
\end{proof}

\section{Applications}
\subsection{Tilting objects in Grothendieck categories}

Along all this section $\cG$ is a fixed Grothendieck category, $T$ is a fixed $n$-tilting object in $\cG$, $\cD$ is the natural $t$-structure in $D(\cG)$ and $\cT_T$ is the $t$-structure compactly generated by $T$ (see Sections III, IV in Preliminaries). By Remark~\ref{rem:tiltingcoaisle} the co-aisle $\cT_T^{\geq 0}$ is closed under direct sums and homotopy colimits.

\begin{lemma}\label{AB5}
The co-aisle $\cD^{\geq 0}$ is closed under
taking homotopy colimits
in $D(\cG)$.
\end{lemma}
\begin{proof}
We use only the fact that any Grothendieck category admits coproducts and filtered colimits of exact sequences are exact.
Let us consider a sequence
$X_0 \overset{f_0}\to X_1 \overset{f_1}\to X_2 \overset{f_2}\to \cdots
$
whose objects $X_n\in \cD^{\geq 0}$ 
and, denote by $\delta$ the truncation functor associated to $\cD$.
Since a coproduct of distinguished triangles is a distinguished triangle (see Remark~\ref{rem:PrelGroth}), we get the following diagram:
\[\xymatrix{
\bigoplus_{n\in \bN}{\rm H}^0_{\cD}(X_n) \ar[rr]^{\id-\oplus_n{\rm H}^0_{\cD}(f_n)} \ar[d] && \bigoplus_{n\in \bN}{\rm H}^0_{\cD}(X_n) \ar[r] \ar[d] &
\HoColim_{n}({\rm H}^0_{\cD}(X_n))\ar[r]^(0.8){+1}\ar[d] & \\
\bigoplus_{n\in \bN}X_n \ar[rr]^{\id -\oplus_n f_n} \ar[d] && \bigoplus_{n\in \bN}X_n \ar[r] \ar[d] &
\HoColim_{n}(X_n)\ar[r]^(0.8){+1}\ar[d] & \\
\bigoplus_{n\in \bN}\delta^{\geq 1}X_n \ar[rr]^{\id -\oplus_n\delta^{\geq 1}f_n} \ar[d]^{+1}&& \bigoplus_{n\in \bN}\delta^{\geq 1}X_n \ar[r]\ar[d]^{+1} &
\HoColim_{n}(\delta^{\geq 1}X_n)\ar[r]^(0.8){+1}\ar[d]^{+1}& \\
 & & & & \\
}\]
whose rows and columns are distingueshed triangles.
 \\ 
 The homotopy colimit
$\HoColim_n(\delta^{\geq 1}X_n)$ belongs to $\cD^{\geq 0}$ 
since it is the mapping cone of a map between direct sums of objects in $\cD^{\geq 1}$, which belong to $\cD^{\geq 1}$
(as seen in Remark~\ref{rem:PrelGroth}).
Then we have
 $\HoColim_{n}(X_n)\in \cD^{\geq 0}$ if and only if
$\HoColim_{n}({\rm H}^0_{\cD}(X_n))\in \cD^{\geq 0}$. 
Let us prove that $g:=\id -\oplus_n{\rm H}^0_{\cD}(f_n)$ is a monomorphism in $\cG$, obtaining that $\HoColim_{n}({\rm H}^0_{\cD}(X_n))\cong \limind_{n}{\rm H}^0_{\cD}(X_n)\in\cG$ 
and hence $\HoColim_{n}(X_n)\in \cD^{\geq 0}$.
\\
In order to prove that $g$ is a monomorphism we have to prove that if
 $\alpha$ belongs to ${\rm Hom}_{\cG}(A,\oplus_{n\in \bN}{\rm H}^0_{\cD}(X_n))$
and $g\circ \alpha=0$, then $\alpha =0$.
Since
\[{\rm Hom}_{\cG}(A,\oplus_{n\in \bN}{\rm H}^0_{\cD}(X_n))\hookrightarrow 
\prod_{n\in \bN}{\rm Hom}_{\cG}(A,{\rm H}^0_{\cD}(X_n)),\]
denoted by $p_m:\oplus_{n\in \bN}{\rm H}^0_{\cD}(X_n)\to {\rm H}^0_{\cD}(X_m)$ the $m$-th projection, a morphism $\beta\in {\rm Hom}_{\cG}(A,\oplus_{n\in \bN}{\rm H}^0_{\cD}(X_n))$ is zero if and only if $p_m\circ \beta=0$ for each $m\in\mathbb N$.
Now we have:
\[p_0 \circ g=p_0;\qquad p_m\circ g=p_m-{\rm H}^0_{\cD}(f_{m-1})\circ p_{m-1}\quad \forall
m\geq 1.\] 
Then $g\circ\alpha=0$ implies $p_m\circ g\circ \alpha=0$ for any
$m\in\bN$. Let us prove by induction that
$p_m\circ\alpha=0$ for each $m\in\mathbb N$.
First $p_0\circ\alpha=p_0\circ g\circ \alpha=0$;
assume by induction that $p_{m-1}\circ \alpha=0$ with  $m\geq 1$
and let us prove that
$p_m\circ \alpha=0$:
\[0=p_m\circ g\circ \alpha=p_m\circ\alpha-{\rm H}^0_{\cD}(f_{m-1})\circ p_{m-1}\circ\alpha=
p_m \circ\alpha \]
which concludes the proof.
\end{proof}

Since both $\cD$ and $\cT_T$ are closed under taking homotopy colimits in $D(\cG)$, by point (1) of Lemma~\ref{lemma:esempi} the pair $(\cD,\cT_T)$ is a right filterable pair of $t$-structures. More, the following result holds:

\begin{proposition}\label{Ttiltingpair}
The pair $(\cD,\cT_T)$ is right $n$-tilting.
\end{proposition}
\begin{proof}
It remains to prove that
\begin{enumerate}
\item $\cG\cap\cH_{\cT_T}$ cogenerates $\cG$;
\item $\cD^{\geq 0}\subseteq \cT_T^{\geq 0}$, $\cD^{\geq 0}\not\subseteq \cT_T^{\geq 1}$,  and $n$ is the minimal natural number such that $\cT_T^{\geq 0}\subseteq \cD^{\geq -n}$.
\end{enumerate}
(1) For any injective object $I$ in $\cG$, one has $\Hom_{D(\cG)}(T, I[i])=0$ for any $i\not= 0$ and hence $I\in\cH_\cT$: indeed
\[ I\in\mathcal T_T^{\leq 0}:=\{X\in D(\cG):\Hom_{D(\cG)}(T, X[i])=0 \text{ for each }i>0\},\]
\[I\in\mathcal T_T^{\geq 0}:=\{X\in D(\cG):\Hom_{D(\cG)}(T, X[i])=0 \text{ for each }i<0\}.\]
Since any 
$X\in \cG$ injects into a suitable injective object $I_X\in\cG$, the intersection $\cG\cap\cH_{\cT_T}$ cogenerates $\cG$.\\
(2) Since $T \in \cG\subseteq \cDll0$, then we get $\cDgg0 \subseteq \cT_T^{\geq 0}$; as $T$ belongs to $\cD^{\geq 0}$, but $T\not\in\cT_T^{\geq 1}$, then $\cD^{\geq 0}\not\subseteq \cT_T^{\geq 1}$.
Let us prove that if $X\in\cT_T^{\geq 0}$ then it belongs to $\cD^{\geq -n}$.
By Lemma 1.17 (2) in \cite{zbMATH06323430}, if $\Hom_{D(\cG)}(T, X[i])=0$ for each $i<0$, then the cohomology objects of $X$ satisfy $H^i_\cD(X)=0$ for all $i<-n$ and so $\cT_T^{\geq 0}\subseteq\cD^{\geq -n}$. Next, since $T$ is $n$-tilting, there exists $M\in\cG$ such that 
\[0\not=\Hom_{D(\cG)}(T, M[n]).\]
Then $M[n-1]\in\cD^{\leq -n+1}$, but $M[n-1]$ does not belong to $\cT_T^{\leq 0}$: therefore $\cD^{\leq -n+1}\not\subseteq\cT_T^{\leq 0}$, or equivalently
$\cT_T^{\geq 0}\not\subseteq\cD^{\geq -n+1}$.
\end{proof}

Denoted by $\cD_i$ and $\cH_i$, $i=0,...,n$, the right basic $t$-structures and hearts of the pair $(\cD,\cT_T)$, by Lemma~\ref{lemma:tiltingpairsintermedie} for each $j>i$
the pairs $(\cD_i,\cD_{j})$ are right $(j-i)$-tilting and therefore by Theorem~\ref{genheartderiv} there are triangle equivalences
\[E_{i,j}:D(\cH_{j})\stackrel{\simeq}\rightarrow D(\cH_i)\]
which extend the natural inclusions $\cH_{j}\subseteq D(\cH_i)$. By construction, for each $0\leq i<j<\ell\leq n$ it is $E_{i,j}\circ E_{j,\ell}=E_{i,\ell}$.

We can say something more: indeed, let us prove that all the right basic hearts $\cH_i$ of the pair $(\cD,\cT_T)$ are Grothendieck categories.

\begin{lemma}\label{lemma:(n-i)tilting!!}
For each $i=0,..., n$ the torsion free class $\cY_i$ coincides with the kernel of the functor $\Hom_{\cH_i}(T,-)$.
\end{lemma}
\begin{proof}
By point 5 of Lemma~\ref{lemma:inclusioni}, we have
$\cY_i=\cH_i\cap\cT_T^{\geq 1}$ for $i\geq 0$.
Then, for $i=0$ one gets $\cY_0=\cH_0\cap\cT_T^{\geq 1}=
\Ker\Hom_{\cH_0}(T,-)$. Assume $i>0$;  if $Y$ belongs to $\cT^{\geq 1}_T \cap \cH_i$, then we have
$\Hom_{\cH_i}(T,Y)\cong \Hom_{D(\cG)}(T,Y)=0$
since $T\in\cT_T^{\leq 0}$.
Conversely, let $Y\in\cH_i$ such that $\Hom_{\cH_{i}}(T,Y)=0$; we have to prove that $Y$ belongs to $\cT^{\geq 1}_T$.
Since $\cH_i\subseteq \cT_T^{[0,n-i]}$, we have the following distinguished triangle
$H^0_{\cT_T}(Y)\to Y\to \tau^{\geq 1}Y\buildrel{+1}\over\to$ in $D(\cG)$. Applying the functor $\Hom_{D(\cG)}(T,-)$ we get the exact sequence
\[\xymatrix{
\Hom_{D(\cG)}(T,(\tau^{\geq 1}Y)[-1])\ar[r]& \Hom_{D(\cG)}(T,H^0_{\cT_T}(Y))\ar[r] &
\Hom_{D(\cG)}(T,Y).
}
\]
Since $T$ belongs to $\cT_T^{\leq 0}$ and $(\tau^{\geq 1}Y)[-1]$ belongs to $\cT_T^{\geq 1}[-1]=\cT_T^{\geq 2}$ we have $\Hom_{D(\cG)}(T,(\tau^{\geq 1}Y)[-1])=0$; in view of $\Hom_{D(\cG)}(T,Y)=\Hom_{D(\cH_i)}(T,Y)=0$, we have $0=\Hom_{D(\cG)}(T,H^0_{\cT_T}(Y))=\Hom_{\cH_{\cT_T}}(T,H^0_{\cT_T}(Y))$. Since $T$ is a
generator of $\cH_{\cT_T}$, we get $H^0_{\cT_T}(Y)=0$ and hence
$Y\cong \tau^{\geq 1}Y\in\cT_T^{\geq 1}\cap\cH_i$.
\end{proof}

\begin{proposition}\label{prop:Grottutto}
For $i=0,\dots,n$, the right basic heart $\cH_i$ of the pair $(\cD,\cT_T)$ is a Grothendieck category, and $T$ is a $(n-i)$-tilting object in $\cH_i$.
\end{proposition}
\begin{proof}
We proceed by induction. For $i=0$ we have $\cH_0=\cG$ which is Grothendieck and $T$ is a $n$-tilting object in $\cG$.
Assume $i\geq 0$, $\cH_{i}$ is a Grothendieck category and $T$ is a $(n-i)$-tilting object in $\cH_i$.
By Proposition~\ref{Ttiltingpair} and Lemma~\ref{lemma:tiltingpairsintermedie},
the pair $(\cD_i,\cD_{i+1})$ is 1-tilting.
Therefore $\cD_{i+1}$ is obtained by tilting $\cD_i$ with respect to the tilting torsion pair
$(\cX_i,\cY_i):=(\cD_{i+1}^{\leq 0}\cap\cH_i,\cD_{i+1}^{\geq 1}\cap\cH_i)$. 
By \cite[Corollary~4.10]{parra2013direct}
the heart $\cH_{i+1}$ is Grothendieck if and only if 
$\cY_i
=\cD_i^{\leq 0}\cap\cD_{i+1}^{\geq 1}$ is closed under taking direct limits in $\cH_i$.
Since $T$ is a $(n-i)$-tilting object in $\cH_{i}$, by Theorems 6.8 and 6.7 in \cite{zbMATH06323430} the functor $\Hom_{\cH_{i}}(T,-)$ preserves direct limits. Since $\cY_{i}=\Ker\Hom_{\cH_{i}}(T,-)$ by 
Lemma~\ref{lemma:(n-i)tilting!!}, we get that $\cY_{i}$ is closed under direct limits in $\cH_{i}$, and therefore $\cH_{i+1}$ is Grothendieck.\\
Let us prove that $T$ is a $(n-i-1)$-tilting object in $\cH_{i+1}$. First of all $T$ belongs to $\cG\cap\cH_{\cT_T}\subseteq \cH_{i+1}$. 
By Lemma~\ref{lemma:tiltingpairsintermedie}
the pair $(\cD,\cD_{i+1})$ is a pair of right $(i+1)$-tilting $t$-structures, so 
by Theorem~\ref{genheartderiv} there is a triangulated equivalence
$E_{0,i+1}:D(\cH_{i+1})\to D(\cG)$ which extends the natural inclusion $\cH_{i+1} \subseteq D(\cG)$. Since $T$ is a fixed point for the equivalence $E_{0,i+1}:D(\cH_{i+1})\to D(\cG)$, it is a compact generator in $D(\cH_{i+1})$ and $\Hom_{D(\cH_{i+1})}(T,T[j])\cong\Hom_{D(\cG)}(T,T[j])=0$ for each $j>0$. Since $(\cD_{i+1},\cT_T)$ is of type $(n-i-1,0)$, by point 2 of Remark~\ref{rem:steps} we have that $i$ is the maximum natural number such that $\cH_{i+1}\subseteq \cT_T^{[0,n-i-1]}$ and hence $\Hom_{D(\cH_{i+1})}(T, L[n-i])\cong\Ext^{n-i}_{\cH_{i+1}}(T,L)=0$ for each $L\in\cH_{i+1}$ and $\Ext^{n-i-1}_{\cH_{i+1}}(T,-)\not\equiv 0$.
\end{proof}

Since the right basic hearts $\cH_i$ are Grothendieck and $T$ is a $(n-i)$-tilting object in $\cH_i$, we have for $i=0,...,n$ the triangle equivalences
\[\R\Hom_{\cH_i}(T,-):D(\cH_i)\to D(\End_{\cH_i}(T))=D(\End_{\cG}(T)).\]

\begin{lemma}\label{lemma:senzator}
For each $i=0,..., n$ we have
\[\R\Hom_{\mathcal G}(T,-)\circ E_{0,i}=\R\Hom_{\cH_i}(T,-);\] 
\end{lemma}
\begin{proof}
By Proposition~\ref{prop:Grottutto}
 $\cH_i$ is a Grothendieck category and so it has enough 
injective objects and we can use the injective model structure to compute the
derived functor $\R\Hom_{\cH_i}(T,-)$.
Let us denote by $\cI_{\cH_i}$ the full subcategory of injective objects in $\cH_i$.
It is sufficient to prove that given $I\in\cI_{\cH_i}$ we have
$\R\Hom_{\mathcal G}(T,I)=\R\Hom_{\cH_i}(T,I)$. 
Since $0=\Hom_{D(\cH_i)}(T,I[j])\cong \Hom_{D(\cG)}(T,I[j])$ for any $j\not= 0$
we have $\cI_{\cH_i}\subseteq \cH_i\cap\cH_\cT$. Since for any
$I\in\cI_{\cH_i}$ both $T, I\in\cH_i\cap\cH_\cT$, we obtain:
\[
\R\Hom_{\mathcal G}(T,I)\cong H_{\End_\cG(T)}^0\R\Hom_{\mathcal G}(T,I)=
\Hom_{D(\cG)}(T,I)\cong\]
\[\cong \Hom_{D(\cH_i)}(T,I)=\Hom_{\cH_i}(T,I)=\R\Hom_{\cH_i}(T,I).
\]
\end{proof}
Considering also the equivalence $\Hom_{\cH_{\cT_T}}(T,-)\colon \cH_{\cT_T} \to \End_\cG(T)\lMod$ proved in \cite[Ch.3 Corollary 4.2]{MR2327478}, we get the following commutative diagram where all dotted arrows are triangulated equivalences and the dashed arrow is a Morita equivalence:
{\tiny
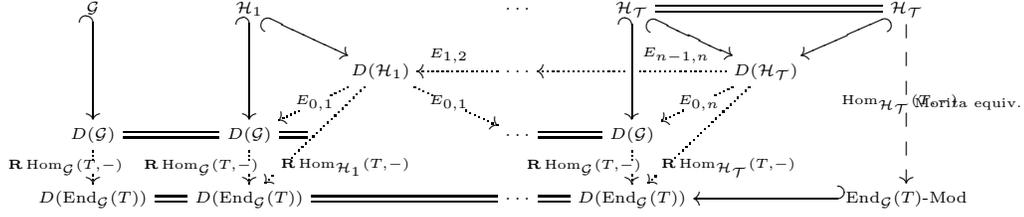
{\[\xymatrix@-1pc{{}\cG\ar@{_(->}[dd]&
\cH_1\ar@{_(->}[dd]\ar@{_(->}[dr]&&&
\cdots&
\cH_{\cT}\ar@{_(->}[dd]
\ar@{_(->}[dr]\ar@{=}[rr]&&\cH_{\cT}\ar@{^(->}[dl]\ar@{-->}[ddd]|{\Hom_{\cH_\cT}(T,-){\phantom{A}}}^[@!-90]
{\text{Morita equiv.}}
\\
{}&&D(\cH_1)\ar@{.>}[ddl]|(.75){{\phantom{AAAAAAAAA}}\R\Hom_{\cH_1}(T,-)}
\ar@{.>}[ld]|{E_{0,1}}\ar@{.>}[rrd]|{E_{0,1}}&&\cdots\ar@{.>}[ll]_{E_{1,2}}&&
D(\cH_{\cT})\ar@{.>}[ddl]|(.75){{\phantom{AAAAAAAAA}}\R\Hom_{\cH_\cT}(T,-)}
\ar@{.>}[ll]|!{[ul];[ld]}\hole_{{\phantom{AAAAA}}E_{n-1,n}}\ar@{.>}[ld]|{E_{0,n}}
\\
D(\cG)\ar@{=}[r]&D(\cG)\ar@{=}[rrr]|!{[ur];[d]}\hole
&&&\cdots\ar@{=}[r]&D(\cG)\\
D(\End_\cG(T))\ar@{<.}[u]|{\R\Hom_\cG(T,-){\phantom{AAAA}}}
\ar@{=}[r]&D(\End_\cG(T))\ar@{<.}[u]|{\R\Hom_\cG(T,-){\phantom{AAAAAAA}}}\ar@{=}[rrr]&&&\cdots\ar@{=}[r]&D(\End_\cG(T))\ar@{<.}[u]|{\R\Hom_\cG(T,-){\phantom{AAAAAAA}}}&&\End_\cG(T)\lMod\ar@{_(->}[ll]
}\]}}

\subsection{The $t$-tree associated to a $n$-tilting module}
For the rest of this section $\cG=R\lMod$ with $R$ an arbitrary associative ring and therefore $T={}_RT$ is a $n$-tilting left $R$-module. Miyashita in \cite{Miya} introduced the following $n+1$ full subcategories:
\[KE_{e}(T):=\{M\in R\lMod:\Ext^i_{R}(T,M)=0,  \text{ if }i\not=e\},\text{ and}\]
\[KT_{e}(T):=\{N\in S\lMod:\Tor_i^{S}(T,N)=0, \text{ if }i\not=e\},
\quad e=0,1,...,n,\]
and proved that the functors  $\Ext^e_{R}(T,-)$ and $\Tor_e^S(T,-)$ induce inverse equivalences between $KE_{e}(T)$ and $KT_{e}(T)$.
Following Definition~\ref{def:static}, the objects in $KE_{e}(T)$ are exactly the objects in $R\lMod$ which are $\cT_T$-static of degree $e$.

If $n\leq 1$, Brenner and Butler in \cite{zbMATH03697330} observed that $(KE_{0}(T), KE_{1}(T))$ is a torsion pair in $R\lMod$. In particular any object in $R\lMod$ is an extension of a $\cT_T$-static module of degree 1 by a $\cT_T$-static module of degree 0.

As soon as $n>1$, we loose the 
possibility to \emph{decompose} all the objects in $R\lMod$ in  $\cT_T$-static modules: in \cite{zbMATH01958651} examples of simple non $\cT_T$-static
modules are provided. We recover the decomposition of all left $R$-modules in $\cT_T$-static objects with the construction of their $t$-trees (see Section~\ref{Sec:ttree}): indeed, as we have seen in  Remark~\ref{rem:ttree}, the $t$-leaves are $\cT_T$-static. In the case $n=1$ the $t$-tree of a module coincide with its decomposition with respect to the torsion pair $(KE_{0}(T), KE_{1}(T))$.
Therefore we can regard 
the construction of the $t$-tree as a generalization of the Brenner and Butler Theorem.

\begin{example}
In this example, $k$ denotes an algebraically closed field. We will consider a finite-dimensional path  $k$-algebra given by a quiver with relations. If $\ell$, $m$ and $n$ are vertices of the quiver, we continue to denote by $\ell$, $m$ and $n$ the correspondent simple module; $\begin{smallmatrix} \ell\\ n\end{smallmatrix}$ denotes the indecomposable module whose radical (and also socle) is the simple module $n$ and whose top is the simple module $\ell$, while $\begin{smallmatrix} \ell&&m\\&n\end{smallmatrix}$ denotes the indecomposable module whose radical (and also socle) is the simple module $n$ and whose top is the direct sum $\ell\oplus m$.
Let $R$ denote the path $k$-algebra given by the quiver
	\[
\xymatrix{
{\cdot}^{1}\ar[r]&{\cdot}^{2}\ar[r]&{\cdot}^{3}\ar[r]&{\cdot}^{4}\ar[r]&{\cdot}^{5}&{\cdot}^{6}\ar[l]}
\]
with relations such that the left projective modules are $\begin{smallmatrix} 1\\2\end{smallmatrix}$,
$\begin{smallmatrix} 2\\3\end{smallmatrix}$, $\begin{smallmatrix} 3\\4\end{smallmatrix}$,
$\begin{smallmatrix} 4\\5\end{smallmatrix}$, $\begin{smallmatrix} 6\\5\end{smallmatrix}$, $\begin{smallmatrix} 5\end{smallmatrix}$. Let $_RT$ be the left $R$-module 
\[_RT:=\begin{smallmatrix} 4&&6\\&5\end{smallmatrix}\oplus \begin{smallmatrix} 6\end{smallmatrix}
\oplus \begin{smallmatrix} 3\\4\end{smallmatrix}\oplus \begin{smallmatrix} 2\\3\end{smallmatrix}\oplus \begin{smallmatrix} 2\end{smallmatrix}\oplus \begin{smallmatrix} 1\\ 2\end{smallmatrix}.\]
The module $_RT$ is a classical 3-tilting object in $R\lMod$.
It is not difficult to verify that, denoted by $\Ind_R$ the subcategory of indecomposable modules in $R\lMod$, we have
\[\Ind_R=\{\begin{smallmatrix} 1\end{smallmatrix}, \begin{smallmatrix} 2\end{smallmatrix}, \begin{smallmatrix} 3\end{smallmatrix}, \begin{smallmatrix} 4\end{smallmatrix}, \begin{smallmatrix} 5\end{smallmatrix}, \begin{smallmatrix} 6\end{smallmatrix}, \begin{smallmatrix} 1\\2\end{smallmatrix}, \begin{smallmatrix} 2\\3\end{smallmatrix}, \begin{smallmatrix} 3\\4\end{smallmatrix}, \begin{smallmatrix} 4\\5\end{smallmatrix}, \begin{smallmatrix} 6\\5\end{smallmatrix}, \begin{smallmatrix} 4&&6\\&5\end{smallmatrix}\}.\]
If $L$, $M$ and $N$ are left $R$-modules, we will denote by $L\to \stackrel{\point}M\to N$ the bounded complex with $M$ in degree zero.
The derived category $D(R)$ has a finite number of indecomposable complexes, which are up to shifts
\[\{\begin{smallmatrix} 1\end{smallmatrix}, \begin{smallmatrix} 2\end{smallmatrix}, \begin{smallmatrix} 3\end{smallmatrix}, \begin{smallmatrix} 4\end{smallmatrix}, \begin{smallmatrix} 5\end{smallmatrix}, \begin{smallmatrix} 6\end{smallmatrix}, \begin{smallmatrix} 1\\2\end{smallmatrix}, \begin{smallmatrix} 2\\3\end{smallmatrix}, \begin{smallmatrix} 3\\4\end{smallmatrix}, \begin{smallmatrix} 4\\5\end{smallmatrix}, \begin{smallmatrix} 6\\5\end{smallmatrix}, \begin{smallmatrix} 4&&6\\&5\end{smallmatrix},\begin{smallmatrix} 2\\3\end{smallmatrix}\to\stackrel{\point}{\begin{smallmatrix} 1\\2\end{smallmatrix}}, \begin{smallmatrix} 3\\4\end{smallmatrix}\to\stackrel{\point}{\begin{smallmatrix} 2\\3\end{smallmatrix}}, \begin{smallmatrix} 4\\5\end{smallmatrix}\to\stackrel{\point}{\begin{smallmatrix} 3\\4\end{smallmatrix}}, 
\begin{smallmatrix} 4&&6\\&5\end{smallmatrix}\to \stackrel{\point}{\begin{smallmatrix} 3\\4\end{smallmatrix}}, \]
\[
\begin{smallmatrix} 3\\4\end{smallmatrix}\to\begin{smallmatrix} 2\\3\end{smallmatrix}\to \stackrel{\point}{\begin{smallmatrix} 1\\2\end{smallmatrix}}, \begin{smallmatrix} 4\\5\end{smallmatrix}\to\begin{smallmatrix} 3\\4\end{smallmatrix}\to\stackrel{\point}{\begin{smallmatrix} 2\\3\end{smallmatrix}},
\begin{smallmatrix} 4&&6\\&5\end{smallmatrix}\to\begin{smallmatrix} 3\\4\end{smallmatrix}\to\stackrel{\point}{\begin{smallmatrix} 2\\3\end{smallmatrix}}, \begin{smallmatrix} 4\\5\end{smallmatrix}\to\begin{smallmatrix} 3\\4\end{smallmatrix}\to\begin{smallmatrix} 2\\3\end{smallmatrix}\to\stackrel{\point}{\begin{smallmatrix} 1\\2\end{smallmatrix}}, \begin{smallmatrix} 4&&6\\&5\end{smallmatrix}\to\begin{smallmatrix} 3\\4\end{smallmatrix}\to\begin{smallmatrix} 2\\3\end{smallmatrix}\to\stackrel{\point}{\begin{smallmatrix} 1\\2\end{smallmatrix}}
\}\]
Denoted by $\cD$ the natural $t$-structure in $D(R)$, we have proved that $(\cD,\cT_T)$ is a right 3-tilting pair of $t$-structures, in particular it is right filterable. Therefore for each object in $\cH_\cD=R\lMod$ we can construct its right $t$-tree.\\
Some computation permits to obtain the indecomposable complexes belonging to the right basic hearts $R\lMod=\cH_0$, $\cH_1$, $\cH_2$ and $\cH_3=\cT_T$ associated to the pair $(\cD,\cT_T)$ of $t$-structures:
\[\cH_0=R\lMod=\{\begin{smallmatrix} 1\end{smallmatrix}, \begin{smallmatrix} 2\end{smallmatrix}, \begin{smallmatrix} 3\end{smallmatrix}, \begin{smallmatrix} 4\end{smallmatrix}, \begin{smallmatrix} 5\end{smallmatrix}, \begin{smallmatrix} 6\end{smallmatrix}, \begin{smallmatrix} 1\\2\end{smallmatrix}, \begin{smallmatrix} 2\\3\end{smallmatrix}, \begin{smallmatrix} 3\\4\end{smallmatrix}, \begin{smallmatrix} 4\\5\end{smallmatrix}, \begin{smallmatrix} 6\\5\end{smallmatrix}, \begin{smallmatrix} 4&&6\\&5\end{smallmatrix}\}\]
\[\cH_1=\{\begin{smallmatrix} 1\end{smallmatrix}, \begin{smallmatrix} 2\end{smallmatrix}, \begin{smallmatrix} 3\end{smallmatrix}, \begin{smallmatrix} 4\end{smallmatrix},\begin{smallmatrix} 5\end{smallmatrix}\!\!\begin{smallmatrix} [1]\end{smallmatrix},  \begin{smallmatrix} 6\end{smallmatrix}, \begin{smallmatrix} 1\\2\end{smallmatrix}, \begin{smallmatrix} 2\\3\end{smallmatrix}, \begin{smallmatrix} 3\\4\end{smallmatrix}, \begin{smallmatrix} 4\\5\end{smallmatrix}\!\!\begin{smallmatrix} [1]\end{smallmatrix}, \begin{smallmatrix} 6\\5\end{smallmatrix}\!\!\begin{smallmatrix} [1]\end{smallmatrix}, \begin{smallmatrix} 4&&6\\&5\end{smallmatrix}, \begin{smallmatrix} 4\\5\end{smallmatrix}\to\stackrel{\point}{\begin{smallmatrix} 3\\4\end{smallmatrix}}, \begin{smallmatrix} 4&&6\\&5\end{smallmatrix}\to\stackrel{\point}{\begin{smallmatrix} 3\\4\end{smallmatrix}}\}
\]
\[\cH_2=\{\begin{smallmatrix} 1\end{smallmatrix}, \begin{smallmatrix} 2\end{smallmatrix}, \begin{smallmatrix} 6\end{smallmatrix}, \begin{smallmatrix} 1\\2\end{smallmatrix}, \begin{smallmatrix} 2\\3\end{smallmatrix}, \begin{smallmatrix} 3\\4\end{smallmatrix}, \begin{smallmatrix} 4\\5\end{smallmatrix}\!\!\begin{smallmatrix} [1]\end{smallmatrix}, \begin{smallmatrix} 6\\5\end{smallmatrix}\!\!\begin{smallmatrix} [2]\end{smallmatrix}, \begin{smallmatrix} 4&&6\\&5\end{smallmatrix}, \begin{smallmatrix} 4\\5\end{smallmatrix}\to\stackrel{\point}{\begin{smallmatrix} 3\\4\end{smallmatrix}}, \begin{smallmatrix} 4&&6\\&5\end{smallmatrix}\to\stackrel{\point}{\begin{smallmatrix} 3\\4\end{smallmatrix}},
\begin{smallmatrix} 4&&6\\&5\end{smallmatrix}\to\begin{smallmatrix} 3\\4\end{smallmatrix}\to\stackrel{\point}{\begin{smallmatrix} 2\\3\end{smallmatrix}}\}
\]
\[\cH_3=\{\begin{smallmatrix} 1\end{smallmatrix}, \begin{smallmatrix} 2\end{smallmatrix}, \begin{smallmatrix} 6\end{smallmatrix}, \begin{smallmatrix} 1\\2\end{smallmatrix}, \begin{smallmatrix} 2\\3\end{smallmatrix}, \begin{smallmatrix} 3\\4\end{smallmatrix}, \begin{smallmatrix} 4\\5\end{smallmatrix}\!\!\begin{smallmatrix} [1]\end{smallmatrix}, \begin{smallmatrix} 6\\5\end{smallmatrix}\!\!\begin{smallmatrix} [3]\end{smallmatrix}, \begin{smallmatrix} 4&&6\\&5\end{smallmatrix}, \begin{smallmatrix} 4\\5\end{smallmatrix}\to\stackrel{\point}{\begin{smallmatrix} 3\\4\end{smallmatrix}}, \begin{smallmatrix} 4&&6\\&5\end{smallmatrix}\to\stackrel{\point}{\begin{smallmatrix} 3\\4\end{smallmatrix}},\]
\[
\begin{smallmatrix} 4&&6\\&5\end{smallmatrix}\to\begin{smallmatrix} 3\\4\end{smallmatrix}\to\stackrel{\point}{\begin{smallmatrix} 2\\3\end{smallmatrix}},
\begin{smallmatrix} 4&&6\\&5\end{smallmatrix}\to\begin{smallmatrix} 3\\4\end{smallmatrix}\to
\begin{smallmatrix} 2\\3\end{smallmatrix}\to\stackrel{\point}{\begin{smallmatrix} 1\\2\end{smallmatrix}}\}
\]
The indecomposable objects belonging to the right basic torsion pairs are:
\[\cX_0=\{\begin{smallmatrix} 1\end{smallmatrix}, \begin{smallmatrix} 2\end{smallmatrix}, \begin{smallmatrix} 3\end{smallmatrix}, \begin{smallmatrix} 4\end{smallmatrix}, \begin{smallmatrix} 6\end{smallmatrix},\begin{smallmatrix} 1\\2\end{smallmatrix}, \begin{smallmatrix} 2\\3\end{smallmatrix}, \begin{smallmatrix} 3\\4\end{smallmatrix},\begin{smallmatrix} 4&&6\\&5\end{smallmatrix}\},\ \cY_0=\{\begin{smallmatrix} 5\end{smallmatrix},\begin{smallmatrix} 4\\5\end{smallmatrix}, \begin{smallmatrix} 6\\5\end{smallmatrix}\}\]
\[\cX_1=\{\begin{smallmatrix} 1\end{smallmatrix}, \begin{smallmatrix} 2\end{smallmatrix},  \begin{smallmatrix} 3\end{smallmatrix}, \begin{smallmatrix} 6\end{smallmatrix},\begin{smallmatrix} 1\\2\end{smallmatrix}, \begin{smallmatrix} 2\\3\end{smallmatrix}, \begin{smallmatrix} 3\\4\end{smallmatrix},\begin{smallmatrix} 4&&6\\&5\end{smallmatrix},\begin{smallmatrix} 4\\5\end{smallmatrix}\!\!\begin{smallmatrix} [1]\end{smallmatrix}, \begin{smallmatrix} 4\\5\end{smallmatrix}\to\stackrel{\point}{\begin{smallmatrix} 3\\4\end{smallmatrix}}, \begin{smallmatrix} 4&&6\\&5\end{smallmatrix}\to\stackrel{\point}{\begin{smallmatrix} 3\\4\end{smallmatrix}}
\},\ \cY_1=\{ \begin{smallmatrix} 6\\5\end{smallmatrix}\!\!\begin{smallmatrix} [1]\end{smallmatrix}\}\]
\[\cX_2=\{\begin{smallmatrix} 1\end{smallmatrix}, \begin{smallmatrix} 2\end{smallmatrix},  \begin{smallmatrix} 6\end{smallmatrix},\begin{smallmatrix} 1\\2\end{smallmatrix}, \begin{smallmatrix} 2\\3\end{smallmatrix}, \begin{smallmatrix} 3\\4\end{smallmatrix},\begin{smallmatrix} 4&&6\\&5\end{smallmatrix},\begin{smallmatrix} 4\\5\end{smallmatrix}\!\!\begin{smallmatrix} [1]\end{smallmatrix}, \begin{smallmatrix} 4\\5\end{smallmatrix}\to\stackrel{\point}{\begin{smallmatrix} 3\\4\end{smallmatrix}}, \begin{smallmatrix} 4&&6\\&5\end{smallmatrix}\to\stackrel{\point}{\begin{smallmatrix} 3\\4\end{smallmatrix}}, \begin{smallmatrix} 4&&6\\&5\end{smallmatrix}\to\begin{smallmatrix} 3\\4\end{smallmatrix}\to\stackrel{\point}{\begin{smallmatrix} 2\\3\end{smallmatrix}}
\},\ \cY_2=\{ \begin{smallmatrix} 6\\5\end{smallmatrix}\!\!\begin{smallmatrix} [2]\end{smallmatrix}\}\]
The simple modules $\begin{smallmatrix} 3\end{smallmatrix}$, $\begin{smallmatrix} 4\end{smallmatrix}$ and $\begin{smallmatrix}5\end{smallmatrix}$ are not  $\cT$-static; since they are simple, they admit in $\R\lMod$ only trivial decompositions. Let us construct their $t$-trees:
{\tiny{\[\xymatrix@-1.5pc{
&&&&&&&{\begin{smallmatrix}3\end{smallmatrix}}\ar@{->>}[drrrr]
\\
&&&{\begin{smallmatrix}3\end{smallmatrix}}\ar@{^(->}[urrrr]\ar@{->>}[drr]&&&&&&&&
{\begin{smallmatrix}0\end{smallmatrix}}\ar@{->>}[drr]
\\
&{\begin{smallmatrix}3\end{smallmatrix}}\ar@{^(->}[urr]\ar@{->>}[dr]&&&&
{\begin{smallmatrix}0\end{smallmatrix}}\ar@{->>}[dr]&&&&
{\begin{smallmatrix}0\end{smallmatrix}}\ar@{^(->}[urr]\ar@{->>}[dr]&&&&
{\begin{smallmatrix}0\end{smallmatrix}}\ar@{->>}[dr]
\\
{\begin{smallmatrix}4&&6\\&5\end{smallmatrix}}
\to\stackrel{\point}{\begin{smallmatrix} 3\\4\end{smallmatrix}}\ar@{^(->}[ur]&&
{\begin{smallmatrix} 6\\5\end{smallmatrix}}[2]&&
{\begin{smallmatrix}0\end{smallmatrix}}\ar@{^(->}[ur]&&
{\begin{smallmatrix}0\end{smallmatrix}}&&
{\begin{smallmatrix}0\end{smallmatrix}}\ar@{^(->}[ur]&&
{\begin{smallmatrix}0\end{smallmatrix}}&&
{\begin{smallmatrix}0\end{smallmatrix}}\ar@{^(->}[ur]&&
{\begin{smallmatrix}0\end{smallmatrix}}
}\]}}

{\tiny{\[\xymatrix@-1.5pc{
&&&&&&&{\begin{smallmatrix}4\end{smallmatrix}}\ar@{->>}[drrrr]
\\
&&&{\begin{smallmatrix}4\end{smallmatrix}}\ar@{^(->}[urrrr]\ar@{->>}[drr]
&&&&&&&&
{\begin{smallmatrix}0\end{smallmatrix}}\ar@{->>}[drr]
\\
&{\begin{smallmatrix}4&&6\\&5\end{smallmatrix}}\ar@{^(->}[urr]\ar@{->>}[dr]
&&&&
{\begin{smallmatrix}6\\5\end{smallmatrix}}[1]\ar@{->>}[dr]
&&&&
{\begin{smallmatrix}0\end{smallmatrix}}\ar@{^(->}[urr]\ar@{->>}[dr]
&&&&
{\begin{smallmatrix}0\end{smallmatrix}}\ar@{->>}[dr]
\\
{\begin{smallmatrix}4&&6\\&5\end{smallmatrix}}
\ar@{^(->}[ur]&&
{\begin{smallmatrix} 0\end{smallmatrix}}&&
{\begin{smallmatrix}0\end{smallmatrix}}\ar@{^(->}[ur]&&
{\begin{smallmatrix}6\\5\end{smallmatrix}}[1]&&
{\begin{smallmatrix}0\end{smallmatrix}}\ar@{^(->}[ur]&&
{\begin{smallmatrix}0\end{smallmatrix}}&&
{\begin{smallmatrix}0\end{smallmatrix}}\ar@{^(->}[ur]&&
{\begin{smallmatrix}0\end{smallmatrix}}
}\]}}

{\tiny{\[\xymatrix@-1.5pc{
&&&&&&&{\begin{smallmatrix}5\end{smallmatrix}}\ar@{->>}[drrrr]
\\
&&&{\begin{smallmatrix}0\end{smallmatrix}}\ar@{^(->}[urrrr]\ar@{->>}[drr]&&&&&&&&
{\begin{smallmatrix}5\end{smallmatrix}}\ar@{->>}[drr]
\\
&{\begin{smallmatrix}0\end{smallmatrix}}\ar@{^(->}[urr]\ar@{->>}[dr]&&&&
{\begin{smallmatrix}0\end{smallmatrix}}\ar@{->>}[dr]&&&&
{\begin{smallmatrix}6\end{smallmatrix}}[-1]\ar@{^(->}[urr]\ar@{->>}[dr]&&&&
{\begin{smallmatrix}6\\ 5\end{smallmatrix}}\ar@{->>}[dr]
\\
{\begin{smallmatrix}0\end{smallmatrix}}
\ar@{^(->}[ur]&&
{\begin{smallmatrix}0\end{smallmatrix}}&&
{\begin{smallmatrix}0\end{smallmatrix}}\ar@{^(->}[ur]&&
{\begin{smallmatrix}0\end{smallmatrix}}&&
{\begin{smallmatrix}6\end{smallmatrix}}[-1]\ar@{^(->}[ur]&&
{\begin{smallmatrix}0\end{smallmatrix}}&&
{\begin{smallmatrix}0\end{smallmatrix}}\ar@{^(->}[ur]&&
{\begin{smallmatrix}6\\5\end{smallmatrix}}
}\]}}
The $t$-leaves in the $t$-trees are $\cT$-static. In particular
the simple left $R$-module $\begin{smallmatrix} 3\end{smallmatrix}$ has a leaf of degree 0 and one of degree 1; the simple left $R$-module $\begin{smallmatrix} 4\end{smallmatrix}$ has a leaf of degree 0 and one of degree 2; the simple left $R$-module $\begin{smallmatrix} 5\end{smallmatrix}$ has a leaf of degree 1 and one of degree 3.
\end{example}

\subsection{The compatible case}
This section is devoted to the applications of Sections \ref{sect:cloth} and \ref{Sec:ttree}
to the case of \emph{compatible} $t$-structures.

The concept of compatible $t$-structures has been first 
introduced by Keller and Vossieck in
\cite{zbMATH04083879} and it has
been recently studied independently by Bondal in \cite{Bondal}
under the name of \emph{consistent pairs of $t$-structures}.
We adopt
the notation of Keller and Vossieck.

\begin{definition}\label{ComTSt}
Let $\mathcal D:=(\mathcal D^{\leq 0}, \mathcal D^{\geq 0})$ and $\mathcal T:=(\mathcal T^{\leq 0}, \mathcal T^{\geq 0})$ be two $t$-structures in a triangulated category 
$\mathcal C$. We denote by $\delta$ and $\tau$ the truncation functors associated with 
$\mathcal D$ and $\mathcal T$, respectively.
The $t$-structure $\mathcal T$ is called:
\begin{enumerate}
\item \emph{left $\cD$-compatible} if $\mathcal T^{\leq 0}$ is stable under the truncation functors $\delta^{\leq n}$, $n\in\mathbb Z$;
\item \emph{right $\cD$-compatible} if $\mathcal T^{\geq 0}$ is stable under the truncation functors $\delta^{\geq n}$, $n\in\mathbb Z$.
\end{enumerate}
\end{definition}

It is not hard to check that if $\cT$ is left $\cD$-compatible, then $\mathcal T^{\leq 0}$ is also  stable under the truncation functors $\delta^{\geq n}$ and therefore 
$H_{\mathcal D}^n(\mathcal T^{\leq 0})[-n]\subseteq \mathcal T^{\leq 0}$  
for each $n\in\mathbb Z$. 
Analogously if $\cT$ is right $\cD$-compatible, then $\mathcal T^{\geq 0}$ is also  stable under the truncation functors $\delta^{\leq n}$ and therefore $H_{\mathcal D}^n(\mathcal T^{\geq 0})[-n]\subseteq \mathcal T^{\geq 0}$ for each  $n\in\mathbb Z$.

\begin{remark}
In \cite{Bondal}, Bondal defined a pair of $t$-structures $(\cD,\cT)$ to be \emph{lower consistent} if
$\delta^{\leq 0}\cTll0\subseteq \cTll0$. 
So the $t$-structure $\cT$ is left $\cD$-compatible if and only if $(\cD[n],\cT)$ is lower consistent for any $n\in\Bbb Z$.
\end{remark}

Let us recall the principal result of Keller and Vossieck concerning a left compatible $t$-structures. 
The statement of the following proposition involves the concept of \emph{bounded $t$-structure} $\cD$ in $\cC$
i.e., for any $X\in\cD$ there exist $m<n$ in $\Bbb Z$ such that $X\in\cD^{[m,n]}$.
In particular any bounded $t$-structure is non degenerate.

\begin{proposition}\label{Keller}
\cite{zbMATH04083879}
Let $\cD$ and $\cT$
be two bounded $t$-structures on the triangulated category 
$\mathcal C$. 
The following are equivalent:
\begin{enumerate}
\item $\cT$ is left $\cD$-compatible;
\item $\cTll0=\left\{X\in \cC \; | \; H_\cD^i(X)\in \cH_\cD \cap \cTll{-i}, \text{~for all~} i\in\bZ \right\}$;
\item we have
\begin{enumerate}
\item $H_\cT^iH_\cD^j(X)=0$, for all $X\in \cH_\cT$ and $i+j>0$,
\item for each morphism $f \colon Y \to Y'$ in $\cH_\cD$ with $Y\in \cTll i$ and $Y' \in \cTll {i-1}$, 
we have $\Ker(f) \in \cTll i$ and $\Coker(f) \in \cTll {i-1}$.
\end{enumerate}
\end{enumerate}
\end{proposition}

\begin{remark}\label{rem:comtilted}
Let $\cD$ be a $t$-structure in a triangulated category $\cC$,  $(\cX,\cY)$ be a torsion pair in $\cH_\cD$ and $\cT_{(\cX,\cY)}$ the $t$-structure associated to $(\cX,\cY)$  (see Proposition~\ref{prop:HRS}). 
Note that $\cT_{(\cX,\cY)}$ is both left and right $\cD$-compatible.
Indeed by Proposition~\ref{prop:HRS} one has:
\[
\begin{matrix}
\cT^{\leq 0}_{(\cX,\cY)}= & \{ C\in \cC\; | \; H_\cD^0(C)\in\cX,\; H_\cD^i(C)=0 \;  \forall i>0 \} \hfill\\
\cT^{\geq 0}_{(\cX,\cY)}= & \{C\in \cC\; | \; H_\cD^{-1}(C)\in\cY,\; H_\cD^i(C)=0 \;  \forall i<-1 \} \\
\end{matrix}
\]
which are stable by both $\delta^{\leq n}$ and $\delta^{\geq n}$ for any $n\in\Bbb Z$. 
On the other side, $\cD$ is both left and right $\cT_{(\cX,\cY)}$-compatible:
we can regard $\cD$ as the $t$-structure 
obtained by tilting $\cT_{(\cX,\cY)}[-1]$ with respect to the torsion pair
$(\cY[0],\cX[-1])$ which proves that both $\cD^{\leq 0}$ and $\cD^{\geq 0}$ are stable by
 $\tau^{\leq n}$ and $\tau^{\geq n}$ for any $n\in\Bbb Z$. 
\end{remark}

Let us recall that by point (ii) of Lemma~\ref{lemma:esempi} if $\cT$ is left $\cD$-compatible, then the pair $(\cD,\cT)$ of $t$-structures is left filterable; then we can associate to this pair its left basic $t$-structures ${}_i\cD$
whose aisle is ${}_i\cD^{\leq 0}=\cDll0\cap\cTll{n-i}$ by Definition~\ref{def:{}_iD}.

\begin{proposition}\label{ComTh}
Let $(\cD,\cT)$ be pair of $t$-structures of type $(n,0)$ in  
$\mathcal C$ and assume $\cT$ is left $\cD$-compatible. 
Then for any $0\leq i\leq j\leq n$ the $t$-structure ${}_{j}\cD$ is left ${ }_i\cD$-compatible.
\end{proposition}
\begin{proof}
Let us recall that since $\cT$ is left $\cD$-compatible and
${}_i\cD^{\leq 0}=\cDll0\cap\cTll{n-i}$, 
then  by Lemma~\ref{lemma:esempi} the truncation functor associated to ${}_i\cD$
is ${}_i\sigma^{\leq h}:=\delta^{\leq h}\tau^{\leq n-i+h}$.
Let $0\leq i\leq j\leq n$ and
$h\in\Bbb Z$; since $\cT$ is left $\cD$-compatible, we have:
\[
{}_i\sigma^{\leq h}({}_j\cDll0)= 
\delta^{\leq h}\tau^{\leq n-i+h}(\cDll0\cap\cTll {n-j})\subseteq \delta^{\leq h}(\cT^{\leq n-j})\subseteq\cTll {n-j}.
\]
Next, for any $h\leq 0$ it is clear that 
\[{}_i\sigma^{\leq h}({}_j\cDll0)= 
\delta^{\leq h}\tau^{\leq n-i+h}(\cDll0\cap\cTll {n-j})
\subseteq\cDll {h}\subseteq \cDll 0;\]
on the other side if $h>0$ we have $0\leq i\leq j<j+h$ and so
$n-i+h>n-j$ which implies:
\[{}_i\sigma^{\leq h}({}_j\cDll0)= 
\delta^{\leq h}\tau^{\leq n-i+h}(\cDll0\cap\cTll {n-j})=
\delta^{\leq h}(\cDll0\cap\cTll {n-j})
\subseteq\delta^{\leq h}(\cDll0)\subseteq \cDll 0.\]
Therefore ${}_i\sigma^{\leq h}({}_j\cDll0)\subseteq
\cD^{\leq 0}\cap\cT^{\leq n-j}=
{}_j\cDll0$
which proves that ${}_{j}\cD$ is left ${ }_i\cD$-compatible.
\end{proof}

\begin{lemma}\label{eqccotor}
Let $(\cD,\cT)$ be pair of $t$-structures of type $(n,0)$ in  
$\mathcal C$ and assume $\cT$ is left $\cD$-compatible. Then 
for any fixed $0\leq k\leq n-1$, the left basic torsion classes of $(\cD,\cT)$ satisfy 
$${}_k\cX:={ }_k\cH\cap\cTll{n-k-1}={}_i\cH\cap\cTll{n-k-1}=\cH_\cD\cap\cTll{n-k-1}
\quad \text{for any}\quad 0\leq i\leq k$$
and hence
\[ {}_{n-1}\cX\subseteq {}_{n-2}\cX\subseteq \cdots \subseteq {}_1\cX\subseteq {}_0\cX\subseteq \cH_\cD=: {}_{-1}\cX.\]
Moreover for any
$f:X_{k-1}\to X_{k}$ with $0\leq k\leq n-1$, 
$X_{k-1}\in {}_{k-1}\cX$ and $X_{k}\in {}_{k}\cX$ we have
$${\Ker}(f)=H_\cD^{-1}\Cone(f)\in {}_{k-1}\cX \qquad
{\Coker}(f)=H_\cD^{0}\Cone(f)\in {}_{k}\cX$$
(kernels and cokernels are computed in $\cH_\cD$ while
$\Cone(f)$ indicates the mapping cone of $f$ in $\cC$).
\end{lemma}
\begin{proof}
First of all let us note that by the definition of the left basic $t$-structures
one obtains ${}_{-1}\cH=\cH_{\cD}$ and so
${}_{-1}\cX=\cH_{\cD}=\cH_{\cD}\cap\cT^{\leq n}$ which justifies the definition of ${}_{-1}\cX$.
Let us prove that  for $0\leq k\leq n-1$ we have
${}_k\cX=\cH_\cD\cap\cT^{\leq n-k-1}$. 
For $k=0$ the statement holds true for any $n\geq 1$ since 
${}_0\cX=\cH_0\cap{}_1\cD^{\leq 0}=\cH_\cD\cap\cT^{\leq n-1}$
as stated in point (5) of Lemma~\ref{lemma:inclusionileft}.
Let us now use induction on the gap $n\in\Bbb N$.
For $n=1$ we have $k=0$ and so the statement is true.
Let us suppose
$1\leq k\leq n-1$ and $n\geq 2$.
By inductive hypothesis 
our statement is true for any 
pair $(\cU,\cV)$ of $t$-structures with gap less than $n$ such that $\cV$ is left $\cU$-compatible.
In particular, for $1\leq i\leq n$, by Remark~\ref{rem:stepsleft} the pair $({ }_{i}\cD,\cT)$ is of type $(n-i,0)$ and by Proposition~\ref{ComTh} the $t$-structure $\cT$ is left ${ }_{i}\cD$-compatible. Observe that the left basic $t$-structures
of the pair $({ }_{i}\cD,\cT)$
coincide with the $t$-structures  $ { }_{\ell}\cD$ for $\ell=i,\dots, n$.
Then by inductive hypothesis
${}_k\cX:={}_k\cH\cap\cTll{n-k-1}={}_i\cH\cap\cTll{n-k-1}$
for any $1\leq i\leq k$. 
\\
It remains to prove that ${}_{1}\cH\cap\cTll{n-k-1}=\cH_\cD\cap\cTll{n-k-1}$.
We have
\[\cH_\cD\cap\cTll{n-k-1}\subseteq \cH_\cD\cap\cTll{n-1}= {}_0\cX\subseteq {}_{1}\cH\] and 
therefore $\cH_{\cD}\cap\cTll{n-k-1}\subseteq {}_1\cH\cap\cTll{n-k-1}$.
Viceversa, if $X\in {}_{1}\cH\cap\cTll{n-k-1}$ then $H^{-1}_\cD(X)\in{}_0\cY$; on the other hand $H^{-1}_\cD(X)[1]=\delta^{\leq -1}(X)$ belongs to $\cTll {n-k-1}$
by the left compatibility hypothesis and so 
$H^{-1}_\cD(X)\in  \cH_\cD\cap\cTll {n-k}\subseteq  \cH_\cD\cap\cTll {n-1}=:{}_0\cX$
which implies $H^{-1}_\cD(X)=0$ and so $X\in \cH_\cD\cap\cTll{n-k-1}$.\\
Now, let $f:X_{k-1}\to X_{k}$ with $0\leq k\leq n-1$,
$X_{k-1}\in {}_{k-1}\cX$ and $X_{k}\in {}_{k}\cX$. 
For what we have proved, $f$ is a morphism in $\cH_\cD$, $X_{k-1}\in\cTll{n-k}$ and
$X_{k}\in\cTll {n-k-1}$; therefore $\Cone(f)\in\cTll {n-k-1}$.
By the definition of the abelian structure of the heart $\cH_\cD$ and by the left compatibility we obtain
${\Ker}(f)=H_\cD^{-1}\Cone(f)=H_\cD^{0}(\Cone(f)[-1])
\in \cH_\cD\cap\cTll {n-k}={}_{k-1}\cX $ while
${\Coker}(f)=H_\cD^{0}\Cone(f)\in \cH_\cD\cap\cTll {n-k-1}= {}_{k+1}\cX$.
\end{proof}
Lemma~\ref{eqccotor} says that the left basic torsion classes ${}_k\cX$, $k=0,..., n-1$, associated to a pair $(\cD,\cT)$ of $t$-structures of type $(n,0)$ with $\cT$ left $\cD$-compatible are ``strongly closed'' with respect to homomorphic images and ``weakly hereditary''.

\begin{theorem}\label{teo:leftcomHRS}
Let $(\cD,\cT)$ be a left filterable pair of $t$-structures of type $(n,0)$. The $t$-structure $\cT$ is left $\cD$-compatible if and only if the left basic torsion classes ${}_i\cX$ lie in $\cH_\cD$; in particular they satisfy 
\[{}_{n-1}\cX\subseteq {}_{n-2}\cX\subseteq \cdots \subseteq {}_1\cX\subseteq {}_0\cX\subseteq 
\cH_\cD.\]
\end{theorem}
\begin{proof}
By Lemma~\ref{eqccotor}, if $\cT$ is left $\cD$-compatible then ${}_i\cX\subseteq\cH_\cD$.
Let us assume the left basic torsion classes ${}_i\cX$ lie in $\cH_\cD$. By Lemma~\ref{heart-intleft} 
\[{}_i\cX\subseteq {}_i\cH\cap \cH_\cD=\bigcap_{\ell=0}^{i-1}{}_\ell\cX\subseteq {}_{i-1}\cX;\]
thus we get
\[{}_i\cX\subseteq {}_{i-1}\cX\cap \cT^{\leq n-i-1}\subseteq {}_i\cH\cap \cT^{\leq n-i-1}={}_i\cX.\]
Then we deduce that
\[\begin{array}{rl}
 {}_i\cX    & ={}_{i-1}\cX\cap \cT^{\leq n-i-1}={}_{i-2}\cX\cap \cT^{\leq n-i}\cap \cT^{\leq n-i-1}=  \\
    &  ={}_{i-2}\cX\cap \cT^{\leq n-i-1}=\dots=\cH_\cD\cap \cT^{\leq n-i-1}.\end{array}\]
As proved by Keller and Vossieck in the bounded case (see Proposition~\ref{Keller}), let us prove by induction on $i=1,\dots ,n$ that we have
\[
{}_i\cD^{\leq 0}=\{C\in\cC\; |\; 
H^k_{\cH_\cD}(C)\in{}_{i-1+k} \cX\; \hbox{if}\; -i+1\leq k\leq 0\; \hbox{and } 
H^k_{\cH_\cD}(C)=0\; \hbox{if}\; k>0\}
\]
which permits to conclude since they are stable by $\delta^{\leq h}$ for any $h\in\Bbb Z$.\\
For $i=1$ the left basic $t$-structure ${}_1\cD$ is obtained by tilting $\cD$ with respect to the torsion pair
$({}_0\cX,{}_0\cY)$ in $\cH_\cD$ and so:
\[{}_1\cD^{\leq 0}=\{C\in\cC\; |\; 
H^0_{\cH_\cD}(C)\in{}_{0}\cX \;\hbox{and } 
H^k_{\cH_\cD}(C)=0\; \hbox{if}\; k>0\}
\]
and hence the statement holds true.
Let assume that 
\[
{}_i\cD^{\leq 0}=\{C\in\cC\; |\; 
H^k_{\cH_\cD}(C)\in{}_{i-1+k}\cX \; \hbox{if}\; -i+1\leq k\leq 0\; \hbox{and } 
H^k_{\cH_\cD}(C)=0\; \hbox{if}\; k>0\}
\]
and let us prove that
\[
{}_{i+1}\cD^{\leq 0}=\{C\in\cC\; |\; 
H^k_{\cH_\cD}(C)\in{}_{i+k}\cX \; \hbox{if}\; -i\leq k\leq 0\; \hbox{and } 
H^k_{\cH_\cD}(C)=0\; \hbox{if}\; k>0\}.
\]
By definition 
$
{}_{i+1}\cD^{\leq 0}=\{C\in\cC\; |\; 
H^0_{{}_i\cH}(C)\in{}_{i}\cX \;  \hbox{and } 
H^k_{{}_i\cH}(C)=0\; \hbox{if}\; k>0\}
$ and by Lemma~\ref{lemma:inclusionileft} one has ${}_{i+1}\cD^{\leq 0}\subseteq {}_{i}\cD^{\leq 0}\subseteq \cD^{\leq 0}$. Therefore, first for any $C\in {}_{i+1}\cD^{\leq 0}$ and $k>0$ we get
$H^k_ {\cH_\cD}(C)=0$; next, since
${}_i\delta^{\leq -1}C\in{}_{i}\cD^{\leq -1}\subseteq \cD^{\leq -1}$
and $H^0_{{}_i\cH}(C)$ belongs to ${}_{i}\cX\subseteq {\cH_\cD}\subseteq \cD^{\geq 0}$, the distinguished triangle ${}_i\delta^{\leq -1}C\to C\to H^0_{{}_i\cH}(C)\stackrel{+1}\to $
coincides with the approximating triangle with respect to the $t$-structure $\cD$:
\[
\xymatrix{
{}_i\delta^{\leq -1}C\ar[r]\ar[d]^{\cong}& C\ar[r]\ar[d]^{\rm{id}_C} & H^0_{{}_i\cH}(C)\ar[r]^(0.7){+1}\ar[d]^{\cong} &\\
\delta^{\leq -1}C\ar[r]    & C\ar[r]   & H^0_{\cH_\cD}(C)\ar[r]^(0.7){+1}& \\
}\]
Therefore, since by inductive hypothesis we have
\[
{}_i\cD^{\leq -1}=\{C\in\cC\; |\; 
H^k_{\cH_\cD}(C)\in{}_{i+k}\cX \; \hbox{if}\; -i\leq k\leq -1\; \hbox{and } 
H^k_{\cH_\cD}(C)=0\; \hbox{if}\; k>-1\},
\]
one gets
\[
H^0_ {\cH_\cD}(C)\cong H^0_{{}_i\cH}(C)\in {}_i\cX; \;
H^k_ {\cH_\cD}(C)\cong H^k_ {\cH_\cD}({}_i\delta^{\leq -1}C)\in{}_{i+k}\cX
\hbox{ for any } -i\leq k\leq -1,
\]
and hence ${}_{i+1}\cD^{\leq 0}\subseteq \{C\in\cC\; |\; 
H^k_{\cH_\cD}(C)\in{}_{i+k}\cX \; \hbox{if}\; -i\leq k\leq 0\; \hbox{and } 
H^k_{\cH_\cD}(C)=0\; \hbox{if}\; k>0\}$.\\
Let us prove the other inclusion: consider $C\in\cC$ belonging to the right side.
By the previous description of ${}_i\cD^{\leq -1}$, due to the inductive hypothesis, we have
$\delta^{\leq -1}C\in{}_i\cD^{\leq -1}\subseteq {}_{i+1}\cD^{\leq 0}$,
while 
$\delta^{\geq 0}C=H^0_{\cH_\cD}(C)\in {}_i\cX\subseteq {}_{i+1}\cH\subseteq {}_{i+1}\cD^{\leq 0}$; then $C$ is an extension of objects in ${}_{i+1}\cD^{\leq 0}$ and hence it belongs to ${}_{i+1}\cD^{\leq 0}$.
\end{proof}

The left compatible case has been studied by Vit\'oria in
 \cite{Vit2014}. 
In particular the author considers the bounded derived category  $D^b(\cA)$ of
an AB4 abelian category $\cA$ endowed with its natural $t$-structure $\cD$.
Then he proves  in \cite[Theorem 3.13]{Vit2014} that, under a technical hypothesis,  the data of $n$  hereditary torsion classes of $\cA$ such that
\[{}_{n-1}\cX\subseteq {}_{n-2}\cX\subseteq \cdots \subseteq {}_1\cX\subseteq {}_0\cX\subseteq 
\cH_\cD=:{}_{-1}\cX\] permits to construct 
 (via an iterated HRS procedure of length $n$) a new left $\cD$-compatible
 $t$-structure $\cT$.
 So Theorem~\ref{teo:leftcomHRS} can be seen as a partial generalization of \cite{Vit2014}.

\begin{remark}
We have seen in Proposition~\ref{polisch} that by Polishchuk result a pair of $t$-structures
$(\cD,\cT)$ in a triangulated category $\cC$ verifies
$\cD^{\leq -1}\subseteq \cT^{\leq 0}\subseteq \cD^{\leq 0}$ (or equivalently
$\cD^{\geq 0}\subseteq \cT^{\geq 0}\subseteq \cD^{\geq -1}$)
if and only if $\cT$ is a $t$-structure obtained by tilting $\cD$ with respect to a torsion pair in 
$\cH_\cD$. 
Moreover we proved in Remark~\ref{rem:polisfilterable} that such a pair of $t$-structures is always both
left and right filterable and we proved in Remark~\ref{rem:comtilted} that one is both
left and right compatible with respect to the other. These remarkable properties of a pair of $t$-structures obtained by 
HRS procedure of length $1$ do not hold true for $t$-structures with gap $n\geq 2$.
Actually we can deduce from the previous theorem the following corollary.
\end{remark}
 
\begin{corollary}\label{Cor:tiltnoleftcom}
Let $(\cD,\cT)$ be a $n$-tilting pair of $t$-structures.
Then $\cT$ is left $\cD$-compatible if and only if $n=0$ or $n=1$.
\end{corollary}

\begin{proof}
It is clear that if $n=0$ then $\cD=\cT$ is left $\cD$-compatible and we have seen in 
Remark~\ref{rem:comtilted} that for any pair $(\cD,\cT)$ of type $(1,0)$, the $t$-structure $\cT$ is left $\cD$-compatible.
\\
Viceversa let $(\cD,\cT)$ be a $n$-tilting pair of $t$-structures such that $\cT$ is left $\cD$-compatible; then $(\cD,\cT)$ is left filterable.
By Theorem~\ref{teo:genPolleft}, the $t$-structure $\cT$ is obtained by an iterated  HRS procedure of length $n$ 
(via its left basic $t$-structures)
and by
Lemma~\ref{heart-intleft} we have
${}_0\cH\cap{}_n\cH=\bigcap_{i=0}^{n-1}{}_i\cX={}_{n-1}\cX$. Since $(\cD,\cT)$ is a $n$-tilting torsion pair we have that ${}_{n-1}\cX$ cogenerates $\cH_\cD$. 
So given an element $M\in\cH_\cD$ there exist a short exact sequence
$0\to M\stackrel{i}{\to}X\stackrel{p}{\to}\coker_{\cH_\cD}(i)\to 0$ in $\cH_\cD$ with 
$X\in {}_{n-1}\cX\subseteq{}_{n-2}\cX$
and a monomorphism
$j:\coker_{\cH_\cD}(i)\hookrightarrow Y$ with 
$Y\in {}_{n-1}\cX$.
Then we obtain: 
\[f:=j\circ p:X\longrightarrow Y\qquad \hbox{with}\quad X\in {}_{n-2}\cX,\quad Y\in {}_{n-1}\cX\]
 and so by Theorem~\ref{teo:leftcomHRS} we have
$\Ker_{\cH_\cD}(f)=M\in{}_{n-2}\cX$ for any $M\in\cH_\cD$ which proves that ${}_{n-2}\cX=\cH_\cD$ and hence
${}_{n-2}\cX=\cdots ={}_0\cX={}_{-1}\cX=\cH_\cD$. This implies that $n=0$ or $n=1$
otherwise for $n\geq 2$ we would have $ \cH_\cD\not={}_0\cX\subset \cH_\cD$ which can not occur.
\end{proof}


\end{document}